\documentclass[12pt]{article}
\usepackage{amsmath}
\usepackage{amssymb}
\usepackage{amsthm}
\usepackage[english]{babel}
\usepackage[mathscr]{eucal}
\usepackage[usenames]{color}
\usepackage{array}
\usepackage{authblk}
\usepackage[
bookmarks=true,
backref=true,
colorlinks=true,
linkcolor=red,
urlcolor=green]{hyperref}

\newcommand{\dil}{t\gen t+x\gen x+y\gen y}
\newcommand{\vf}{\mathbf{v}}
\newcommand{\gen}[1]{\partial_{#1}}
\newcommand{\curl}[1]{ \left\{#1\right\} }
\newcommand{\lie}{\mathfrak g}
\DeclareMathOperator{\Sl}{sl}
\DeclareMathOperator{\e}{e}
\DeclareMathOperator{\E}{E}
\DeclareMathOperator{\Ort}{SO}

\DeclareMathOperator{\PP}{P}
\DeclareMathOperator{\ort}{so}
\DeclareMathOperator{\conf}{Conf}
\DeclareMathOperator{\GL}{GL}
\DeclareMathOperator{\SL}{SL}

\DeclareMathOperator{\Heis}{h}
\DeclareMathOperator{\rank}{rank}
\DeclareMathOperator{\Span}{span}
\DeclareMathOperator{\Diff}{Diff}
\DeclareMathOperator{\Ad}{Ad}
\DeclareMathOperator{\Nor}{Nor}
\DeclareMathOperator{\nor}{nor}

\DeclareMathOperator{\arctanh}{arctanh}
\DeclareMathOperator{\proj}{proj}
\newcommand{\pr}[1]{{\rm pr}^{(#1)}}

\theoremstyle{plain}
\newtheorem{theorem}{Theorem}[section]

\newtheorem{proposition}[theorem]{Proposition}

\theoremstyle{definition}
\newtheorem{definition}[theorem]{Definition}	
\newtheorem{remark}[theorem]{Remark}

\newtheorem{example}[theorem]{Example}

\begin{document}

\pagenumbering{arabic}
\clearpage
\thispagestyle{empty}

\title{Lie symmetry group methods for differential equations}

\author[1]{F.~G\"ung\"or\thanks{gungorf@itu.edu.tr}}
\affil[1]{Department of Mathematics, Faculty of Science and
Letters, Istanbul Technical University, 34469 Istanbul, Turkey}

\date{\today}
\maketitle

\begin{abstract}
Fundamentals on Lie group methods and applications to differential equations are surveyed. Numerous examples are included to elucidate their extensive applicability for analytically solving both ordinary and partial differential equations.
\end{abstract}

\newpage

\tableofcontents
\newpage

\section*{Introduction}

This paper reviews a treatment of differential equations using methods from Lie group theory.
Symmetry group methods are amongst the most powerful universal tools for the study of differential equations. There has been rapid progress on these methods over the last few decades. Methods and algorithms for classifying  subalgebras of Lie algebras, new results on the structure and classification of abstract finite and infinite dimensional Lie algebras \cite{SnobleWinternitz2014} and methods for solving group classification problems for  differential equations greatly facilitated to systematically obtain exact analytic solutions by quadratures to ordinary differential equations and  group-invariant solutions to partial differential equations and to identify equivalent equations. The application of Lie groups to differential equations has a long history. In the second half of the nineteenth century, Norwegian mathematician  Sophus Lie (1842-1899) introduced continuous groups of transformations \cite{Lie1881, LieEngel1891} to give a unified and systematic theory for the study of properties of solutions of differential equations just like Evariste Galois's (1811-1831) dream to solve  algebraic equations by radicals, which led to the theory of Galois.  In his memoir of 1831, he considered the group of admissible permutations.
The theory of Lie groups and algebras originated precisely in  the context of differential equations. Over the years, these transformations evolved into the modern theory of abstract Lie groups  and algebras.

As far as differential equations are concerned, the main observation was that much of the known solutions methods  were actually specific cases of a general solution (general or particular solution) method based on the invariance of a system of differential equations under a continuous  group of transformations (called symmetry group of the system).

Symmetry group of a partial differential equation (PDE) can be used to reduce the number of independent variables, to transform known simple solutions to new solutions.  For an ordinary differential equation (ODE), even reduction in order can be made. Under some special structure of the symmetry group, reduction can even go all the way down to an algebraic equation, from which general solution can be obtained.

The computation of the symmetry group of a system of differential equations can be computationally complicated, but nevertheless completely algorithmic. Computer algebra systems can automate most of the steps of the Lie symmetry  algorithm.

It should be emphasized that applications of Lie group methods using classical group of point transformations are not only restricted  to differential equations. They, and when applicable, their generalizations to higher order symmetries can also be carried over to conservation laws, Hamiltonian systems, difference and  differential-difference equations,  integro-differential equations,  delay differential equations, fractional differential equations.

Basic ideas, definitions, theorems and results needed to be able to apply Lie group methods for solving differential equations are presented. Many examples with mathematical and physical  applications
are considered to illustrate applications  to integration of ODEs by the method of reduction of order, construction of group-invariant solutions to PDEs, identification of equivalent equations based on the existence of isomorphic symmetry groups, generating new solutions from known ones, group-classification problem and construction of invariant differential equations.

There exists a vast literature on the methods and applications of Lie groups and algebras to differential equations and other fields of mathematics and physics. For more detailed treatments of the subject, the readers are referred to the books
\cite{Olver1993, Ovsyannikov1982, Ibragimov1985, BlumanCheviakovAnco2010, Hydon2000, Olver1995,  Stephani1989, OvsyannikovIbragimov2013, Hill1992} and the handbooks \cite{Ibragimov1994a, Ibragimov1994b, Ibragimov1995}. A basic knowledge of Lie groups and algebras will be assumed. The reader is directed to, for example, a recent book  \cite{SnobleWinternitz2014} for a good account of basic definitions,  theorems and applications of Lie algebras.

\section{Vector fields and integral curves}
We begin with a brief review of some essential objects that will be employed throughout.
Let $M$ be a differentiable manifold of dimension $n$. A curve $\gamma$ at a point $x$ of $M$ is a differentiable map $\gamma: I\to M$, where $I$ is a subinterval of $\mathbb{R}$, such that $\gamma(0)=x$, $0\in I$.

A vector field $\mathbf{v}$ of the manifold $M$ is a $C^{\infty}$-section of $TM$, in other words a $C^{\infty}$ mapping from $M$ to $TM$ that assigns to each point $x$ of M a  vector in $T_xM$. In the local coordinate system $x=(x_1,\ldots,x_n)\in M$, $\mathbf{v}$ can be expressed as
\begin{equation}\label{vf}
  \mathbf{v}\vert_{x}=\sum_{i=1}^n \xi_i(x)\gen{x_i},
\end{equation}
where $\xi_i(x)\in C^{\infty}(M)$, $i=1,\ldots, n$.

An integral curve of the vector field at the point $x$ is the curve $\gamma$  at $x$ whose tangent vector $\dot{\gamma}(t)$ coincides with $\mathbf{v}$ at the point $x=\gamma(t)$ such that $\dot{\gamma}(t)=\mathbf{v}\vert_{\gamma(t)}$ for each $t\in I$. In the local representation of the curve $\gamma$ it amounts to saying that the curve satisfies an autonomous  system of first order ordinary differential equations
\begin{equation}\label{comp-sys}
\frac{d\gamma_i(t)}{dt}=\xi_i(\gamma(t)),  \quad i=1,\ldots, n.
\end{equation}
The existence and uniqueness theorem for systems of ODES ensures that there is a unique solution to the system with the initial data $\gamma(0)=x_0$ (the Cauchy problem). This gives rise to the existence of a unique maximal integral curve $\gamma(t)$ passing through the point $x_0=\gamma(0)\in M$. We call such a maximal integral curve the flow of $\mathbf{v}$ at $x=\gamma(t)$ and denote $\Phi(t,x)$ with the basic properties
\begin{equation}\label{propert}
  \Phi(0,x)=x,  \quad \Phi(s,\Phi(t,x))=\Phi(t+s,x),  \quad \frac{d}{dt}\Phi(t,x)=\mathbf{v}\big\vert_{\Phi(t,x)}, \quad x\in M,
\end{equation}
for all sufficiently small $t,s\in \mathbb{R}$. A more suggestive notation for the flow is $\Phi(t,x)=\exp{(t \mathbf{v})}x$. The reason is simply that it satisfies the ordinary exponential rules. The second property implies that $\Phi(-t,x)=\Phi^{-1}(t,x)$ or $\exp{(t \mathbf{v})}^{-1}x=\exp{(-t \mathbf{v})}x$. One can infinitesimally express the flow
\begin{equation}\label{flow}
  \exp{(t \mathbf{v})}x=x+t \mathbf{v}\vert_x+\mathcal{O}(t^2).
\end{equation}
The flow $\exp{(t \mathbf{v})}x$ generated by the vector field $\mathbf{v}$ is sometimes called a one-parameter group of transformations as it arises as the action of the Lie group $\mathbb{R}$ on the manifold $M$.

Conversely, given a flow with the first two properties of \eqref{propert}, we can reconstruct its generating vector field $\mathbf{v}$ by differentiating the flow:
$$\mathbf{v}\vert_x=\frac{d}{dt}\exp{(t \mathbf{v})}x\big\vert_{t=0},  \quad x\in M.$$ The inverse process of constructing the flow is usually called exponentiation (or integration) of $\mathbf{v}$.

Rectification of a vector field $\mathbf{v}$ in a neighborhood of a regular point (a point $x$ at which $\mathbf{v}\vert_x$ does not vanish) is always possible.
\begin{theorem}
If $x_0$ is a regular point of $\mathbf{v}$, then there exist local rectifying (or straightening out)  coordinates $y=(y_1,\ldots,y_n)$ near $x_0$ such that $y=\gen {y_1}$ generates the translational flow $\exp{(t \mathbf{v})}y=(y_1+t,y_2,\ldots,y_n)$.
\end{theorem}

Let $\mathbf{v}$ be a vector field in the local coordinates
\begin{equation}\label{vf-2}
 \mathbf{v}=\sum_{i=1}^n \xi_i(x)\gen {x_i}.
\end{equation}
Application of a vector field $\mathbf{v}$ to a smooth function $f:M \to \mathbb{R}$, also referred to the Lie derivative of $f$ with respect to the vector field,   determines the infinitesimal rate of change  in $f$ as the parameter $t$ of the induced flow  $\exp{(t \mathbf{v})x}$ of $\mathbf{v}$ varies:
\begin{equation}\label{action-on-func}
 \mathbf{v}(f(x))=\sum_{i=1}^n \xi_i(x)\frac{\partial f}{\partial x_i}=\frac{d}{dt}\Big\vert_{t=0}f(\exp{(t \mathbf{v})x}).
\end{equation}
It is observed that vector fields act on functions as derivations in the sense that they are linear and satisfy the Leibnitz rule
\begin{equation}\label{derivation}
  \mathbf{v}(f+g)=\mathbf{v}(f)+\mathbf{v}(g),  \quad \mathbf{v}(fg)=f \mathbf{v}(g)+g \mathbf{v}(f).
\end{equation}
The action of the flow generated by  $\mathbf{v}$ on a function can be obtained from  the following Lie series (assuming its convergence)
\begin{equation}\label{Lie-series}
  f(\exp{(t \mathbf{v})x})=\sum_{j=0}^{\infty}\frac{\varepsilon^{j}}{j!}\mathbf{v}^{j}(f(x)).
\end{equation}
The Lie series is derived by expanding $f(\exp{(t \mathbf{v})x})$ into the Taylor's series in $\varepsilon$. For the special choice $f(x)=x$, we recover the flow \eqref{flow}.

If $y=\psi(x)$ is a  change of coordinates, then by the chain rule,   the vector field \eqref{vf-2} in the new coordinates $y$ is expressed as
\begin{equation}\label{vf-new-coord}
  \mathbf{v}=\sum_{j=1}^n\sum_{i=1}^n \xi_i(\psi^{-1}(y))\frac{\partial \psi_j}{\partial x_i}(\psi^{-1}(y))\frac{\partial}{\partial y_j}=\sum_{j=1}^n \mathbf{v}(\psi_j(x))\big\vert_{x=\psi^{-1}(y)}\frac{\partial}{\partial y_j}.
\end{equation}

The commutator or Lie bracket of two vector fields $\mathbf{v}$ and $\mathbf{w}$ is defined as the unique vector field satisfying
\begin{equation}\label{Lie-bracket}
  [\mathbf{v}, \mathbf{w}]f=\mathbf{v}(\mathbf{w}(f))-\mathbf{w}(\mathbf{v}(f))
\end{equation}
for all smooth functions $f:M\to \mathbb{R}$. If, in local coordinates,
$$\mathbf{v}=\sum_{i=1}^n \xi_i(x)\gen {x_i},  \quad \mathbf{w}=\sum_{i=1}^n \eta_i(x)\gen {x_i},$$ then
$$[\mathbf{v}, \mathbf{w}]=\sum_{i=1}^n\curl{\mathbf{v}(\eta_i)-\mathbf{w}(\xi_i)}\gen {x_i}.$$ The coefficients $\mathbf{v}(\eta_i)$ and $\mathbf{w}(\xi_i)$ are the actions of the vector fields $\mathbf{v}$ and $\mathbf{w}$ as derivations on the functions $\eta_i$ and $\xi_i$, respectively as defined in \eqref{action-on-func}.

\section{Differential equations and their symmetry group}
We consider a system of $n$-th order differential equations $\mathcal{E}$
\begin{equation}\label{sys}
  \mathcal{E}: \quad \mathsf{E}_{\nu}(x,u^{(n)})=\mathsf{E}_{\nu}(x,u,u^{(1)},\ldots, u^{(n)})=0, \quad \nu=1,2,\ldots, N,
\end{equation}
where $x=(x_1,\ldots,x_{p})\in \mathbb{R}^{p}$, $u=(u_1,\ldots,u_{q})\in \mathbb{R}^{q}$ ($p, q\in \mathbb{Z}^{+}$) are the independent and dependent variables, which form local coordinates on the space of independent and dependent variables  $E=X\times U \simeq \mathbb{R}^p\times \mathbb{R}^q$. The derivatives of $u$ are denoted by $u_{\alpha,J}=\partial^{J}u_{\alpha}/\partial x_{J}$, where $J=(j_1,\ldots,j_k)$, $1\leq j_{\nu}\leq p$, $k=j_1+\ldots+j_k$, is a symmetric multi-index of order $k=\#J$.  $u^{(k)}$ denotes all partial derivatives of order $\leq k$ of the components $u_{\alpha}$ of $u$, which provide coordinates on the jet space $J^n(x,u^{(n)})=J^n E$ . If there is a single independent and  dependent variable, namely $p=1$ and $q=1$, then the system becomes a scalar  ordinary differential equation. In that case, we simply write
$$\mathsf{E}(x,u,u_1,u_2,\ldots, u_n)=0,$$ where $u_1=u_x$, $u_2=u_{xx}$, $\ldots$, $u_n=u^{(n)}$.

The system $\mathsf{E}_{\nu}=0$ defined by a collection of smooth functions $\mathsf{E}=(\mathsf{E}_1, \ldots, \mathsf{E}_N)$ can be identified with a variety $\mathcal{S}_{\mathsf{E}}=\curl{(x,u^{(n)}):\mathsf{E}=0}$ contained in the $n$-th order jet space $J^n$ with local coordinates $(x,u^{(n)})$.

The functions $\mathsf{E}_{\nu}(x,u^{(n)})$ will be assumed to be regular, meaning that the Jacobian matrix of $\mathsf{E}_{\nu}$ with respect to the jet coordinates $(x,u^{(n)})$ has  maximal rank
$$\rank\left(\frac{\partial \mathsf{E}_{\nu}}{\partial x_i},\frac{\partial \mathsf{E}_{\nu}}{\partial u_{\alpha,J}}\right)=N,$$
at each $(x,u^{(n)})$ satisfying the system.

A classical  symmetry group of \eqref{sys} is a local group $G$ of point transformations $\Phi:E\to E$, a locally defined invertible map on the space  of independent and dependent variables, mapping solutions of the system to solutions
$$\Phi: \quad (\tilde{x},\tilde{u})=g.(x,u)=(\Phi_1(x,u),\Phi_2(x,u).$$
Such transformations act on solutions $u=f(x)$  by mapping pointwise their graphs. More precisely, if $\Gamma_{f}=\curl{(x,f(x)}$ is the graph of $f(x)$, then the mapped graph will have the graph $\Gamma_{\tilde{f}}=\{(\tilde{x},\tilde{f}(\tilde{x})\}=g.\Gamma_{f}\equiv \curl{g.(x,f(x))}$.

Contact or generalized transformations where $\Phi$ depends on higher order derivatives will not be treated here.

\begin{definition}
A local Lie group  of point transformations $G$ is called a symmetry group of
the system of partial differential equations \eqref{sys} if $\tilde{f}=g.f$ is a solution whenever $f$ is.
\end{definition}

To find the symmetry group, the prolonged transformation $\pr{n}\Phi:J^n\to J^n$ is required to preserve  the differential structure of the equation manifold $\mathcal{S}_{\mathsf{E}}$. In order to find the symmetry group Lie's infinitesimal approach will be used. We need to use the prolongation tool for the group transformation and the vector field generating it. Let $\Phi_{\varepsilon}=\exp(\varepsilon \vf)$ be a one-parameter subgroup of the connected group $G$ and let
\begin{equation}\label{sym-vf}
  \vf=\sum_{i=1}^p \xi_i(x,u)\gen{x_i}+\sum_{\alpha=1}^q \varphi_{\alpha}(x,u)\gen{u_\alpha}
\end{equation}
be the infinitesimal generator of $\Phi_{\varepsilon}$. The infinitesimal generator of the  prolonged one-parameter subgroup $\pr{n}\Phi_{\varepsilon}$ is defined to be the prolongation of the vector field $\vf$.

\begin{definition}\label{def-prolong}
The $n$-th prolongation $\pr{n}\vf$ of $\vf$ is a vector field on the $n$-th jet space $J^n$ defined by
\begin{equation}\label{nth-pro}
  \pr{n}\vf\big\vert_{(x,u^{(n)})}=\frac{d}{d\varepsilon}\Big\vert_{\varepsilon=0}\pr{n}\Phi_{\varepsilon}(x,u^{(n)})
\end{equation}
for every $(x,u^{(n)})\in J^n$.
\end{definition}
If we integrate $\pr{n}\vf$ we find the prolongation of the group action $\pr{n}\Phi_{\varepsilon}$ on the space $J^n$.  The prolonged vector field $\pr{n}\vf$  has the form
\begin{equation}\label{pro-vf}
  \pr{n}\vf=\sum_{i=1}^p \xi_i\gen{x_i}+\sum_{\alpha=1}^q \sum_{\#J\leq n}\varphi_{\alpha}^J\gen{u_{\alpha,J}},
\end{equation}
where the coefficients $\varphi_{\alpha}^J$ are given by the formula
\begin{equation}\label{vf-coeff}
  \varphi_{\alpha}^J=D_J\left(\varphi_{\alpha}-\sum_{i=1}^p \xi_i u_{\alpha,i}\right)+\sum_{i=1}^p\xi_i u_{\alpha,J,i},
\end{equation}
where $u_{\alpha,i}=\partial u_{\alpha}/\partial x_i$, $u_{\alpha,J,i}=\partial u_{\alpha,J}/\partial x_i$ and
$D_J=D_{j_1}\ldots D_{j_k}$, $1\leq j_\nu \leq p$ is the $J$-th total derivative operator  \cite{Olver1993}.
Here $D_i$ is the total differentiation operator defined by
$$D_i=\frac{\partial}{\partial x_i}+\sum_{\alpha=1}^q\sum_{J}u_{\alpha,J,i}\frac{\partial}{\partial u_{\alpha,J}}.$$
$D_i$ involves infinite summation, but its application to a particular differential function will only require finitely many terms of order $0\leq \# J\leq n$, where $n$ is the highest order derivative in the differential function on which $D_i$ acts.

There is a useful recursive formula for the coefficients of the prolonged vector field in \eqref{pro-vf}
\begin{equation}\label{recurs}
  \varphi_{\alpha}^{J,i}=D_i \varphi_{\alpha}^J-\sum_{k=1}^p (D_i \xi_k) u_{\alpha,J,k}.
\end{equation}
If $n$-th prolongation is known, the $(n+1)$-th prolongation can be calculated by the formula \eqref{recurs}. In particular, the coefficients of the first order derivatives $u_{x_j}$ in \eqref{pro-vf} are then given by
\begin{equation}\label{first-coeff-pro}
  \varphi_{\alpha}^{j}=D_j \varphi_{\alpha}-\sum_{k=1}^p (D_j \xi_k) u_{\alpha,k}.
\end{equation}
In the special case $p=q=1$, the recursion formula \eqref{recurs} simplifies to
\begin{equation}\label{recurs-p-q-1}
  \varphi^j=D_x \varphi^{j-1}-(D_x \xi)u^{(j)}, \quad j=1,2,\ldots.
\end{equation}
The coefficients of the second prolongation of the vector field $\mathbf{v}=\xi(x,u)\gen x+\varphi(x,u)\gen u$ corresponds to $j=1,2$ with the convention $\varphi^{0}=\varphi$
$$\varphi^{x}=D_x \varphi-(D_x\xi) u_{x},  \quad \varphi^{xx}=D_x \varphi^{x}-(D_x\xi) u_{xx}, \quad Q=\varphi-\xi u_x.$$

The prolongations of vector fields satisfy the linearity
\begin{equation}\label{pr-prop-lin}
  \pr{n}(a \vf+b \mathbf{w})=a\;  \pr{n}\vf+b \; \pr{n}\mathbf{w},
\end{equation}
for constants $a$, $b$ and the Lie algebra property
\begin{equation}\label{pr-prop-Lie}
  \pr{n}[\vf,\mathbf{w}]=[\pr{n}\vf,\pr{n}\mathbf{w}].
\end{equation}
Hence, the prolongation process defines a Lie algebra homomorphism from the space of vector fields on $J^0$ to the space of vector fields on $J^n$. If the vector fields $\vf$ form a Lie algebra, then their prolongations realize an isomorphic Lie algebra of vector fields on $J^n$.

\begin{example}\label{pr-tr-ex}
  Let us consider the smooth projective vector field
\begin{equation}\label{proj-vf}
  \vf=x^2 \gen x+xu\gen u
\end{equation}
generating the one-parameter local projective group
\begin{equation}\label{proj-group}
 \Phi_{\varepsilon}(x,u):\quad \tilde{x}(x,u;\varepsilon)=\frac{x}{1-\varepsilon x},  \quad \tilde{u}(x,u;\varepsilon)=\frac{u}{1-\varepsilon x},
\end{equation}
whenever $1-\varepsilon x\ne 0$.  The vector field $\mathbf{v}$ can be recovered by differentiating $\Phi_{\varepsilon}(x,u)$ at $\varepsilon=0$.

The first and second prolonged group transformations are derived by the usual chain rule for the ordinary derivatives:
$$\pr{1}\Phi_{\varepsilon}(x,u,u_1)=(\tilde{x},\tilde{u},\tilde{u}_1), \quad \tilde{u}_1=u_1+\varepsilon(u-xu_1),$$
and
$$\pr{2}\Phi_{\varepsilon}(x,u,u_1,u_2)=\left(\tilde{x},\tilde{u},u_1+\varepsilon(u-xu_1),(1-\varepsilon x)^3u_2\right).$$
Applying definition \ref{def-prolong} we find the prolonged vector fields
\begin{equation}\label{prvfs}
 \pr{1}\vf=\vf+(u-xu_1)\gen {u_1},  \quad \pr{2}\vf=\vf+(u-xu_1)\gen {u_1}-3xu_2\gen {u_2}.
\end{equation}
\end{example}


The following theorem determines the Lie algebra of the symmetry group $G$ and known as the infinitesimal criterion of invariance of \eqref{sys}.
\begin{theorem}\label{Thm-inf}
A connected local group of  transformations $G$ is a symmetry group of the system $\mathcal{E}$ of \eqref{sys} if and only if the $n$-th prolongation $\pr{n}\vf$ annihilates the system on solutions, namely
\begin{equation}\label{inv-criter}
  \pr{n}\vf(\mathsf{E}_{\nu})=0,\quad \nu=1,2,\ldots, N,
\end{equation}
whenever $u=f(x)$ is a solution to the system \eqref{sys} for every infinitesimal generator $\vf$ of $G$.
\end{theorem}
Eqs. \eqref{inv-criter} are known as the determining equations of the symmetry group for the system.  They form a large over-determining linear system of partial differential equations for the coefficients $\xi_i$  and $\varphi_{\alpha}$ of $\mathbf{v}$.
This criterion has been applied to many differential equations arising in different branches of mathematics, physics and engineering to compute symmetry groups. The computation of symmetry group using the infinitesimal approach (Theorem \eqref{Thm-inf}) have been implemented in several computer algebra systems, such as MATHEMATICA, MAPLE, REDUCE, MACSYMA (or freely available MAXIMA) \cite{ChampagneHeremanWinternitz1991, Hereman1996}. There are packages dedicated to the symmetry group calculations which make considerably easy the routine steps of finding the determining system and   partial integration of them. Some packages are capable of triangularize the overdetermined system using differential Gr\"obner basis method. Packages equipped with automatic integrators can usually fail to provide the general solution of the determining system depending on the complexity of the system. We refer to \cite{SteinbergMeloMarinhoJunior2014} for the symbolic calculation of  symmetries of differential equations.

There is an alternative formulation of the prolongation formula, which is  useful in prolongation computations. This requires the formalism of the evolutionary vector fields. Given the vector field $\mathbf{v}$ as in \eqref{vf}, we define the $q$-tuple $Q(x,u^{(1)})=(Q_1, \ldots, Q_q)$ defined by
$$Q_{\alpha}(x,u^{(1)})=\varphi_{\alpha}(x,u)-\sum_{i=1}^p\xi_i(x,u)u_{\alpha,i}\quad \alpha=1,2,\ldots,q.$$ The functions $Q_{\alpha}$ are called the characteristics of the vector field $\mathbf{v}$. Then, we have
\begin{equation}\label{coeff-Q-pro}
  \varphi_{\alpha}^J=D_JQ_{\alpha}+\sum_{i=1}^p \xi_i u_{\alpha,J,i}, \quad D_J=D_{j_1}\ldots D_{j_k}, \quad 1\leq j_\nu \leq p
\end{equation}
and the $n$-th prolongation of $\mathbf{v}$ can be expressed as
$$\pr{n}\mathbf{v}=\pr{n}\mathbf{v}_Q+\sum_{i=1}^p\xi_i D_i,$$
where
$$\mathbf{v}_Q=\sum_{\alpha=1}^q Q_{\alpha}(x,u^{(1)})\gen{u_{\alpha}},  \quad \pr{n}\mathbf{v}_Q=\sum_{\alpha=1}^q\sum_{J}D_JQ_{\alpha}\gen{u_{\alpha,J}}.$$

Obviously, $\mathbf{v}_Q$ and their prolongations do not act on the independent variables $x_i$. In terms of characteristics $Q_{\alpha}$, the infinitesimal transformations  can be written as
$$\tilde{x}_i=x_i, \quad \tilde{u}_{\alpha}=u_{\alpha}+\varepsilon Q_{\alpha}+\mathcal{O}(\varepsilon^2).$$
Since $D_i\mathsf{E}_{\nu}=0$ on solutions, we can replace the infinitesimal symmetry condition \eqref{inv-criter} by the simpler formula
\begin{equation}\label{inv-criter-2}
  \pr{n}\mathbf{v}_Q\Big\vert_{\mathsf{E}_{\nu}=0}=0.
\end{equation}

\begin{example}\label{Lap}
  We show that the Laplace equation $\Delta u(x,y)=u_{xx}+u_{yy}=0$ in the plane is invariant under the symmetry group generated by the vector field
\begin{equation}\label{inf-vf}
 \mathbf{v}=\xi(x,y)\gen x+\eta(x,y)\gen y,
\end{equation}
where $\xi$ and $\eta$ satisfy the Cauchy--Riemann equations $\xi_x=\eta_y$ , $\xi_y=-\eta_x$, in other words $\xi$, $\eta$ are harmonic functions and therefore $\mathbf{v}$ generates an infinite-dimensional symmetry group of the two-dimensional Laplace equation. We recall that the general solution can be expressed in terms of two arbitrary analytic functions.

The second prolongation of $\mathbf{v}$ is
$$\pr{2}\mathbf{v}=\mathbf{v}+\varphi^{x}\gen{u_x}+\varphi^{y}\gen{u_y}+\varphi^{xx}\gen{u_{xx}}
+\varphi^{xy}\gen{u_{xy}}+\varphi^{yy}\gen{u_{yy}}.$$
The coefficients $\varphi^{xx}$, $\varphi^{yy}$ are calculated from the general prolongation formulas \eqref{pro-vf}-\eqref{vf-coeff}
$$\varphi^{xx}=-(2\eta_x u_{xy}+2\xi_x u_{xx}+\eta_{xx}u_y+\xi_{xx}u_x), $$ $$\varphi^{yy}=-(2\xi_y u_{xy}+2\eta_y u_{yy}+\eta_{yy}u_y+\xi_{yy}u_x).$$
So from the Cauchy--Riemann equations we find that the infinitesimal criterion of invariance \eqref{inv-criter} is satisfied
$$\pr{2}\mathbf{v}(\Delta u)=\varphi^{xx}+\varphi^{yy}=-2\xi_x \Delta u=0$$ on the solution surface. The linearity of the equation implies that it also admits the additional trivial  symmetries $u\gen u$ and $\rho(x,y)\gen u$ with $\Delta \rho=0$ (one can multiply solutions by constants and add them).  The symmetry condition then becomes $\pr{2}(u\gen u)(\Delta u)=\Delta u=0$ and  $\pr{2}(\rho\gen u)(\Delta u)=\Delta \rho=0$ on solutions. Obviously, the second one is satisfied if $\rho$ is an arbitrary harmonic function.

The special choice   $(\xi,\eta)=(x^2-y^2,2xy)$ leads to the conformal invariance of the Laplace equation. The one-parameter conformal symmetry group corresponding to the vector field $\mathbf{c}_x=(x^2-y^2)\gen x+2xy\gen y$ is easily obtained solving the following complex initial value problem for $\tilde{z}(x,y;\varepsilon)$ (more precisely, integrating the analytic vector field $z^2\gen z$)
$$\frac{d\tilde{z}}{d\varepsilon}=\tilde{z}^2=(\tilde{x}^2-\tilde{y}^2)+2i\tilde{x}\tilde{y}$$ with the condition $\tilde{z}(x,y;0)=z(x,y)=x+i y$. The flow is given by
$$\tilde{z}=\frac{z}{1-\varepsilon z}=\frac{z-\varepsilon|z|^2}{(1-\varepsilon z)(1-\varepsilon \bar{z})}.$$ We separate the real and complex parts of $\tilde{z}$ to obtain the following (not necessarily globally-defined) symmetry group $\exp(\varepsilon \mathbf{c}_x)(x,y)$
\begin{equation}\label{conf-cx}
  \tilde{x}=\frac{x-\varepsilon(x^2+y^2)}{1-2\varepsilon x+\varepsilon^2(x^2+y^2)}, \quad \tilde{y}=\frac{y}{1-2\varepsilon x+\varepsilon^2(x^2+y^2)},
\end{equation}
possessing the invariant function $\zeta(x,y)=y(x^2+y^2)^{-1}$, satisfying the relation $\zeta(\tilde{x},\tilde{y})=\zeta(x,y)$  on $\mathbb{R}^2\setminus (0,0)$. $\zeta(x,y)$ is readily obtained by eliminating the group parameter $\varepsilon$ in \eqref{conf-cx} (see the next Section for a more precise definition and infinitesimal derivation of invariant function of a symmetry group).

It is a well-know fact that the inversion map  $I(x,y)=(x^2+y^2)^{-2}(x,y)$, $(x,y)\ne 0$ (an involution: $I^{-1}=I$) is a discrete (not connected) symmetry, i.e. if $f(x,y)$ satisfies the Laplace equation, so does $f((x^2+y^2)^{-2}x,(x^2+y^2)^{-2}y)$.
We observe that the map $(\hat{x},\hat{y})=I(x,y)$ also provides the coordinates rectifying $\mathbf{c}_x$ to $-\gen{\hat{x}} $.

Conjugating any symmetry  of the equation by $I$ will produce a new symmetry (a conformal mapping here). Indeed, the pushforward $I_{*}$ of the vector field $-\gen x$ through $I$ is $I_{*}(-\gen x)=\tilde{\mathbf{c}}_{\tilde{x}}$, where tilde means that the vector field is written in the new coordinates.
So $\exp(\varepsilon \mathbf{c}_x)(x,y)$  can be recovered by conjugating the translational group along the $x$-axis: $x\to x-\varepsilon$ by $I$
$$\exp\curl{\varepsilon \mathbf{c}_x}(x,y)=I(x,y)\circ\exp\curl{-\varepsilon \gen x}\circ I(x,y). $$
Similarly, since $I_{*}(-\gen y)=\tilde{\mathbf{c}}_{\tilde{y}}=2\tilde{x}\tilde{y}\gen {\tilde{x}}+(\tilde{y}^2-\tilde{x}^2)\gen {\tilde{y}}$,  conjugating  $-\gen y$ by $I$,
$$\exp\curl{\varepsilon \mathbf{c}_y}(x,y)=I(x,y)\circ\exp\curl{-\varepsilon \gen y}\circ I(x,y). $$
generates another conformal transformation
\begin{equation}\label{conf-cy}
  \tilde{x}=\frac{x}{1-2\varepsilon y+\varepsilon^2(x^2+y^2)}, \quad \tilde{y}=\frac{y-\varepsilon(x^2+y^2)}{1-2\varepsilon y+\varepsilon^2(x^2+y^2)}.
\end{equation}
The one-parameter group transformations generated by the elements of the abelian subalgebra  $\curl{\mathbf{c}_x, \mathbf{c}_y}$ are conformal because they leave form invariant the planar metric:
$$d\tilde{x}^2+d\tilde{y}^2=\lambda(x,y;\varepsilon)(dx^2+dy^2),$$ for some function $\lambda$ (conformal factor). For $\mathbf{c}_x$, $\lambda=\left(1-2\varepsilon x+\varepsilon^2(x^2+y^2)\right)^{-2}$. Note that the inversion itself is also a conformal mapping with $\lambda=(x^2+y^2)^{-2}$. The full (special) conformal transformations are given by the translation group conjugation $I\circ\exp\{\varepsilon_1 \gen x+\varepsilon_2  \gen y\}\circ I(x,y)$ in the form
\begin{equation}\label{full-special-conf}
  \tilde{\mathbf{x}}=(\tilde{x},\tilde{y})=\frac{\mathbf{x}-\boldsymbol{\varepsilon}(x^2+y^2)}{1-2\boldsymbol{\varepsilon}\cdot\mathbf{x} +\varepsilon^2(x^2+y^2)}, \quad \boldsymbol{\varepsilon}=(\varepsilon_1,\varepsilon_2), \quad \mathbf{x}=(x,y).
\end{equation}
Note that this transformation is not  globally defined because  the conformal factor $\lambda(x,y;\varepsilon_1,\varepsilon_2)=1-2\boldsymbol{\varepsilon}\cdot\mathbf{x} +\varepsilon^2(x^2+y^2)=0$ at the point $\mathbf{x}=\boldsymbol{\varepsilon}/\varepsilon^2$.

We conclude that action of this group on solutions states that $u=f(\tilde{x},\tilde{y})$ is also a solution, whenever $f(x,y)$ is solution to the Laplace equation. For example, with the help of the invariant $\zeta$, the radial solution $f(x,y)=\log(x^2+y^2)$ or the angular solution $f(x,y)=\arctan(y/x)$, among many others (homogeneous harmonics) can be mapped to produce the new solutions
$$u=\log\frac{x^2+y^2}{1-2\varepsilon x+\varepsilon^2(x^2+y^2)},  \quad u=\arctan\frac{y}{x-\varepsilon(x^2+y^2)}.$$

Adding to $\mathbf{c}_x$ and $\mathbf{c}_y$ the subalgebras obtained by other  choices $(\xi,\eta)=(1,0), (0,1)$, $(\xi,\eta)=(-y,x)$ and, $(\xi,\eta)=(x,y)$ leading to the translational, rotational and dilatational  invariance, in terms of vector fields, $\mathbf{p}_x=\gen x$, $\mathbf{p}_y=\gen y$, $\mathbf{j}=-y\gen x+x\gen y$, $\mathbf{d}=x\gen x+y\gen y$, respectively,  we obtain the 6-dimensional Lie algebra  of the conformal group $\conf(\mathbb{R}^2)$ of the Euclidean plane $\mathbb{R}^2$, isomorphic to $\Ort(3,1)$, the Lorentz group of four-dimensional Minkowski space \cite{Gonzalez-LopezKamranOlver1990}. Obviously,  the subalgebra spanned by $\curl{\mathbf{p}_x,\mathbf{d},\mathbf{c}_x}$ is $\Sl(2,\mathbb{R})$. This symmetry group is  the two-dimensional analogue of the full conformal group in dimensions $n\geq 3$. Note that the full conformal group in the plane $\mathbb{R}^2\cong \mathbb{C}$ is infinite-dimensional, with the Lie group $\Ort(3,1)$ as its maximal finite-dimensional subgroup, because any analytic function $f:\mathbb{C}\to \mathbb{C}$ leads to a conformal transformation (In our case $f(z)=z,iz,z^2$). We have excluded the trivial symmetry algebra stemming from the linearity of the PDE. Their non-zero commutators satisfy
$$[\mathbf{p}_{x,y},\mathbf{d}]=\mathbf{p}_{x,y}, \quad [\mathbf{j},\mathbf{p}_x]=-\mathbf{p}_y, \quad  [\mathbf{j},\mathbf{p}_y]=\mathbf{p}_x, \quad [\mathbf{p}_1,\mathbf{c}_x]=[\mathbf{p}_2,\mathbf{c}_y]=2 \mathbf{d},$$
$$[\mathbf{p}_x,\mathbf{c}_y]=-[\mathbf{p}_y,\mathbf{c}_x]=-2\mathbf{j},  \quad [\mathbf{d},\mathbf{c}_{x,y}]=\mathbf{c}_{x,y}, \quad [\mathbf{j},\mathbf{c}_x]=-\mathbf{c}_y,  \quad [\mathbf{j},\mathbf{c}_y]=\mathbf{c}_x.$$

A nonlinear variant of the Laplace equation, known as the conformal scalar curvature equation, or the elliptic Liouville's equation, occurs in the study of isothermal coordinates in differential geometry and has the form
\begin{equation}\label{conf-pde}
  u_{xx}+u_{yy}=K e^u,
\end{equation}
where $K$ is constant (Gaussian curvature).

The conformal symmetry structure of this equation on the $(x,y)$-plane is preserved.  The vector field generating the symmetry group $G$ of the equation is given by
$$\mathbf{v}=\xi\gen x+\eta\gen y-2\xi_x\gen u,$$ where $\xi(x,y)$, $\eta(x,y)$ satisfy the Cauchy--Riemann equations. For $(\xi,\eta)=(x^2-y^2,2xy)$, $\mathbf{v}=(x^2-y^2)\gen x+2xy\gen y-4x\gen u$. We solve the initial value problem $d\tilde{u}/d\varepsilon=-4\tilde{x}$, $\tilde{u}(x,y;0)=u(x,y)$ using \eqref{conf-cx} and find the transformation of $u$ under the group action:
\begin{equation}\label{tr-u}
 \tilde{u}(x,y;\varepsilon)=2\ln \sigma(x,y;\varepsilon)+u(x,y), \quad \sigma(x,y;\varepsilon)=1-2\varepsilon x+\varepsilon^2(x^2+y^2).
\end{equation}
We note that
$$\sigma(x,y;\varepsilon)=\sigma(\tilde{x},\tilde{y};-\varepsilon)^{-1}=1+2\varepsilon \tilde{x}+\varepsilon^2(\tilde{x}^2+\tilde{y}^2).$$
Application of the one-parameter transformation group defined by \eqref{conf-cx} and \eqref{tr-u} to a solution $f(x,y)$, where the coordinates $(x,y)$  are written in terms of $(\tilde{x},\tilde{y})$  leads to the transformed new solution $u_{\varepsilon}(x,y)$ (after the tildes are removed)
$$u_{\varepsilon}(x,y)=-2\ln \sigma(x,y;-\varepsilon)+f(\tilde{x},\tilde{y}),$$
where
$$\tilde{x}=\frac{x+\varepsilon(x^2+y^2)}{1+2\varepsilon x+\varepsilon^2(x^2+y^2)},  \quad \tilde{y}=\frac{y}{1+2\varepsilon x+\varepsilon^2(x^2+y^2)}.$$
\end{example}

\begin{remark}
The Laplace equation in $\mathbb{R}^n$ with $n\geq 3$ is invariant only under a finite dimensional conformal Lie symmetry group of $\mathbb{R}^n$ with dimension $\binom{n+2}{2}=(n+1)(n+2)/2$, consisting of the groups of translations, rotations, dilation  and conformal transformations (obtained by conjugating the $n$-components of the translational group via inversion $I(\mathbf{x})=|\mathbf{x}|^{-2}\mathbf{x}$) on $\mathbb{R}^n\setminus\curl{0}$.

The  symmetry algebra of the Laplace equation (see \cite{SattingerWeaver1986} for its derivation)
$$\Delta_nu(x)=0,  \quad x=(x_1,x_2,\ldots,x_n)\in\mathbb{R}^n$$  is spanned by
\begin{equation}\label{sym-Laplace}
  \begin{split}
     &  \mathbf{p}_i=\gen {x_i}, \quad \mathbf{d}=\sum_{j=1}^{n}x_j\gen {x_j},\quad
        \mathbf{j}_{ij}=x_i\gen {x_j}-x_j\gen {x_i}, \quad i\ne j,\\
     & \mathbf{c}_i=2x_i\mathbf{d}-r^2\gen {x_i}+(2-n)x_iu\gen u, \quad r^2=\sum_{j=1}^{n}x_j^2,\\
     &  \mathbf{m}=u\gen u, \quad \mathbf{v}(\rho)=\rho(x)\gen u,  \quad \Delta_n \rho=0,
  \end{split}
\end{equation}
where $i,j=1,2,\ldots, n$.
The $(n+1)(n+2)/2$-dimensional Lie algebra $\{\mathbf{p}_i,\mathbf{d},\mathbf{j}_{ij},\mathbf{c}_i\}$ generates the conformal group $\conf(\mathbb{R}^n)$ of the Euclidean plane $\mathbb{R}^n$.

We note that the conformal transformations of $\mathbf{R}^n$ are generated by the conformal vector fields $\mathbf{v}=\sum_{i=1}^n \xi_i(x)\gen {x_i}$, where the coefficients $\xi_i$ satisfy the conformal Killing equations
$$\frac{\partial\xi_i}{\partial x_j}+\frac{\partial\xi_j}{\partial x_i}=\lambda(x)\delta_{ij}, \quad i,j=1,2,\ldots, n$$ for some undetermined function $\lambda$ (the conformal factor).

The linear wave equation $u_{tt}=\Delta u$, $u=u(t,x)$, $(t,x)\in\mathbb{R}^{n+1}$ is invariant under a Lie point symmetry algebra isomorphic to  the Lorentz group $\Ort(n+1,2)$, of dimension $\binom{n+3}{2}=(n+2)(n+3)/2$, $n\geq 2$ in a Minkowski space with an indefinite underlying metric $ds^2=dt^2-dx_1^2-\cdots-dx_n^2$.

The nonlinear wave (or Klein--Gordon)  equation $u_{tt}-u_{xx}=F(u)$ is invariant under the Poincar\'{e} group $\PP(1,1)$  of $1+1$-dimensional Minkowski plane for any $F(u)$ with $F''\ne 0$. Its Lie symmetry algebra is generated by the translational vector fields and Lorenz boost
$$\mathbf{v}_1=\gen t,  \quad \mathbf{v}_2=\gen x, \quad  \mathbf{v}_3=t\gen x+x\gen t.$$ For two specific forms of $F(u)$, the symmetry algebra is larger (see Section \ref{group-class-prob} for the group classification problem). The additional vector field for $F(u)=F_0 u^k$ is $\mathbf{v}_4=t\gen t+x\gen x+\frac{2}{1-k}u\gen u$, and $\mathbf{v}_4=t\gen t+x\gen x-\frac{2}{\lambda}\gen u$ for $F(u)=F_0 e^{\lambda u}$.

The second special case  is known as the hyperbolic Liouville equation, which in the light-cone (characteristic) coordinates $r=t+x$, $s=t-x$ has the form
\begin{equation}\label{hyper-Liouv}
  u_{rs}=ae^u.
\end{equation}
This equation is invariant under the infinite-dimensional symmetry algebra generated by the vector field
\begin{equation}\label{vf-Liouv}
  \mathbf{v}=\xi(r)\gen r+\eta(s)\gen s-(\xi'(r)+\eta'(s))\gen u,
\end{equation}
where $\xi(r)$ and $\eta(s)$ are arbitrary functions. This is true for any $a$ because
$$\pr{2}\mathbf{v}(u_{rs}-ae^u)=-(u_{rs}-ae^u)(\xi'(r)+\eta'(s)).$$
$\mathbf{v}$ generates the infinite dimensional symmetry group depending on two arbitrary functions
\begin{equation}\label{inf-dim-sym-group}
  \tilde{r}=f(r), \quad \tilde{s}=g(s), \quad \tilde{u}=u-\ln (f'(r) g'(s)),  \quad f'g'\ne 0.
\end{equation}
The group action implies that if $F(r,s)$ is a solution,  so is the function $u=F(\tilde{r},\tilde{s})+\ln (f'g')$.
The general solution of \eqref{hyper-Liouv} for $a\ne 0$ can be found by acting on a particular solution of the form $$u=F(r+s)=\ln \frac{2}{a(r+s)^2}$$ as
\begin{equation}\label{Liouv-sol}
  u=F(\tilde{r}+\tilde{s})+\ln (f'g')=\ln \left(\frac{2}{a}\frac{f'(r)g'(s)}{(f(r)+g(s))^2}\right).
\end{equation}
So the general solution of
$$u_{tt}-u_{xx}=ae^{u}$$
can be expressed in terms of two arbitrary functions as
$$u=\ln \left(\frac{8}{a}\frac{f'(t+x)g'(t-x)}{(f(t+x)+g(t-x))^2}\right).$$
This solution was obtained by Liouville as early as 1853.

The linear case $F''=0$ is quite different, the symmetry group is the infinite-dimensional conformal group.

\end{remark}
\section{Differential invariants}
Given a Lie algebra $\lie$, characterization of all invariant equations, equations that remain invariant under the symmetry group $G$ of $\lie$ requires the notion of differential invariants, which are functions unaffected by the action of $G$ on some manifold $M$. An ordinary invariant  is a $C^{\infty}(J(x,u))$ function $I(x,u)$ on $J(x,u)\subset M$, which satisfies $I(g.(x,u))=I(x,u)$ for all group elements $g\in G$ and coordinates $(x,u)$.
\begin{definition}
A differential invariant of order $n$ of a connected transformation group $G$ is a differential function $I(x,u^{(n)})$ on the jet space $J^{n}$ if $I(g^{(n)}.(x,u^{(n)}))=I(x,u^{(n)})$ for all $g\in G$ and $(x,u^{(n)})\in J^{n}$.
\end{definition}
An ordinary invariant is a differential invariant of order 0.
The following infinitesimal invariance criterion for differential invariants serves to determine differential invariants of a given connected group of transformations in a simple manner by just solving a system of linear first order PDEs.
\begin{proposition}
A function $I$ is a differential function for a connected group $G$ if and only if it is annihilated by all the prolonged vector fields (infinitesimal generators)
\begin{equation}\label{criter-DI}
 \mathbf{v}^{(n)}(I)\equiv\pr{n}\mathbf{v}(I)=0
\end{equation}
for all $\mathbf{v}\in \lie$.
\end{proposition}
An function $I$ is an ordinary invariant if and only if $\mathbf{v}(I)=0$. For a general vector field
$$\mathbf{v}=\sum_{i=1}^n \xi_i(x)\gen{x_i},$$ the  coordinates $y=\eta(x)$ rectifying $\mathbf{v}=\gen{y_1}$ are found by solving the first order partial differential equations $\mathbf{v}(\eta_1)=1$, $\mathbf{v}(\eta_i)=0$, $i>1$. So the new coordinates $y_i(x)$ are the functionally independent invariants of the one-parameter group generated    by $\mathbf{v}$.
\begin{remark}
The dimension of the space $J^{n}$ is $\dim J^{n}=p+q\binom{p+n}{n}$. The number of derivative coordinates of order exactly $n$ is given by $q_n=\dim J^{n}-\dim J^{n-1}=q\binom{p+n-1}{n}$.
\end{remark}

The number of functionally independent differential invariants of order $n$ is equal to
\begin{equation}\label{number-DI}
  k=\dim J^n-(\dim G-\dim G_0),
\end{equation}
where $G_0$ is the stabilizer subgroup (also called the isotropy group) of a generic point on $J^n$. $d=\dim G-\dim G_0$ is the dimension of the orbit of $G$ at a generic point.
An equivalent formula for $k_n$ is
$$k_n=\dim J^n-\rank Z\geq 0,$$
where $Z$ is the matrix of size $r\times \dim J^n$ formed by
the coefficients  of the $n$-th prolongations $\pr{n}(\mathbf{v}_\nu)$, $\nu=1,2,\ldots,r$ of the basis vector fields $\mathbf{v}_1,\ldots,\mathbf{v}_r$ of the Lie algebra $\lie$ of the group $G$ as rows
$$\mathbf{v}_{\nu}=
 \sum_{i=1}^p \xi_{i,\nu}(x,u)\gen{x_i}+\sum_{\alpha=1}^q \varphi_{\alpha,\nu}(x,u)\gen{u_\alpha},  \quad \nu=1,2,\ldots,r.
$$
The rank $m$ of $Z$ is calculated at a generic point of $J^n$. For the special case of one independent and one dependent variable $p=q=1$ we have $\dim J^n=n+2$ and the number of functionally independent invariants is $k_n=n+2-m$.

The set of $n$-th order differential invariants  $I_1, I_2,\ldots,I_k$ of $\lie$ will be denoted by $\mathcal{I}_n(\lie)$. This set is an $\mathbb{R}$-algebra. This means that $\mathcal{I}_n(\lie)$ is a vector space over the field $\mathbb{R}$ and satisfies the property that  any arbitrary smooth function $H(I_1,\ldots,I_k)$ of the set of differential invariants $I_1, I_2,\ldots,I_k$ is also a differential invariant, i.e. if $I_1, I_2,\ldots,I_k\in \mathcal{I}_n(\lie)$, then $H(I_1,\ldots,I_k)\in \mathcal{I}_n(\lie)$. They also satisfy the inclusions
$$\mathcal{I}_0(\lie)\subset \mathcal{I}_1(\lie)\subset \ldots \mathcal{I}_n(\lie)\subset \ldots$$ The algebra $\bigcup_{n=0}^{\infty}\mathcal{I}_n(\lie)$ is called the algebra of differential invariants.

\subsection{Invariant differentiation}
Lie \cite{LieEngel1891, Lie1927, LieHermann1975} and Tresse \cite{Tresse1894} introduced the notion of "invariant" differential operators to obtain $(n+1)$-st order differential invariants from $n$-th order ones. This enables one to produce all the higher order functionally independent differential invariants by successive application of the invariant operators to lower order invariants. The situation is easier when there is only  one dependent variable. Let the group $G$ with Lie algebra $\lie$ act on the basic space $E=J^{0}(\mathbb{R},\mathbb{R}^{n})\simeq X\times U$.

\begin{proposition}\label{prop}
Suppose that $I(x,u^{(n)})$ and $J(x,u^{(n)})$ are functionally independent invariants, at least one of which has order exactly $n$. Then the ratio $D_x J/D_x I$ (the Tresse derivative) is an $(n+1)$-st order differential invariant.
\end{proposition}
If $I(x,u^{(n)})$ is any given differential invariant, then $\mathcal{D}=(D_x I)^{-1}D_x$ is an invariant differential operator so that iterating on $\mathcal{D}$ one can generate an hierarchy $\mathcal{D}^{k}J$, $k=0,1,2,\ldots,$ of  higher order differential invariants.

If $z=I(x,u)$ and $w=J(x,u,u_1)$ are a complete set of functionally independent  invariants of the first prolongation $\pr{1}\lie$, i.e. they form the basis of $\mathcal{I}_1(\lie)$, then $I, J$ together with the derivatives $\mathcal{D}^{k}J=d^kw/dz^k$, $k=1,2,\ldots,n-1$ generate a complete set of functionally independent invariants for the prolonged algebra $\pr{n}\lie$ for $n\geq 1$.  They all satisfy the infinitesimal invariant condition $\pr{n}\mathbf{v}(\mathcal{D}^{k}J)=0$ for a vector field $\mathbf{v}\in\lie$.

\begin{example}\label{exa-seq-dif-inv}
  We know from Example \ref{pr-tr-ex} that the algebra of vector fields $\mathbf{v}=x^2\gen x+xu\gen u$ on $J^{0}(\mathbb{R},\mathbb{R})$ has the first order differential invariants $I=u/x$, $J=x u_1-u$, which satisfy $\mathbf{v}(I)=0$, $\pr{1}\mathbf{v}(J)=0$. The Tresse derivative $DJ/DI=J^{-1}(x^3 u_2)$ gives the second order differential invariant. We can take it as $J_2=J DJ/DI= x^3 u_2$ as it is needed only up to functional independence. Iterating the Tresse derivatives and multiplying by $J$ we obtain the sequence of all other differential invariants as $J_k=x^2 D_x J_{k-1}$, $k=3,4,\ldots$.
\end{example}

Determination of a complete set of functionally independent differential invariants (the basis $\mathcal{I}_n(\lie)$)  allows us to construct classes of differential equations with a prescribed symmetry algebra $\lie$.

\begin{theorem}\label{Thm-inv-charac}
 If the differential functions
$$I_1(x,u^{(n)}), I_2(x,u^{(n)}),\ldots,I_k(x,u^{(n)})\in \mathcal{I}_n(\lie)\subset J^n$$ form a set of functionally independent $n$-th order differential invariants  of $G$, then a system of $n$-th order differential equations are invariant under $G$ if and only if it can be written in terms of the differential invariants:
\begin{equation}\label{inv-charac}
  \mathsf{E}_{\nu}(x,u^{(n)})=H_{\nu}(I_1,I_2,\ldots, I_k)=0, \quad \nu=1,2,\ldots, N.
\end{equation}
\end{theorem}
Invariant equations obtained in this way are called strongly invariant ($\pr{n}H_{\nu}(I_1,\ldots,I_k)=0$ is satisfied everywhere).


\begin{example}\label{2nd-ODE-build}
We find all second order equations invariant under the abelian subgroup of the projective group $\SL(3,\mathbb{R})$ generated by
$$\mathbf{v}_1=x^2\gen x+xy \gen y,  \quad \mathbf{v}_2=xy\gen x+y^2 \gen y.$$ From Example \ref{exa-seq-dif-inv} we  know the following set of second order differential invariants
$$I_1=\frac{y}{x}, \quad I_2=y-xy_1,  \quad I_3=x^3y_2.$$ The second prolongation of $\mathbf{v}_2$ is
$$\pr{2}\mathbf{v}_2=\mathbf{v}_2+y_1 (y-xy_1)\gen {y_1}-3xy_1 y_2\gen {y_2}.$$
Imposing the condition $\pr{2}\mathbf{v}_2(H)=0$ and changing to the invariants as new coordinates we find that $H(I_1,I_2,I_3)$ satisfies
$$I_2^2 \frac{\partial H}{\partial I_2}+3I_2I_3\frac{\partial H}{\partial I_3}=0.$$

Solving this PDE by the method of characteristics, it follows that there are two independent invariants  $I=I_1=y/x$ and $J=I_3I_2^{-3}=x^3(y-xy_1)^{-3}y_2$. The most general equation can now be written as
\begin{equation}\label{2nd-inv-Eq}
 y_2=x^{-3}(y-xy_1)^{3}G\left(\frac{y}{x}\right),
\end{equation}
where $G$ is an arbitrary function.

The rectifying coordinates for $\mathbf{v}_1$ to $\mathbf{\tilde{v}}_1=\gen s$ are $r=y/x$ (an invariant) and $s=-1/x$. In terms of $r, s$, $\mathbf{v}_2$ gets transformed to $\mathbf{\tilde{v}}_2=r\gen s$. The  invariant equation corresponding to the abelian algebra $\curl{\gen s,r\gen s}$ (compare with the canonical realization $A_{2,2}$ of \eqref{two-dim-cano-abelian}) is a linear one $d^2s/dr^2=G(r)$. We conclude that the same transformation  linearizes Eq. \eqref{2nd-inv-Eq}.

If we replace $\mathbf{v}_2$ by $\mathbf{v}_2=x\gen x+ky\gen y$, this time $H(I_1,I_2,I_3)$ has to satisfy the zero-degree quasi-homogeneous function PDE
$$(k-1)I_1\frac{\partial H}{\partial I_1}+kI_2\frac{\partial H}{\partial I_2}+(k+1)I_3\frac{\partial H}{\partial I_3}=0.$$ Integrating the characteristic equations of this PDE  we obtain the invariants $I=I_1^{k/(1-k)}I_2$, $J=I_1^{(k+1)/(1-k)}I_3$, $k\ne 1$, while for $k=1$, we have $I=I_1$ and $J=I_3 I_2^{-2}$. In the former case, the invariant equation will have the form $J=G(I)$, or
$$y_2=x^{-(n+3)}y^nG(I),  \quad I=\left(\frac{y}{x}\right)^{-(n+1)/2}(y-xy_1),  \quad k=\frac{n+1}{n-1}.$$  In this case, the algebra belongs to the nonabelain realization $A_{2,3}$ of \eqref{two-dim-cano-nonabelian}.  For the specific choice of $G=K=\text{const.}$ it reduces to the celebrated Emden--Fowler equation (see also \eqref{EF-eq}).

In the latter case, we have the nonabelian algebra of linearly connected (or rank-one) vector fields with $\mathbf{v}_1=x(x\gen x+y\gen y)=x\mathbf{v}_2$, which is the $A_{2,4}$ \eqref{two-dim-cano-nonabelian} canonical form $\curl{\gen s, s\gen s}$, up to change of coordinates $r=y/x$, $s=-1/x$. The corresponding invariant equation is
$$y_2=(y-xy_1)^2 F\left(\frac{y}{x}\right).$$ Changing to new coordinates $(r,s)$ linearizes this equation to
$$\frac{d^2 s}{dr^2}+F(r)\frac{d s}{dr}=0.$$
\end{example}

\begin{example}
 We construct all third order ODEs invariant under the solvable group $\E(2)$ of rigid motions in the plane (isometries of the Euclidean space $\mathbb{R}^2$) of the $\e(2)$ algebra of symmetries composed of translations along $x$ and $y$ axes and the planar rotations
\begin{equation}\label{e2-basis}
  \mathbf{v}_1=\gen x,  \quad \mathbf{v}_2=\gen y, \quad \mathbf{v}_3=-y\gen x+x \gen y
\end{equation}
with non-zero commutators
$$[\mathbf{v}_1,\mathbf{v}_3]=\mathbf{v}_2,  \quad [\mathbf{v}_2,\mathbf{v}_3]=-\mathbf{v}_1.$$
Prolongations of $\mathbf{v}_1$ and $\mathbf{v}_2$ do not alter their local form, but the third order prolongation of $\mathbf{v}_3$ is given by
$$\pr{3}\mathbf{v}_3=\mathbf{v}_3+(1+y_1^2)\gen {y_1}+3y_1y_2\gen {y_2}+(4y_1y_3+3y_2^2)\gen {y_3}.$$
We note that the Euclidean group $\E(2)$ has no ordinary invariants on the space $(x,y)$ ($k=n+2-m=2-2=0$), nor differential invariant of the first order ($k=1+2-3=0$) because the group acts transitively on $(x,y)$ and $(x,y,y_1)$,  but there are differential invariants of order $\geq 2$. We can not use invariant differentiation process to find higher order ones. If $I(x,y,y_1,y_2,y_3)$ is a differential invariant of order three of $\e(2)$, then $\pr{3}\mathbf{v}_1(I)=0$, $\pr{3}\mathbf{v}_2(I)=0$, $\pr{3}\mathbf{v}_3(I)=0$. From the first two equations, $I$ must be independent of $x$ and $y$ coordinates, namely $I(y_1,y_2,y_3)$.
We solve the characteristic system to find the two independent differential invariants, satisfying $\pr{3}\mathbf{v}_3(I)=0$.
The characteristic system is given by
\begin{equation}\label{char-sys}
  \frac{dy_1}{1+y_1^2}=\frac{dy_2}{3y_1y_2}=\frac{dy_3}{4y_1y_3+3y_2^2}.
\end{equation}
From the first characteristic equation, a second order invariant is $\kappa=(1+y_1^2)^{-3/2}y_2$ (the curvature). The other one is obtained by replacing $y_2$ by $\kappa(1+y_1^2)^{3/2}$, with $\kappa$ treated constant, in the last term and then integrating the first order linear ODE
$$\frac{dy_3}{dy_1}-\frac{4y_1}{1+y_1^2}y_3=3 \kappa^2(1+y_1^2)^2.$$
This provides us the differential invariant $$\zeta(y_1,y_2,y_3)=\kappa^{-2}y_3(1+y_1^2)^{-2}-3y_1=(1+y_1^2)y_2^{-2}y_3-3y_1$$
so that $\kappa$ and $\zeta$ form a basis of third order invariants of $\e(2)$ (a set of functionally independent invariants). The invariant equation now can be written as
\begin{equation}\label{3rd-ODE}
  (1+y_1^2)y_3-3y_1y_2^{2}=y_2^{2}H(\kappa),
\end{equation}
where $H$ is any smooth function of the curvature.

In view of proposition \ref{prop}, higher order differential invariants and invariant ODEs can be constructed using the Tresse derivatives of  $\kappa$ and $\zeta$. For instance, the ratio $D_x \zeta/D_x \kappa$ gives a fourth order invariant.

\end{example}

Third order ordinary differential equations invariant under three-di\-men\-si\-onal solvable algebras can be constructed in a similar manner. All such algebras with a two-dimensional abelian ideal of rank-two will be realized by the vector fields
\begin{equation}\label{sol-3-dim}
  \mathbf{v}_1=\gen x, \quad  \mathbf{v}_1=\gen y,  \quad  \mathbf{v}_3=(ax+cy)\gen x+(bx+dy)\gen y,
\end{equation}
satisfying the commutation relations
\begin{equation}\label{comm-solvable}
[\mathbf{v}_1,\mathbf{v}_2]=0,  \quad [\mathbf{v}_1,\mathbf{v}_3]=a \mathbf{v}_1+b \mathbf{v}_2,  \quad [\mathbf{v}_2,\mathbf{v}_3]=c \mathbf{v}_1+d \mathbf{v}_2,
\end{equation}
where the matrix
$$A=\begin{pmatrix}
      a & b \\
      c & d
    \end{pmatrix}$$
can be transformed into one of the following Jordan canonical forms
$$
\begin{pmatrix}
      1 & 0 \\
      0 & p
\end{pmatrix}
,\quad -1\leq p\leq 1, \quad
\begin{pmatrix}
      0 & 0 \\
      1 & 0
\end{pmatrix}
,\quad
\begin{pmatrix}
      q & -1 \\
      1 & q
\end{pmatrix}
\quad q\geq 0.
$$

A different algebra is obtained for each canonical form. For example, the third canonical form for $q=0$ ($a=d=0$, $b=-c=1$) corresponds to the algebra $\e(2)$ as discussed in the above example. The differential invariants of order three, $I(y_1,y_2,y_3)$,  can be found from the condition $\pr{3}\mathbf{v}_3(I)=0$, namely from integrating the first order homogeneous PDE
\begin{equation}\label{3rd-order-i}
[b+(d-a)y_1-cy_1^2]I_{y_1}+
(d-2a-3cy_1)y_2I_{y_2}+[-3cy_2^2+(d-3a-4cy_1)y_3]I_{y_3}=0.
\end{equation}

The second canonical form ($a=b=d=0$, $c=1$) is the nilpotent algebra with the center $\curl{\mathbf{v}_1}$. This means that we are dealing with the solvable algebra $$\mathbf{v}_3=\gen x,  \quad \mathbf{v}_2=\gen y,  \quad \mathbf{v}_3=y\gen x.$$
Calculating differential invariants from
$$\left[y_1^2\gen{y_1}+3y_1y_2\gen{y_2}+(3y_2^2+4y_1y_3)\gen{y_3}\right]I=0$$
we find the third order  invariant equation
$$y_1y_3=3y_2^2+y_1^5 H(\tau),  \quad \tau=y_2y_1^{-3},$$ where $H$ is an arbitrary function of its argument.

This equation and  other third order invariant equations can be integrated using three consecutive quadratures (see Example \ref{3rd-ord-quadrature} for the integration procedure).

\section{Reduction of order for ordinary differential equations}

\begin{theorem}
  Let the scalar ODE
$$\mathsf{E}(x,y^{(n)})=\mathsf{E}(x,y,y_1,\ldots,y_n)=0, \quad \frac{\partial \mathsf{E}}{\partial y_n}\ne 0$$ admit a one-parameter symmetry group $G$ generated by $\mathbf{v}\in \lie$. All nontangential solutions ($\mathbf{v}\in \lie$ is nowhere tangent to the graph of the solution) can be found by quadrature from the solutions to a reduced ODE $(\mathsf{E}/G)(x,y^{(n-1)})$.
\end{theorem}
\begin{proof}
  If we introduce rectifying (canonical or normal) coordinates $r=r(x,y)$ and $s=s(x,y)$ in which $\mathbf{v}$ generates a group of translation $r\to r$, $s\to s+\varepsilon$ with the corresponding normal form $\widehat{\mathbf{v}}=\gen s$.  Its prolongation $\pr{n}\mathbf{\widehat{v}}=\gen s$ and therefore the derivatives in the new coordinates remain unchanged so that from the invariance condition \eqref{inv-criter} it follows that the equation in normal form should be independent of the variable $s$, but  can depend on the derivatives. Therefore we have reduced our equation to one of order $n-1$, $(\mathsf{E}/G)(r,s',\ldots, s^{(n-1)})$ for the derivative $z=\vartheta(r)=ds/dr=s'(r)$. Once we know the solution of the reduced equation, the solution to the original one is obtained by a quadrature $s=\int \vartheta(r)dr$.
\end{proof}
The rectifying coordinates $r, s$ are constructed as solutions of the partial differential equations $\mathbf{v}(r)=0$, $\mathbf{v}(s)=1$. Note that the coordinate $r$ is an invariant of $\mathbf{v}(r)$ and  $r, s$ can be replaced by any arbitrary functions of $r$ and $s$.
For a first order equation $y'=F(x,y)$ admitting the symmetry generated by $\mathbf{v}=\xi(x,y)\gen x+\eta(x,y)\gen y$,  the determining equation is
\begin{equation}\label{det-eq-first-ODE}
  \xi F_x+\eta F_y=(\eta_y-\xi_x)F- \xi_y F^2+\eta_x.
\end{equation}
This equation admits the solution $\eta=F \xi$. But the corresponding vector fields $\mathbf{v}=\xi(x,y)(\gen x+F(x,y)\gen y)$  are everywhere tangential to solutions and do not serve our purpose for reduction because finding  canonical  coordinates equally require integrating the equation itself. Other than these trivial symmetries, the above determining equation can allow particular solutions for given $F$ leading to one-parameter symmetry groups. Transforming  to canonical coordinates  reduces the equation to quadrature. It is quite straightforward to see that, if $\eta-F \xi\ne 0$, the infinitesimal symmetry condition \eqref{det-eq-first-ODE} can be re-expressed as
$$\mu_x+(\mu F)_y=0,  \quad \mu(x,y)=(\eta-F \xi)^{-1}, $$
which implies the existence of an integrating factor $\mu(x,y)$ of the equation.

In practise it is more feasible to solve the inverse problem of constructing the most general first order ODEs admitting a given  group as a symmetry group. The same problem for higher order equations are equally useful.

\begin{example}\label{fiber-preser}

We consider the following one-parameter local group of transformation (a special form of the so-called fiber-preserving transformation in which the changes in $x$ are not affected by the dependent variable $y$)
\begin{equation}\label{tr-1}
 \Phi_{\varepsilon}(x,y):\quad \tilde{x}=X(x;\varepsilon),  \quad \tilde{y}=Y(x;\varepsilon)y
\end{equation}
with the infinitesimal generator
\begin{equation}\label{inf-1}
  \mathbf{v}=\xi(x)\gen x+\eta(x)y\gen y,
\end{equation}
and its prolongation
$$\pr{1}\vf=\vf+[\eta' y+(\eta-\xi')y']\gen{y'}.$$
Solving the characteristic equations of the first order PDE $\pr{1}\vf(I)=0$
$$\frac{dx}{\xi(x)}=\frac{dy}{\eta(x)y}=\frac{dy'}{\eta' y+(\eta-\xi')y'}$$
we find the first order fundamental invariants
\begin{equation}\label{inv}
  r(x,y)=\nu(x)y,  \quad w(x,y,y')=\nu(x)(\xi y'-\eta y),  \quad \nu(x)=\exp\curl{-\int^x\frac{\eta(t)}{\xi(t)}dt}.
\end{equation}
From Theorem \ref{Thm-inv-charac}, the corresponding ODE can be expressed in terms of invariants in the form $F(r,w)=0$ or $w=H(r)$, more precisely
\begin{equation}\label{inv-ode}
  y'-\frac{\eta(x)}{\xi(x)}y=\frac{H(\nu(x)y)}{\nu(x)\xi(x)},
\end{equation}
where $F$ and $H$ are arbitrary functions of their arguments. In terms of canonical coordinates it has the form, which is independent of $s$ (invariant under the translations $s\to s+\varepsilon$),
\begin{equation}\label{inv-cano}
  \frac{dr}{ds}=H(r),  \quad w=\frac{dr}{ds},  \quad s=\int^x\frac{dt}{\xi(t)}.
\end{equation}
This means  that Eq. \eqref{inv-cano} can be integrated by quadrature (separation of variables).
This equation involves interesting types of ODEs like Abel equation of the second type or Riccati equation.

For example, if we choose $\xi$ arbitrary, $\eta=-1$ and $H(\tau)=A+B \tau^{-1}+\tau$, $A,B\ne 0$ we have the Abel equation
\begin{equation}\label{Abel}
  yy'=\frac{A}{\xi\nu}y+\frac{B}{\xi\nu^2}
\end{equation}
with the one-parameter symmetry group generated by $\mathbf{v}=\xi(x)\gen x-y\gen y$. In terms of $r,s$, its general solution in implicit form is given by the quadrature
\begin{equation}\label{quadrat}
 \int\frac{rdr}{r^2+Ar+B}=s+c,
\end{equation}
where $c$ is an integration constant. If $\Delta=A^2-4B<0$, it is useful to use
$$\int\frac{rdr}{r^2+Ar+B}=\frac{1}{2}\ln (r^2+Ar+B)-\frac{A}{2}\int\frac{dr}{r^2+Ar+B}$$
together with the formula \eqref{indef-int}  for  the integral on the right side.

For the choice $\xi(x)=x$, $A=1$, $B=-2$ ($\Delta=9>0$), Eq. \eqref{Abel} specializes to
$$yy'=\frac{1}{x^2}y-\frac{2}{x^3}$$
with solution
$$(xy+2)^2(xy-1)=k x^3.$$

The choice $H(\tau)=A \tau^2+B \tau+C$ leads to a Riccati equation depending on two arbitrary functions and three parameters
\begin{equation}\label{Riccati-gen}
  y'=a(x)y^2+b(x)y+c(x),
\end{equation}
where
\begin{equation}\label{Riccati-gen-coeff}
  a(x)=\frac{A\nu}{\xi},   \quad b(x)=\frac{\eta+B}{\xi},  \quad c(x)=\frac{C}{\nu\xi}, \quad \xi\ne 0.
\end{equation}
This equation is  invariant under the symmetry \eqref{inf-1}.
Conversely, one can show that the most general Riccati equation of the form \eqref{Riccati-gen} invariant under the group \eqref{tr-1} should have the coefficients given by \eqref{Riccati-gen-coeff}. In \eqref{Riccati-gen-coeff}, choosing, in particular,
$a(x)=1, b(x)=0$, $\xi=A+Bx$, the following normal form of \eqref{Riccati-gen-coeff} is obtained
\begin{equation}\label{Riccati-normal}
  y'=y^2+c(x),  \quad c(x)=\frac{C}{A}\nu^{-2}=\frac{C}{A}(A+Bx)^{-2}
\end{equation}
admitting the one-dimensional symmetry algebra $\mathbf{v}=(A+Bx)\gen x-By\gen y$.

From its separable form in coordinates $r, s$ (namely, Eq. \eqref{inv-cano}), it follows the following solution
\begin{equation}\label{gen-sol}
  \int\frac{dr}{Ar^2+Br+C}=s+c.
\end{equation}
The integral on the left side depends on the discriminant $\Delta=B^2-4AC$
\begin{equation}\label{indef-int}
\begin{cases}
  \frac{-2}{\sqrt{\Delta}}\arctanh\frac{2Ar+B}{\sqrt{\Delta}}, &  \Delta>0, \;\;\\
  \frac{-2}{2Ar+B}, &  \Delta=0, \;\;\\
  \frac{2}{\sqrt{-\Delta}}\arctan\frac{2Ar+B}{\sqrt{-\Delta}}, &  \Delta<0.
\end{cases}
\end{equation}
The particular choice $\xi=1$, $\eta=\alpha=\text{const.}$  ($\nu=e^{-\alpha x}$) leads to the Riccati equation
\begin{equation}\label{Riccati-ABC}
  y'=Ae^{-\alpha x}y^2+(B+\alpha) y+Ce^{\alpha x}
\end{equation}
admitting the symmetry group $(x,y)\to (x+\varepsilon, e^{\alpha \varepsilon}y)$ generated by $\mathbf{v}=\gen x+\alpha y\gen y$. This equation can be replaced by
\begin{equation}\label{Riccati-ABC-2}
  y'=Ae^{-\alpha x}y^2+B y+Ce^{\alpha x}
\end{equation}
by redefining the parameter $B$ as $B-\alpha$.

The other choice $\xi=x^2$, $\eta=x$ ($\nu=1/x$) gives the Riccati equation
\begin{equation}\label{Riccati-ABC-3}
  y'=\frac{A}{x^3}y^2+\frac{B+x}{x^2}y+\frac{C}{x}
\end{equation}
invariant under the projective symmetry group generated by $\mathbf{v}=x^2\gen x+xy\gen y$ (see Example \ref{pr-tr-ex}).

In the case $\Delta<0$, the general solution of \eqref{Riccati-gen} from \eqref{gen-sol} has the explicit form
$$y=\frac{1}{2A\nu}\curl{\sqrt{-\Delta}\tan\left[\frac{\sqrt{-\Delta}}{2}(s+c)\right]-B},$$
in the case $\Delta=0$
$$y=-\frac{B(s+c)+2}{2A\nu(s+c)}, \quad s=-\frac{1}{x}.$$
\end{example}

\begin{remark}\label{rmk-Lie}
Lie showed that the Lie symmetry algebra of a second order ODE can be of dimension $\dim\lie=0,1,2,3,8$. The maximum dimension is attained if and only if and only if it can be mapped by a point transformation to the canonical linear equation $y''=0$, which admits a symmetry group isomorphic to the  $\SL(3,\mathbb{R})$ group, acting as the group of projective transformations of the Euclidean plane with coordinates $(x,y)$. Moreover, he classified equations with point symmetries into equivalence classes under the action of the infinite dimensional group $\Diff(2,\mathbb{C})$ of all local diffeomorphisms of a complex plane $\mathbb{C}$.
\end{remark}

\begin{remark}\label{Lie-canonical}
There are two isomorphism classes of two-dimensional Lie algebras exist (over $\mathbb{R}$ and over $\mathbb{C}$): abelian and nonabelian
\begin{equation}\label{isomorf-classes}
  [\mathbf{v}_1,\mathbf{v}_2]=0,  \qquad [\mathbf{v}_1,\mathbf{v}_2]=\mathbf{v}_1.
\end{equation}

Their realizations in the plane are summarized as
\begin{equation}\label{two-dim-cano-abelian}
    A_{2,1}: \quad \mathbf{v}_1=\gen x,  \quad \mathbf{v}_2=\gen y, \quad A_{2,2}: \quad \mathbf{v}_1=\gen y,  \quad \mathbf{v}_2=x\gen y,
\end{equation}
\begin{equation}\label{two-dim-cano-nonabelian}
    A_{2,3}: \quad \mathbf{v}_1=\gen y,  \quad \mathbf{v}_2=x\gen x+y\gen y, \quad A_{2,4}: \quad \mathbf{v}_1=\gen y,  \quad \mathbf{v}_2=y\gen y.
\end{equation}
They can be obtained by the following argument.

{\bf Abelian case}:
We first assume the algebra $\lie=\curl{\mathbf{v}_1,\mathbf{v}_2}$ has rank 1. We  start by assuming that $\mathbf{v}_1$ is in some canonical form,  say $\mathbf{v}_1=\gen y$ and $\mathbf{v}_2$ in its generic form
$$\mathbf{v}_2=\xi(x,y)\gen x+\eta(x,y)\gen y.$$
Then from the first commutation relation of \eqref{isomorf-classes} we have
$\mathbf{v}_2=\eta(x)\gen y.$ The transformation preserving $\mathbf{v}_1$, but changing $\mathbf{v}_2$ (equivalence group) is given by
\begin{equation}\label{equiv-group}
  \tilde{x}=f(x), \quad \tilde{y}=y+g(x), \quad f'\ne 0.
\end{equation}
Under such a transformation, we find $\tilde{\mathbf{v}}_2=\eta(x)\gen {\tilde{y}}$. We simply choose $f(x)=\eta(x)$ and obtain the $A_{2,2}$ algebra.

Now assume that the rank of $\lie$ is 2. Then we must have
$$\mathbf{v}_2=\xi(x)\gen x+\eta(x)\gen y,  \quad \xi(x)\ne 0.$$
Then $\mathbf{v}_2$ is transformed to
$$\tilde{\mathbf{v}}_2=\xi f'\gen {\tilde{x}}+(\xi g'+\eta)\gen {\tilde{y}}.$$ We choose $\xi f'=1$ and $\xi g'+\eta=0$ so that we have $\tilde{\mathbf{v}}_2=\gen {\tilde{x}}$, which gives $A_{2,1}$ (in the reverse order).

{\bf Non-abelian case}: First, $\rank\lie=1$ realization. In this case, from the second commutation relation we find
$\mathbf{v}_2=(y+\eta(x))\gen y$. The transformation \eqref{equiv-group} takes $\mathbf{v}_2$ to
$\tilde{\mathbf{v}}_2=(\tilde{y}+\eta-g)\gen {\tilde{y}}$. The choice $g=\eta(x)$ leads to the $A_{2,4}$ realization of the algebra.

Finally, we consider the $\rank\lie=2$ realization. In this case, we find
$$\mathbf{v}_2=\xi(x)\gen x+(y+\eta(x))\gen y, \quad \xi(x)\ne 0.$$ Again, using \eqref{equiv-group} with $f$ and $g$ satisfying
$$\xi f'=f,  \quad \xi g'-g+\eta=0$$ we arrive at $A_{2,3}$.

In the terminology of Lie, for the realizations $A_{2,2}$, $A_{2,4}$, the vector fields $\mathbf{v}_1$ and $\mathbf{v}_2$ are called linearly connected, for $A_{2,1}$, $A_{2,3}$, linearly unconnected.

The corresponding canonical invariant equations are
\begin{equation}\label{abelian-inv-eqs}
  y''=F(y'), \qquad y''=F(x)
\end{equation}
for $A_{2,1}$, and  $A_{2,2}$,  and
\begin{equation}\label{nonabelian-inv-eqs}
  x y''=F(y'),  \qquad  y''=F(x)y'
\end{equation}
for $A_{2,3}$, and  $A_{2,4}$, respectively.

Any second order ODE with a  symmetry group $G$ of dimension $\geq 2$ (a nonlinear second order ODE can be invariant at most under a three-dimensional Lie group) can be integrated by two quadratures except for the rotation group $\Ort(3,\mathbb{R})$ when $\dim G=3$, which has no two-dimensional subgroup. In the latter case, there is a method based on first integrals to find the general solution without integration, namely by purely algebraic manipulations (see \cite{Hydon2000, Stephani1989}).

\end{remark}

\begin{example}\label{method-first-int}
  We construct the most general second order ODE invariant under the rotation group $\Ort(3,\mathbb{R})$. The Lie algebra $\ort(3,\mathbb{R})$ acting on the surface of the unit  sphere $\mathbb{S}^2=\curl{(x,y): x\in [0,2\pi),y\in (0,\pi)}\subset \mathbb{R}^3$ (See Remark \ref{rank-2-realiz}) is realized by the vector fields
\begin{equation}\label{realiz-on-S2}
  \begin{aligned}
&\mathbf{v}_1=\gen x,\\
&\mathbf{v}_2=\cos x\cot y\gen x+\sin x\gen y,\\
&\mathbf{v}_3=-\sin x\cot y\gen x+\cos x\gen y
\end{aligned}
\end{equation}
with the commutation relations
\begin{equation}\label{comm-so3}
  [\mathbf{v}_1,\mathbf{v}_2]=\mathbf{v}_3, \quad [\mathbf{v}_2,\mathbf{v}_3]=\mathbf{v}_1, \quad [\mathbf{v}_3,\mathbf{v}_1]=\mathbf{v}_2.
\end{equation}
Obviously,  invariant function  will be of the form $I(y,y_1,y_2)$ because of $$\pr{2}\mathbf{v}_1(I)=\mathbf{v}_1(I)=0.$$
We solve the PDE $\pr{2}\mathbf{v}_2(I)=0$. Splitting it with respect to the set $\curl{\cos x, \sin x}$ we find
\begin{equation}\label{two-ode}
  \begin{split}
     & \tan y I_y+y_1 I_{y_1}-(\tan y+2\sec y \csc y y_1^2-y_2)I_{y_2}=0, \;\; \\
     & (1+\csc^2 y)I_{y_1}+y_1(\cot y-2\cot y \csc^2 y y_1^2+3\csc^2 y_2)I_{y_2}=0.
   \end{split}
\end{equation}
A first integral is immediately obtained as $\tau=(\csc y) y_1$. The second independent integral is found by integrating the linear ODE
$$\frac{dy_2}{dy}-2(\cot y) y_2=-(1+2\tau^2)$$ as
$$\omega=(\csc^2 y) y_2-(1+2\tau^2)\cot y.$$ Expressing the second PDE of \eqref{two-ode} in terms of $\tau, \omega$ gives
$$(1+\tau^2)I_{\tau}+3\tau \omega I_{\omega}=0$$ possessing the single  first integral $\zeta=\omega(1+\tau^2)^{-3/2}$. The condition  $\pr{2}\mathbf{v}_3(\zeta)=0$ is automatically satisfied because
$$\pr{2}\mathbf{v}_3=\pr{2}[\mathbf{v}_1,\mathbf{v}_2]=[\pr{2}\mathbf{v}_1,\pr{2}\mathbf{v}_2].$$
Finally, a single invariant equation depending on an arbitrary parameter $\nu$ is found from the relation $\zeta(y,y_1,y_2)=\nu=\text{constant}$. In terms of the spherical angle variables (azimuth and polar angles) $(x,y)$ it can be written in the form
\begin{equation}\label{so3-inv-ode}
  y_2-2(\cot y) y_1^2-\sin y \cos y=\nu \sin^2 y \left(1+(\csc^2 y) y_1^2\right)^{3/2}.
\end{equation}
The special case $\nu=0$ is the equation of the geodesics on the unit sphere
\begin{equation}\label{geodesic-ode}
  y_2-2(\cot y) y_1^2-\sin y \cos y=0.
\end{equation}
In order to solve this equation, we can use the algorithmic method of first integrals based on symmetries suggested in \cite{Hydon2000} (see \cite{Stephani1989} for an alternative integration strategy). We first need the characteristic functions of the vector fields $\mathbf{v}_1,\mathbf{v}_2,\mathbf{v}_3$
\begin{equation}\label{char-so3}
  \begin{split}
     & Q_1=-y_1,  \quad Q_2=\sin x -\cos x (\cot y) y_1,\ \\
      & Q_3=\cos x +\sin x (\cot y) y_1.
  \end{split}
\end{equation}
We next compute the quantities
\begin{equation}\label{Wij}
  \begin{split}
     & W_{12}=Q_1 \widehat{D}_x Q_2- Q_2 \widehat{D}_x Q_1,\\
      & W_{13}=Q_1 \widehat{D}_x Q_3- Q_3 \widehat{D}_x Q_1,\\
      & W_{23}=Q_2 \widehat{D}_x Q_3- Q_3 \widehat{D}_x Q_2,
  \end{split}
\end{equation}
where $\widehat{D}_x$ is the total derivative operator restricted to the geodesic equation
$y_2=F(y,y_1)=2(\cot y) y_1^2+\sin y \cos y$
$$\widehat{D}_x=\gen x+y_1\gen y+F\gen {y_2}.$$
They are explicitly given by
\begin{equation}\label{Wij-2}
\begin{split}
     & W_{12}=-(\cos x \csc^2 y) y_1^3+(\sin x\cot y) y_1^2-\cos x y_1+\sin x \sin y \cos y,\\
      & W_{13}=(\sin x \csc^2 y) y_1^3+(\cos x\cot y) y_1^2+\sin x y_1+\cos x \sin y \cos y,\\
      & W_{23}=\frac{1}{2}(1-\cos 2y+2y_1^2).
  \end{split}
\end{equation}
The two nonconstant functionally independent first integrals  then arise as the ratios
\begin{equation}\label{first-int}
  \frac{W_{12}}{W_{23}}=c_1,  \qquad  \frac{W_{13}}{W_{23}}=c_1
\end{equation}
because by construction
$$\widehat{D}_x\left(\frac{W_{12}}{W_{23}}\right)=0,  \quad \widehat{D}_x\left(\frac{W_{13}}{W_{23}}\right)=0.$$
Eqs. \eqref{first-int} define the solution parametrically. It is possible to eliminate the parameter $y_1$ between them as
\begin{equation}\label{great-circles}
  c_1 \sin x+c_2 \cos x+\cot y=0
\end{equation}
or in terms of new parameters $\alpha, \beta$ as
\begin{equation}\label{great-circle}
  \cot y=\alpha \cos(x+\beta).
\end{equation}
This is the equation of  great circle (intersection of the plane passing through the origin with the unit sphere).

On the other hand, if $\nu\neq 0$, then the solution can only be obtained in parametric form.

We note that the geodesic equation \eqref{geodesic-ode} is the Euler-Lagrange equation of the (first-order) arc-length functional on $\mathbb{S}^2$
\begin{equation}\label{arc-length}
  J[y]=\int_{x_1}^{x_2}L(y,y_1)dx
\end{equation}
with  Lagrangian
$$L(y,y_1)=\sqrt{y_1^2+\sin^2 y}.$$
It is easy to verify that the vector fields defined by \eqref{realiz-on-S2} are also variational symmetries of the functional \eqref{arc-length}. This is indeed the case because infinitesimal invariance condition
\begin{equation}\label{var-invar}
  \pr{1}\mathbf{v}_i(L)+{D}_x (\xi_i)L=0, \quad i=1,2,3
\end{equation}
is satisfied for each vector field $\mathbf{v}_i$ with $\xi_i$ being the $x$-coefficients of the vector fields.
The Noether theorem then can be used to find three first integrals $\mathscr{I}_1, \mathscr{I}_2, \mathscr{I}_3$ as
$$\mathscr{I}_i=\xi_i L+Q_i\frac{\partial L}{\partial y_1}=c_i.$$
For an excellent discussion of this theorem with interesting applications   the reader can consult the book \cite{Olver1993}.
A simple computation gives the first integrals
\begin{equation}\label{noether-first-int}
  \begin{split}
     & \mathscr{I}_1=L^{-1}\sin^2 y,  \quad  \mathscr{I}_2=L^{-1}(\sin x y_1+\cos x\sin y\cos y),\\
      & \mathscr{I}_3=L^{-1}(\cos x y_1-\sin x\sin y\cos y).
  \end{split}
\end{equation}
The first one is the conservation of energy in which $\mathscr{I}_1=-H=L-y_1 L_{y_1}$ is the Hamiltonian.
Of course, only two of them are functionally independent. Indeed, we have the relation
\begin{equation}\label{geodesics}
  c_2 \cos x-c_3 \sin x-c_1 \cot y=0,
\end{equation}
which is in fact the general solution (without any integration) with the parameter $y_1$ eliminated.
If $c_1\ne 0$, we can put $c_1=1$ by dividing by $c_1$ and  redefining $c_2$ and $c_3$.

Alternatively,  the first integral $\mathscr{I}_1$, which is invariant under $\gen x$,
$$c \sin^4 y=y_1^2+\sin^2 y,  \quad c=c_1^{-2}$$ can be written as
$$[(\cot y)']^2+(\cot y)^2=c-1\equiv \alpha,$$
which integrates to the solution \eqref{great-circle}.

\begin{remark}\label{rank-2-realiz}
The algebra  spanned by  \eqref{realiz-on-S2} is the  two dimensional rank-two realization of $\ort(3,\mathbb{R})$ that  can be obtained by the following procedure.  We may start with the locally smooth canonical form $\mathbf{v}_1=\gen x$. Then with $\mathbf{v}_2$, $\mathbf{v}_3$ in the generic form
$$\mathbf{v}_2=a(x,y)\gen x+b(x,y)\gen y,  \quad \mathbf{v}_3=A(x,y)\gen x+B(x,y)\gen y$$
the commutation relations \eqref{comm-so3} give
$$a_x=A,  \quad b_x=B, \quad a=-A_x,  \quad b=-B_x$$ with solutions
$$a=\alpha(y)\cos(x+\delta(y)),  \quad b=\beta(y)\sin(x+\lambda(y)),$$
$$A=-\alpha(y)\sin(x+\delta(y)),  \quad B=\beta(y)\cos(x+\lambda(y)).$$
We must have $a,b,A,B\ne 0$ for a rank-two realization because all wedge products (bivectors) $\mathbf{v}_1\wedge \mathbf{v}_2$,  $\mathbf{v}_1\wedge \mathbf{v}_3$, $\mathbf{v}_2\wedge \mathbf{v}_3$ must be nonzero. Now the invertible transformation preserving the canonical form $\mathbf{v}_1$ is $\tilde{x}=x+X(y)$, $\tilde{y}=Y(y)$, $Y'\ne 0$ and maps $\mathbf{v}_2$ to
$$\tilde{\mathbf{v}}_2=[\alpha \cos(\tilde{x}-X+\delta)+\beta X' \sin(\tilde{x}-X+\lambda)]\gen {\tilde{x}}+\beta Y' \sin(\tilde{x}-X+\lambda)\gen {\tilde{y}}.$$
We choose $X$ and $Y$ so that $\beta Y'=1$ and $X=\lambda$. We may write $\tilde{\mathbf{v}}_2$ (without tildes) as
$$\mathbf{v}_2=\gamma(y)\cos(x+\sigma(y))\gen x+\sin x\gen y,  \quad \gamma\ne 0,$$
where $\gamma(y)$ and $\sigma(y)$ are new arbitrary smooth functions.  Then we have
$$\mathbf{v}_3=-\gamma(y)\sin(x+\sigma(y))\gen x+\cos x\gen y.$$
The commutation relation $[\mathbf{v}_2,\mathbf{v}_3]=\mathbf{v}_1$ then gives
$$\gamma(\sin \sigma)=0,  \quad -(1+\gamma^2)-\gamma'(\cos \sigma)+\gamma' \sigma'(\sin \sigma)=0.$$ So  we have $\sin \sigma=0$ or $\cos \sigma=\pm 1$. For $\cos \sigma=1$, $\gamma'=-(1+\gamma^2)$ with solution $\gamma=\cot(y+\chi)$ and $\cos(x+\sigma(y))=\cos x$. The translation $\tilde{y}=y+\chi$ removes $\chi$ in $\gamma$ and consequently we find
$$\mathbf{v}_2=\cot y \cos x\gen x+\sin x \gen y,  \quad \mathbf{v}_3=-\cot y \sin x\gen x+\cos x \gen y.$$
For $\cos \sigma=-1$, we have $\gamma'=(1+\gamma^2)$ with solution $\gamma=-\cot(y+\chi)$ and $\cos(x+\sigma(y))=-\cos x$ leading to the earlier realization.
This concludes the rank-two realization of $\ort(3,\mathbb{R})$.

We observe that the vector fields $\mathbf{v}_1, \mathbf{v}_2, \mathbf{v}_3$ are invariant under the transformation $\tilde{x}=x+2n\pi$, $\tilde{y}=y+n\pi$, $n\in \mathbb{Z}$ so we can restrict the coordinates $(x,y)$ onto the unit sphere.

As a result, we can state that any second order ordinary differential equation admitting a rank-two realization of $\ort(3,\mathbb{R})$ algebra is equivalent to \eqref{so3-inv-ode} by point transformations.

We comment that there is no rank-one realization of $\ort(3,\mathbb{R})$ on the $(x,y)$-space. If it were true, then there would exist real smooth functions $f(x,y)$ and $g(x,y)$ such that
$\mathbf{v}_1=\mathbf{w}$, $\mathbf{v}_2=f\mathbf{w}$, $\mathbf{v}_3=g\mathbf{w}$, where $\mathbf{w}$ is a smooth vector field. From the commutation relations we find that $f$ and $g$ must satisfy
$$f=-\mathbf{w}(g), \quad g=\mathbf{w}(f),  \quad f\mathbf{w}(g)-g\mathbf{w}(f)=1.$$ They imply the relation $f^2+g^2+1=0$, which has no solution for real functions $f, g$.

\end{remark}

\begin{remark}
All locally primitive Lie algebras of vector fields in $\mathbb{R}^2$
isomorphic to $\ort(3,\mathbb{R})$ are equivalent to the Lie algebra spanned by the vector
fields
\begin{equation}\label{realiz-on-R2}
  \begin{aligned}
&\mathbf{v}_1=-y\gen x+x \gen y,\\
&\mathbf{v}_2=(1+x^2-y^2)\gen x+2xy\gen y,\\
&\mathbf{v}_3=2xy\gen x+(1+y^2-x^2)\gen y
\end{aligned}
\end{equation}
under a local change of coordinates \cite{Gonzalez-LopezKamranOlver1990}. They satisfy  the commutation relations
$$[\mathbf{v}_1,\mathbf{v}_2]=-\mathbf{v}_3,  \quad [\mathbf{v}_2,\mathbf{v}_3]=-4\mathbf{v}_1,  \quad [\mathbf{v}_3,\mathbf{v}_1]=-\mathbf{v}_2.$$
The scaled vector fields
$$\tilde{\mathbf{v}}_1=-\mathbf{v}_1,  \quad \tilde{\mathbf{v}}_2=\frac{1}{2}\mathbf{v}_2, \quad \tilde{\mathbf{v}}_3=\frac{1}{2}\mathbf{v}_3$$ satisfy the standard $\ort(3,\mathbb{R})$ algebra commutation relations \eqref{comm-so3}.
The
Lie algebra $\ort(3,\mathbb{R})$ acts on the unit sphere  by infinitesimal rotations
\begin{equation}
\mathbf{w}_{1}=x \partial_{y}-y \partial_{x}, \quad
\mathbf{w}_{2}=z \partial_{x}-x \partial_{z}, \quad
\mathbf{w}_{3}=y \partial_{z}-z \partial_{y}.
\end{equation}
The vector fields $\mathbf{v}_1, \mathbf{v}_2, \mathbf{v}_3$ are just the images of $\mathbf{w}_1, \mathbf{w}_2, \mathbf{w}_3$ under the standard
stereographic projection from the north pole $(0, 0,1)$ of the unit sphere $x^2+y^2+z^2=1$ on the $(X,Y)$-plane
\begin{equation}\label{stereo-proj}
\pi:  \quad (X, Y)=\left(\frac{x}{1-z}, \frac{y}{1-z}\right)
\end{equation}
with inverse
\begin{equation}\label{inv-stereo-proj}
\pi^{-1}: \quad (x, y, z)=\left(\frac{2 X}{1+X^{2}+Y^{2}}, \frac{2 Y}{1+X^{2}+Y^{2}}, \frac{X^{2}+Y^{2}-1}{1+X^{2}+Y^{2}}\right).
\end{equation}
This is seen easily forming the following images and using the inverse transformation \eqref{inv-stereo-proj}
$$X_*(\mathbf{w}_1)=-Y,   \quad Y_*(\mathbf{w}_1)=X,$$
$$X_*(\mathbf{w}_2)=-\frac{x^2+z(z-1)}{(1-z)^2}=-\frac{1}{2}(1+X^2-Y^2),   \quad Y_*(\mathbf{w}_2)=-\frac{xy}{(1-z)^2}=-XY,$$
$$X_*(\mathbf{w}_3)=XY,  \quad Y_*(\mathbf{w}_3)=\frac{y^2+z(z-1)}{(1-z)^2}=\frac{1}{2}(1-X^2+Y^2).$$

One can of course construct a second order differential equation invariant under the algebra \eqref{realiz-on-R2} whose solution can be found using the method of Example \ref{method-first-int}. Second order elementary invariants of the rotational vector field $\mathbf{v}_1$ ($\ort(2,\mathbb{R})$-invariants) are known as
$$J_1=x^2+y^2,  \quad J_2=\frac{x y_1-y}{x+y y_1},  \quad J_3=\kappa=(1+y_1^2)^{-3/2}y_2.$$
All invariants of the whole algebra should be functions of these three invariants. It is more convenient to switch to a new basis of  invariants
$$I_1=J_1,  \quad I_2=I_1^{-1}(1+J_2^{-2})=(xy_1-y)^{-2}(1+y_1^2),  \quad I_3=J_3(1+I_1).$$
Then we can write the second PDE $\pr{2}\mathbf{v}_2(I)=0$ in terms of $I_2$ and $I_3$
$$2y_1(y-xy_1)^{-1}(1+I_1)[I_2 \gen{I_2}+I_2^{-1/2}\gen{I_3}]I=0.$$ So we find a family of  invariant equations $I_4=I_3+2I_2^{-1/2}=c$, depending on a single parameter, or in terms of $(x,y)$ coordinates
\begin{equation}\label{inv-so3-2}
  y_2=(1+y_1^2)(1+x^2+y^2)^{-1}[2(xy_1-y)+c(1+y_1^2)^{1/2}].
\end{equation}
Note that $\pr{2}\mathbf{v}_3(I_4)=0$ is again satisfied automatically.
See \cite{Hydon2000} for a derivation of its general solution in non-parametric form when $c=0$ as the family of circles
$$(x-c_1)^2+(y-c_2)^2=1+c_1^2+c_2^2.$$
\end{remark}

\end{example}

The following Theorem will be useful when performing reductions.

\begin{theorem}[\cite{Olver1995}]\label{inherited}
  Suppose $\mathsf{E}(x,y^{(n)})=0$ is an $n$-th order ODE with a symmetry group $G$. Let $H$ be a one-parameter subgroup of $G$. Then the  ODE reduced by $H$, $\mathsf{E}/H$, admits the quotient group  $\Nor_G(H)/H$, where $\Nor_G(H)=\curl{g\in G:g.H.g^{-1}\subset H}$ is the normalizer subgroup of $H$ in $G$, as a symmetry group (often called inherited symmetry group of $\mathsf{E}/H$).
\end{theorem}
The normalizer algebra in $\lie$ of the subalgebra $\mathfrak{h}\subset \lie$ is the maximal subalgebra satisfying
\begin{equation}\label{normalizer}
  \nor_\lie\mathfrak{h}=\curl{\mathbf{v}\in \lie:[\mathbf{v},\mathfrak{h}]\subset\mathfrak{h}}.
\end{equation}
\begin{remark}
The normalizer of a subalgebra $\mathfrak{h}\subset \lie$ in the Lie algebra $\lie$ of  $\dim\lie=r$ with a basis $\curl{\mathbf{v}_1,\ldots,\mathbf{v}_r}$ is easily obtained by solving a linear algebra problem. Let the subalgebra $\mathfrak{h}$  of  $\dim\mathfrak{h}=k$ be spanned by $\curl{\mathbf{v}_{1},\ldots \mathbf{v}_{k}}$. Now for $\mathbf{v}=\sum_{j=1}^r a_j\mathbf{v}_j$, $a_j\in \mathbb{R}$ and $\mathbf{v}_{\alpha}\in \mathfrak{h}$, $1\leq \alpha \leq k$, we  impose the requirement \eqref{normalizer}
$$[\mathbf{v},\mathbf{v}_{\alpha}]=[\sum_{j=1}^r a_j\mathbf{v}_{j},\mathbf{v}_{\alpha} ]=\sum_{j=1}^r[\mathbf{v}_{j},\mathbf{v}_{\alpha}]a_j$$
$$=\sum_{i=1}^r\sum_{j=1}^rC_{j\alpha}^i a_j \mathbf{v}_i=\sum_{i=1}^k\lambda_{\alpha,i}\mathbf{v}_i$$ for some constants $\lambda_{\alpha,i}$.  Here $C_{j\alpha}^i$ are the structure constants of $\lie$. This implies that the following set of linear algebraic equations need to be solved for the coefficients $a_j$, $1\leq j\leq r$
$$\sum_{j=1}^rC_{j\alpha}^i a_j=\lambda_{\alpha,i}, \quad i=1,\ldots,k, \quad \alpha=1,\ldots, k,$$
$$\sum_{j=1}^rC_{j\alpha}^i a_j=0, \quad i=k+1,\ldots, r.$$
\end{remark}
\begin{remark}
If $\mathfrak{h}$ is already an ideal of $\lie$, then $\nor_\lie\mathfrak{h}=\lie$. If $\nor_\lie\mathfrak{h}=\mathfrak{h}$, then $\mathfrak{h}$ is called self-normalizing.
\end{remark}
Theorem \ref{inherited} infinitesimally states that if the Lie symmetry algebra $\lie$ of G has a basis $\mathbf{v}_1,\ldots, \mathbf{v}_r$. Then the ODE reduced by $\mathbf{v}_1$ can be reduced one more if $\mathbf{\hat{v}}_2=\Span\curl{\mathbf{v}_2,\ldots, \mathbf{v}_r}$ is chosen to satisfy  $[\mathbf{\hat{v}}_2,\mathbf{v}_1]=k \mathbf{v}_1$ for some real constant $k$, meaning that $\mathbf{v}_1$ is an ideal (normal subalgebra) of $\mathbf{\hat{v}}_2$. One can reiterate this process to achieve a full reduction.

Theorem \ref{inherited} applied to a two-parameter symmetry group with the Lie algebra $\lie$ being one of the isomorphy classes of  Remark \ref{rmk-Lie} and satisfying the commutation relation $[\mathbf{v}_1,\mathbf{v}_2]=k \mathbf{v}_1$ ensures that given a second order ODE invariant under $\lie$, reducing its order by one by the ideal (normal subalgebra) $\mathbf{v}_1$ in $\lie$ will lead to a first order ODE inheriting the one-parameter subgroup generated by  $\mathbf{v}_2$ as a symmetry group. Note that the normalizer of $\curl{\mathbf{v}_1}$ is $\curl{\mathbf{v}_1,\mathbf{v}_2}$ and $\mathbf{v}_2$ belongs to $\nor_\lie\mathfrak{h}/\mathfrak{h}$. This reduction procedure  makes possible the integration of the ODE by two successive quadratures.

If the reduction is performed in the reverse order, the reduced ODE will in general not inherit the symmetry of the original equation, so we may not be able to complete the full integration.

\begin{example}
  The second order invariant equation with the same symmetry $\mathbf{v}_1=\xi\gen x+\eta y \gen y$ of Example \eqref{fiber-preser} can be expressed in terms of the second order invariants; $r,w$ as defined in \eqref{inv} and
\begin{equation}\label{2nd-inv}
  \zeta(x,y,y',y'')=w\frac{dw}{dr}=\nu(x)[\xi^2 y''+\xi(\xi'-2\eta)y'+(\eta^2-\xi\eta')y]
\end{equation}
as
\begin{equation}\label{2ndODE}
  y''+p(x)y'+q(x)y=\frac{H(r,w)}{\nu\xi^2},
\end{equation}
where
$$p(x)=\xi^{-1}(\xi'-2\eta),  \quad q(x)=\xi^{-2}(\eta^2-\xi\eta'), \quad \xi\ne 0.$$ Elimination of $\eta$ gives the relation
\begin{equation}\label{xi-inv-rel}
 \frac{1}{4}(2\xi\xi''-\xi'^2)+\xi^2 I(x)=0, \quad I(x)=q(x)-\frac{1}{2}p'(x)-\frac{1}{4}p(x)^2.
\end{equation}
The particular choice $\eta=\xi'/2$ leads to the simpler equation
\begin{equation}\label{2ndODE-simpler}
  y''+q(x)y=\xi^{-3/2}H(r,w), \quad q(x)=-\frac{(\sqrt{\xi})''}{\sqrt{\xi}},
\end{equation}
where the invariant variables $r, w$ have the form
$$r=\frac{y}{\sqrt{\xi}},  \quad w=\frac{1}{\sqrt{\xi}}\left(\xi y'-\frac{\xi'}{2}y\right).$$

If $\xi=x^2$ is chosen, $\mathbf{v}_1$  generates an inversional group and this equation simplifies to
\begin{equation}\label{inv-group-inv-ode}
 y''=x^{-3}H(r,w), \quad r=\frac{y}{x}, \quad w=xy'-y.
\end{equation}
If we additionally ask the equation to be invariant under the scaling $(x,y)\to (\lambda x,\lambda^{\alpha}y)$, $\lambda>0$ generated by $\mathbf{v}_2=x\gen x+\alpha y\gen y$, the following condition on $H$ should be imposed
$$(\alpha-1)rH_r+\alpha w H_w=(\alpha+1)H.$$
Thus, if $\alpha\ne 1$, $H$   is restricted to
$$H=r^{(\alpha+1)/(\alpha-1)}\hat{H}(\sigma),  \quad \sigma=r^{\alpha/(1-\alpha)}w, $$ and to $H=w^2 \hat{H}(r)$ if $\alpha=1$ (see also Example \ref{2nd-ODE-build}).

Choosing $\hat{H}=K=\mbox{const.}$, $\alpha=(n+1)/(n-1)$, $n\ne 1$ we obtain the special form of the famous Emden--Fowler equation
\begin{equation}\label{EF-eq}
  y''=K x^{-(n+3)}y^n, \quad n\ne 0,1,  \quad K\ne 0
\end{equation}
with a two-parameter symmetry group generated by the two-dimensional Lie algebra $\lie$ of type $A_{2,3}$
$$\mathbf{v}_1=x^2\gen x+xy\gen y, \quad \mathbf{v}_2=x\gen x+\frac{n+1}{n-1}y\gen y, \quad [\mathbf{v}_1,\mathbf{v}_2]=-\mathbf{v}_1.$$
Theorem \ref{inherited} guarantees that
integration is completed using  two quadratures by reductions in the order of $\mathbf{v}_1$ (an ideal of $\lie$) first and then $\mathbf{v}_2$. We remark that apart from the special case
\begin{equation}\label{minus5}
 y''=K x^{-5}y^2,
\end{equation}
which is obtained for $n=2$ from \eqref{EF-eq} \label{n=2}, there are only two other values of the exponent $m$ in  the equation $y''=K x^{m}y^2$ for which the symmetry algebra is a two-dimensional (rank-two) nonabelian one. They are $m=-15/7$ and $ m=-20/7$ (see \cite{Stephani1989} for classification details). For all values $m$ with $m\notin \curl{-5, -15/7, -20/7 }$, the symmetry algebra is one-dimensional and generated by the scaling symmetry $\mathbf{v}=x\gen x-(m+2)y\gen y$.

Eq. \eqref{minus5} is invariant under the two-dimensional symmetry algebra
\begin{equation}\label{vf-minus5}
  \mathbf{v}_1=x^2\gen x+xy\gen y, \quad \mathbf{v}_2=x\gen x+3y\gen y.
\end{equation}
We can integrate \eqref{minus5} using differential invariants $I=y/x$ and $J=x y'-y$ of $ \mathbf{v}_1$. The reduced first order equation is obtained by invariant differentiation
\begin{equation}\label{red-minus5}
  \frac{dJ}{dI}=\frac{x^3 y''}{x y'-y}=K\frac{x^{-2}y^2}{x y'-y}=K\frac{I^2}{J},
\end{equation}
whose integration gives
$$ J^2-\frac{2K}{3}I^3=c_1$$ or
\begin{equation}\label{red-first-minus5}
  (xy'-y)^2-\frac{2K}{3}x^{-3}y^3=c_1.
\end{equation}
Note that this equation is invariant under the scaling group $(x,y)\to (\lambda x,\lambda^{3}y)$, $\lambda>0$ corresponding to $\mathbf{v}_2$. In terms of the normal coordinates $r=y/x$, $s=-1/x$ of $\mathbf{v}_1$ it has the form
$$r'(s)^2=\frac{2K}{3}r^3+c_1.$$
For $c_1\ne 0$, the general solution can be expressed in terms of elliptic functions. It is elementary, when $c_1=0$, which is given by $r=(1/(\sqrt{6})s+c_2)^{-2}$ and in terms of $(x,y)$ variables
\begin{equation}\label{elem-sol}
  y(x)=\frac{6x^3}{(\sqrt{K}+\sqrt{6}c_2 x)^2}.
\end{equation}

Moreover, the particular case $n=-3$ of \eqref{EF-eq} gives the celebrated Ermakov--Pinney equation
\begin{equation}\label{specialEP}
  y''=Ky^{-3}
\end{equation}
with solution
$$y=\pm \sqrt{A+2Bx+Cx^2}, \quad AC-B^2=K.$$
This equation admits the $\Sl(2,\mathbb{R})$ algebra as the symmetry algebra with the basis
\begin{equation}\label{sl2}
  \mathbf{w}_1=\gen x, \quad \mathbf{w}_2=x\gen x+\frac{1}{2}y\gen y, \quad \mathbf{w}_3= x^2\gen x+xy\gen y
\end{equation}
satisfying the commutation relations
\begin{equation}\label{sl2-comm}
 [\mathbf{w}_1,\mathbf{w}_2]=\mathbf{w}_1,  \quad  [\mathbf{w}_1,\mathbf{w}_3]=2\mathbf{w}_2, \quad [\mathbf{w}_2,\mathbf{w}_3]=\mathbf{w}_3.
\end{equation}

Just as in the previous case, the choice $\xi=x^{2-k}$, $\eta=(1-k)x^{1-k}$ in $\mathbf{v}_1$ of \eqref{2ndODE},  combined with the scaling group generated by $$\mathbf{v}_2=x\gen x+\frac{\delta+2}{1-m}y\gen y,  \quad m\ne 1$$
can be shown to produce the following integrable  variant of the generalized Lane--Emden--Fowler equation
\begin{equation}\label{EF-eq-var}
  y''+\frac{k}{x}y'=K x^{\delta}y^{m},  \quad m\ne 0,1, \quad K\ne 0,
\end{equation}
if the condition $3+\delta+m-k(m+1)=0$ is satisfied.

The case when $\delta=0$ ($m=(k-3)/(1-k)$) appeared in \cite{Gungor2002} as one of the reductions of a radially symmetric nonlinear porous-medium equation. A suitable basis of the algebra in the case $k\ne 1,2$ is
$$\tilde{\mathbf{v}}_1=(1-k)\mathbf{v}_1,  \quad \tilde{\mathbf{v}}_2=(1-k)^{-2}\mathbf{v}_2=\frac{1}{k-1}x\gen x+\frac{1}{k-2}y\gen y,  $$ with the commutation relation $[\tilde{\mathbf{v}}_1,\tilde{\mathbf{v}}_2]=\tilde{\mathbf{v}}_1$ (type $A_{2,3}$). In terms of the canonical (or normal) coordinates
$$r=\frac{x^{k-2}y^{(k-2)/(k-1)}}{k-1},  \quad s=\frac{x^{k-1}}{k-1},$$ we find the standard form of the corresponding equation
$$rs''(r)=-K(k-2)s'(r)^3+(k-1)^3 s'(r),$$ being invariant under the algebra $\curl{\gen s, r\gen r+s\gen s}$.
On solving by two quadratures, implicit solution of the original equation  is obtained.

Another particular case where $m=-3$ and $\delta=-2k$ leads to an $\Sl(2,\mathbb{R})$ invariant  equation. Symmetry vector fields for $k\ne 1$ are
$$\mathbf{v}_1=x^{2-k}\gen x+(1-k)x^{1-k}y\gen y, \quad \mathbf{v}_2=x\gen x+\frac{1-k}{2}y\gen y, \quad \mathbf{v}_3=x^k\gen x,$$ with the commutation relations
$$[\mathbf{v}_1,\mathbf{v}_2]=(k-1)\mathbf{v}_2, \quad [\mathbf{v}_2,\mathbf{v}_3]=(k-1)\mathbf{v}_3,  \quad [\mathbf{v}_1,\mathbf{v}_3]=2(k-1)\mathbf{v}_2.$$
If $k=1$, they are given by
$$\mathbf{v}_1=2x \ln x\gen x+y\gen y, \quad \mathbf{v}_2=x (\ln x)^2\gen x+y\ln x\gen y,  \quad \mathbf{v}_3=x\gen x$$ and satisfy the commutation relations
$$[\mathbf{v}_1,\mathbf{v}_2]=2\mathbf{v}_2,  \quad [\mathbf{v}_2,\mathbf{v}_3]=-\mathbf{v}_1,  \quad [\mathbf{v}_1,\mathbf{v}_3]=-2\mathbf{v}_3.$$
This particular equation reduces to the standard Ermakov--Pinney equation $y''(t)=K y^{-3}(t)$ by  change of the independent variable, $t=x^{1-k}/(1-k)$, $k\ne 1$, and $t=\ln x$ for $k=1$.

We can extract a similar integrable class from \eqref{2ndODE} by imposing  invariance under a two-parameter symmetry group extended by $\mathbf{v}_2=\gen x$. To do this we require that the commutation relation $[\mathbf{v}_2,\mathbf{v}_1]=\mu \mathbf{v}_1$, $\mu\ne 0$  be satisfied, which implies that $\xi=e^{\mu x}$, $\eta=\alpha e^{\mu x}$ for some constant $\alpha$ and  that Eq. \eqref{2ndODE} is autonomous. For this choice of $\xi$, $\eta$ the left hand side of \eqref{2ndODE} is already independent of $x$.  We also require the right hand side to have the same property. This produces the PDE
\begin{equation}\label{x-indep-rhs}
 y''+p y'+q y=y^m G(\omega), \quad \omega=y^{-(m+1)/2}(y'-\alpha y),  \quad \mu=\frac{\alpha(1-m)}{2}.
\end{equation}
The condition \eqref{xi-inv-rel} is identically satisfied.
The special choice $G=K=\mbox{constant}$ and  $m\ne 1$ ($\mu\ne 0$)   gives us the integrable equation (a type of Emden--Fowler equation known as force-free generalized Duffing oscillator)
\begin{equation}\label{integ-2ndODE}
  y''+p y'+q y=K y^m, \quad p=-\frac{\alpha(m+3)}{2},  \quad q=\frac{\alpha^2(m+1)}{2}
\end{equation}
with symmetry algebra generated by
\begin{equation}\label{Duffing-algebra}
  \mathbf{v}_1=e^{\mu x}(\gen x+\alpha y\gen y),  \quad \mathbf{v}_2=\gen x,  \quad \alpha=\frac{2\mu}{1-m}.
\end{equation}
If $\alpha$ is eliminated between the coefficients $p$ and $q$ we find the integrability condition
\begin{equation}\label{int-cond}
  q=\frac{2(m+1)}{(m+3)^2}p^2,  \quad m\ne -3.
\end{equation}
Under this condition, \eqref{integ-2ndODE} passes the Painlev\'e test.

In terms of the canonical coordinates of $\mathbf{v}_1$
$$\tilde{x}=\mu^{-1}e^{-\mu x},  \qquad \tilde{y}=ye^{-\alpha x}, \quad \mu=\frac{m-1}{m+3}p, \quad \alpha=-\frac{2p}{m+3},$$
Eq. \eqref{integ-2ndODE} is reduced to $d^2\tilde{y}/d\tilde{x}^2=K\tilde{y}^m$, which is invariant under the sym\-met\-ry group generated by
$$\tilde{\mathbf{v}}_1=\gen {\tilde{x}},  \quad \tilde{\mathbf{v}}_2=\tilde{x}\gen {\tilde{x}}+\frac{2}{1-m}\tilde{y}\gen {\tilde{y}},$$
and can be integrated by two quadratures.

The travelling wave solutions of Fisher's (also called Kolmogorov--
Petrovsky--Piscunov) equation satisfy the ODE \cite{AblowitzZeppetella1979}
\begin{equation}\label{wave-sol}
  y''+cy'+y(1-y)=0.
\end{equation}

If we identify \eqref{wave-sol} with \eqref{integ-2ndODE} we find $p=c$, $q=1$ and $m=2$, $K=1$ and the integrability condition imposes a constraint    $c=\pm 5/\sqrt{6}$ on the wave speed. \label{special-wave-speed} The corresponding solution  is given in terms of Weierstrass function by $\tilde{y}=\wp(1/\sqrt{6}(\tilde{x}+a);0,g_3)$ (see page \pageref{p-sol}).

The same type of solutions for the Newell--Whitehead--Segel equation satisfy
\begin{equation}\label{Newell–Whitehead–Segel}
   y''+cy'+y(1-y^2)=0.
\end{equation}
The integrability condition  \eqref{int-cond} for $p=c$, $q=1$ and $m=3$  then requires
$c=\pm 3/\sqrt{2}$. Solutions are found in terms of Jacobi elliptic functions.

Surprisingly, the special case $m=3$ of \eqref{integ-2ndODE} turns up in seeking localized
stationary solutions of the form $u(x,t)=e^{-i\lambda t}u(x)$ of the one-dimensional nonlinear Schr\"odinger equation with inhomogeneous nonlinearity
$$iu_t+u_{xx}=V(x)u+g(x)|u|^2u, \quad x\in \mathbb{R},$$ where $V(x)$ is an external potential and $g(x)$ describes the spatial modulation of the nonlinearity (see for example \cite{Belmonte-BeitiaPerez-GarciaVekslerchikTorres2008}). For some special choice of pairs $(V,g)$, dictated by the presence of a Lie point symmetry,   $u(x)$ satisfies
$$u''+2C u'+Eu=g_0 u^3.$$ This equation is integrable if $E=(8/9) C^2$.

In Eqs. \eqref{wave-sol} and \eqref{Newell–Whitehead–Segel}, there is no loss of generality in assuming $c>0$, because using discrete transformation $x\to -x$, we can put $c\to -c$.

The case $m=3$ ($\mu=-\alpha\ne 0$ arbitrary) of \eqref{integ-2ndODE} is also known as the usual Duffing oscillator \label{Duffing} and under the condition $q=(2/9)p^2$ its exact solutions can be found in terms of Jacobi elliptic functions from integrating the first integral (energy) $$I=\frac{1}{2}\tilde{y}'^2-\frac{K}{4}\tilde{y}^4=\frac{1}{4}e^{-4\alpha x}\left[2y'^2-4\alpha y y'+2\alpha^2 y^2-Ky^4\right].$$

On the other hand, the case $m=-3$ ($p=0$, $q=-\alpha^2$, $\mu=2\alpha$) is recognized to be the celebrated Ermakov--Pinney equation
\begin{equation}\label{EP}
  y''-\alpha^2 y=K y^{-3},  \quad \alpha\ne 0.
\end{equation}
Its symmetry algebra \eqref{Duffing-algebra} is  extended by one additional  element
$$\mathbf{v}_3=e^{-2\alpha x}(\gen x-\alpha y\gen y),$$ or in coordinates $(\tilde{x},\tilde{y})$ (up to multiple of $-4\alpha^2$) by the projective element
$$\tilde{\mathbf{v}}_3=\tilde{x}^2\gen {\tilde{x}}+\tilde{x}\tilde{y}\gen {\tilde{y}}.$$ It is isomorphic to the $\Sl(2,\mathbb{R})$ algebra. The   general solution  depending on two independent arbitrary constants is given by
\begin{equation}\label{EP-sol}
  y^2=A e^{2\alpha x} +2B+Ce^{-2\alpha x}, \quad (AC-B^2)\alpha^2=K.
\end{equation}

\begin{remark}
We comment that as we mentioned above, differential equations may remain invariant under discrete symmetry groups. For example, the special Ermakov equation $y_2=Ky^{-3}$ is invariant under the discrete transformation
$$\tilde{x}=\frac{1}{x},  \quad \tilde{y}=\frac{y}{x}$$ since
$$\tilde{y}_2=\frac{d^2\tilde{y}}{d\tilde{x}^2}=x^3 y_2=Kx^3 y^{-3}=K\tilde{y}^{-3}.$$ More trivial ones are
$$(x,y)\mapsto (-x,y),  \quad (x,y)\mapsto (x,-y), \quad (x,y)\mapsto (-x,-y).$$
This type of symmetries can't be obtained from the local Lie symmetry algebra so they can't be characterized as a one-parameter group transformation.

\end{remark}

In order to obtain another interesting subclass integrable by quadratures, we now let Eq. \eqref{2ndODE}  be invariant under the scaling transformation generated by $\mathbf{v}_2=y\gen y$, or equivalently choosing $H(r,w)=r G(w/r)$. With further specialization  $G(z)=Az+B$, where $A,B$ are arbitrary constants,  we obtain the variable coefficient linear invariant equation
\begin{equation}\label{lin-inv-2ndODE}
  y''+\hat{p}(x)y'+\hat{q}(x)y=0,
\end{equation}
\begin{equation}\label{PQ}
  \hat{p}(x)=\xi^{-1}(\xi'-2\eta-A),  \quad \hat{q}(x)=\xi^{-2}(\eta^2-\xi \eta'+A \eta-B).
\end{equation}
Eliminating $\eta$ results in the relation
\begin{equation}\label{relation}
  (2\xi\xi''-\xi'^2)+4\xi^2 I(x)=-(A^2+4B)\equiv -D^2,
\end{equation}
where $I(x)$ is the (semi)-invariant of \eqref{lin-inv-2ndODE}, namely
\begin{equation}\label{semi-inv}
  I(x)=\hat{q}(x)-\frac{1}{2}\hat{p}'(x)-\frac{1}{4}\hat{p}(x)^2.
\end{equation}
Eq. \eqref{relation} is related to the Ermakov--Pinney equation
\begin{equation}\label{EP-0}
  \chi''+I(x)\chi=-\frac{D^2}{4}\chi^{-3}
\end{equation}
by the transformation $\xi=\chi^2$ and to the linear equation
\begin{equation}\label{3rd-lin}
  \xi'''+4I\xi'+2I'\xi=0
\end{equation}
by differentiation. The remarkable properties of \eqref{3rd-lin} can be found in \cite{CarinenaGuengoerTorres2020}.

Eq. \eqref{lin-inv-2ndODE} admits a symmetry group isomorphic to the $\SL(3,\mathbb{R})$ group (the projective group of the plane $(x,y)$, preserving straight lines in the $(x,y)$-plane) as the symmetry group and can be integrated by quadratures using the two-parameter abelian subgroup generated by $\curl{\mathbf{v}_1,\mathbf{v}_2}$. Passing to the canonical coordinates $r$ and $s$ (as defined in Example \ref{fiber-preser}) reduces \eqref{lin-inv-2ndODE} to the constant coefficient linear equation
\begin{equation}\label{cc-ode}
  r''(s)-A r'(s)-B r(s)=0,
\end{equation}
preserving the homogeneity property in $r$ (invariance under the scaling $r\gen r$).

In the special case of the inversional group ($\xi=x^2$, $\eta=x$), the corresponding invariant ODE becomes
\begin{equation}\label{lin-inv-2ndODE-2}
  y''-\frac{A}{x^2}y'+\frac{Ax-B}{x^4}y=0.
\end{equation}
Its normal form is easy to obtain from the relation \eqref{relation} as
\begin{equation}\label{lin-inv-2ndODE-2-norm}
  v''+I(x)v=0, \quad I(x)=-\frac{D^2}{4x^4},   \quad y=\exp\curl{-\frac{A}{2x}}v.
\end{equation}
Then, we have $r=y/x$, $s=-1/x$, and we can easily solve \eqref{cc-ode} to find the general solution of \eqref{lin-inv-2ndODE-2}
\begin{equation}\label{sol}
  y(x)=x\left[c_1\exp\left(-\frac{\lambda_{+}}{x}\right)+
c_2\exp\left(-\frac{\lambda_{-}}{x}\right)\right],
\end{equation}
where $\lambda_{\pm}=(A\pm D)/2$, $D^2=A^2+4B>0$ are the real roots  of the characteristic equation $P(\lambda)=\lambda^2-A\lambda-B=0$. If the discriminant $D^2$ is zero or negative, the solution should be modified appropriately.

Alternatively,  we can use the differential invariant approach. By means of the differential invariant $z=y'/y=v'/v$ of $\mathbf{v}_2^{(1)}=y\gen y+y'\gen{y'}$ we can express \eqref{lin-inv-2ndODE-2-norm} as a Riccati equation
\begin{equation}\label{special-Riccati}
  \frac{dz}{dx}+z^2=\frac{D^2}{4x^4},
\end{equation}
which inherits the symmetry $\tilde{\mathbf{v}}_1=x^2\gen x+(1-2xz)\gen z$. Using the canonical coordinates $\rho=x(xz-1)$, $s=-1/x$ satisfying $\tilde{\mathbf{v}}_1(\rho)=0$, $\tilde{\mathbf{v}}_1(s)=1$, it can be written as a separable equation, invariant under the group of translations $(\rho, s)\to (\rho, s+\varepsilon)$,
$$\frac{d\rho}{ds}=\frac{D^2}{4}-\rho^2$$ with solution
\begin{equation}
  \rho(s)=
  \begin{cases}
    \frac{D}{2}\tanh[\frac{D}{2}(s+c_0)], &  D^2>0, \;\;\; \\
    \frac{D}{2}\cot[\frac{D}{2}(s+c_0)], & D^2<0,\\
    (s+c_0)^{-1},  & D^2=0.
  \end{cases}
\end{equation}
Finally, from the relation
$$z=\frac{y'}{y}=\frac{1}{x}+\frac{\rho}{x^2}$$
by another integration we recover the general solution of \eqref{lin-inv-2ndODE-2}, and in particular solution \eqref{sol} after some manipulation with the arbitrary constants. For the special case $A=0$, $B=-1$ and $D^2=-4$, the solution of \eqref{special-Riccati} is given by
$$z=x^{-2}\left[x-\cot\left(\frac{1}{x}+c_0\right)\right].$$
Consequently, we obtain the solution of
$$y''+x^{-4}y=0$$ as
$$y=c_1 x \sin\left(\frac{1}{x}+c_0\right).$$

The following equation arises in the integrability analysis of the variable coefficient  Basener--Ross model \cite{GuengoerTorres2018}
\begin{equation}\label{BR}
  y''(x)+I(x)y=0, \quad I(x)=-\frac{e^{2x}+4\tau e^x+\tau^2}{4(\tau+2e^{x})^2},
\end{equation}
where $\tau$ is a  constant. The relation \eqref{relation} for the choice $A=1$, $B=-3/16$ ($D^2=A^2+4B=1/4$) and
$$\xi=\frac{1}{2}(2+\tau e^{-x})$$
gives precisely $I(x)$ as above. One can check that
$$\mathbf{v}_1=\xi(x)\left(\gen x-\frac{1}{2}y\gen y\right)$$ generates one-parameter symmetry group of \eqref{BR}. So we have $\eta(x)=-\xi(x)/2$,  $\nu(x)=e^{x/2}$, $s(x)=\ln(\tau+2e^x)$, and $r(x)=\nu(x)y$. In canonical coordinates $(s,r)$, $r$ satisfies the constant coefficient equation (ODE \eqref{cc-ode})
$$r''(s)-r'(s)+\frac{3}{16}r(s)=0.$$ Solving this equation and changing to $(x,y)$ coordinates we obtain the general solution
$$y(x)=e^{-x/2}\left[c_1 (\tau+2e^x)^{1/4}+c_2(\tau+2e^x)^{3/4}\right].$$

In general, given $\xi(x)$, $A, B$, one can solve \eqref{relation} for $I(x)$ and thus construct an invariant equation of the form \eqref{lin-inv-2ndODE} with symmetry $\mathbf{v}=\xi\gen x+\eta y\gen y$. The coefficient $\eta(x)$ is found  from solving $\hat{p}(x)=0$ as $\eta(x)=(\xi'-A)/2$. If  $\eta(x)$ is substituted into  $\hat{q}(x)$ of \eqref{PQ} it follows that $\hat{q}=I(x)$ as expected. We conclude that
\begin{equation}\label{cano-lin}
  y''+I(x)y=0, \quad I(x)=-\frac{1}{4\xi^2}(2\xi\xi''-\xi'^2+A^2+4B)
\end{equation}
is invariant under
$$\mathbf{v}=\xi(x)\gen x+\frac{1}{2}(\xi'(x)-A)y\gen y.$$
Solution is readily obtained by transforming this equation into the constant coefficient equation \eqref{cc-ode} by the linear transformation
$$y(x)=\sqrt{\xi}\exp[-(A/2)s]r(s), \quad s=\int\frac{dx}{\xi(x)}.$$
For an equation of the form \eqref{lin-inv-2ndODE} with $\hat{p}\ne 0$, $\xi, A, B$ given we take $\eta=(\xi'-\hat{p}\xi-A)/2$ and $y(x)=\sqrt{\xi}\exp[-(\int \hat{p}dx+A s)/2)]r(s)$.
\end{example}

The knowledge of invariance of an $n$-th order ODE under an $r$-parameter symmetry group can be useful in reducing in order more than once. But the full reduction to an equation of order $n-r$ can only be guaranteed if the symmetry group is solvable.

A Lie algebra $\lie$ is solvable if the derived series defined recursively by the chain of subalgebras
\begin{equation}\label{chain}
  \lie=\lie^{(0)}\supseteq \lie^{(1)}\ldots \supseteq \lie^{(k)} \supseteq \ldots, \quad \lie^{(k)}=[\lie^{(k-1)},\lie^{(k-1)}]
\end{equation}
terminates, namely there exists  $k\in \mathbb{N}$  such that $\lie^{(k)}=0$. The algebra of commutators $D\lie=\lie^{(1)}=[\lie,\lie]$ is called the derived algebra. For the solvable algebra $D^k\lie=\curl{0}$ holds.

\begin{theorem}[\cite{Olver1993}]
  Let $\mathsf{E}(x,y^{(n)})=0$ be an $n$-th order ODE. If $\mathsf{E}(x,y^{(n)})=0$ admits a solvable $r$-parameter group of symmetries $G$ such that , then the general solution of the equation can be found by quadratures from the general solution of an $(n-r)$-th order reduced ODE $\mathsf{E}/G$. In particular, if the ODE admits a solvable $n$-parameter group of symmetries, then the general solutions can be found by quadratures alone.
\end{theorem}

A solvable three-dimensional Lie algebra $\lie$ always contains a two-dimen\-sion\-al abelian ideal, which is unique up to conjugacy under inner automorphisms unless  $\lie$ is abelian or nilpotent. An integration strategy for an ODE with a  three-parameter solvable symmetry group  is to first reduce the equation by this two-dimensional ideal and then to use the remaining symmetry to complete the integration by quadratures.

\begin{example}\label{3rd-ord-quadrature}
We turn to Eq. \eqref{3rd-ODE}.
  As this equation admits a three-dimensional solvable  algebra $\lie=\e(2)$  as the symmetry algebra we can integrate it by three consecutive quadratures. The derived series of $\lie$ is
$$\lie\supset D\lie=\curl{\mathbf{v}_1,\mathbf{v}_2},  \quad D^2\lie=\lie^{(2)}=\curl{0}.$$

The third order differential invariants of the ideal $\curl{\mathbf{v}_1,\mathbf{v}_2}$ are $z=y_1$, $\rho=y_2$,  $\rho'(z)=d\rho/dz=z_2/z_1$, in terms of which Eq. \eqref{3rd-ODE} reduces to the first order ODE
\begin{equation}\label{1st-ODE}
  (1+z^2)\rho'(z)-3z\rho=\rho H(\kappa),  \quad \kappa=(1+z^2)^{-3/2}\rho.
\end{equation}
This equation should retain the final (inherited) symmetry $\mathbf{v}_3$, which, in terms of $z, \rho$, has the reduced form (from restriction of $\pr{2}\mathbf{v}_3$ to $z, \rho$ coordinates)
$$\tilde{\mathbf{v}}_3=(1+z^2)\gen z+3z\rho\gen \rho.$$ The coordinates $\kappa=(1+z^2)^{-3/2}\rho$, $\chi=\arctan z$ rectifies the vector field $\tilde{\mathbf{v}}_3=\gen z$. In terms of $\kappa$, $\chi$, \eqref{1st-ODE} becomes a separable equation
\begin{equation}\label{separ}
 \frac{d\kappa}{d\chi}=\kappa H(\kappa)
\end{equation}
with implicit solution $ \hat{H}(\kappa)=\chi+c_1$ or solving for $\rho=y_2=G(y_1,c_1)$, which is invariant under the translational algebra $\curl{\mathbf{v}_1,\mathbf{v}_2}$ and can be integrated by two further quadratures.

The special case $H=\lambda=\mbox{const.}$ leads to the similitude invariant equation
$$(1+y_1^2)y_3=(3y_1+\lambda)y_2^2$$
with the additional symmetry $\mathbf{v}_4=x\gen x+y\gen y$. From \eqref{separ}, the solution of the reduced equation is $\kappa=c_1 e^{\lambda \chi}$, and with the original variables, it is the second order ODE
$$y_2=c_1 (1+y_1^2)^{3/2}\exp\curl{\lambda \arctan y_1}.$$

For $\lambda=0$, the solutions are curves with the constant curvature $\kappa=c_1$, i.e. the family of circles with radius $c_1^{-1}$
$$  (x-c_2)^2+(y-c_3)^2=c_1^{-2}.$$
In this case the equation admits  two further additional symmetries
$$\mathbf{v}_5=(x^2-y^2)\gen x+2xy\gen y, \quad  \mathbf{v}_6=2xy\gen x+(y^2-x^2)\gen y.$$ The  maximal symmetry algebra of the equation is the six-dimensional Lorentz algebra $\ort(3,1)$ (see also Example \ref{Lap}).

When $\lambda\ne 0$, the corresponding second order ODE can be integrated using two-parameter translational group. More conveniently, a parametric solution can be produced by the introduction of the parametrization $\tau=\arctan y_1$ in the form
$$x(\tau)=c_1 e^{-\lambda \tau}(\lambda \cos \tau-\sin \tau)+c_2,  \quad y(\tau)=c_1 e^{-\lambda \tau}(\lambda \sin \tau+\cos \tau)+c_3.$$

To see what happens when the symmetry group is not solvable we consider the following example of a third-order ODE, known as the Schwarzian equation,
\begin{equation}\label{Schwarz-ODE}
  \frac{y_3}{y_1}-\frac{3}{2}\left(\frac{y_2}{y_1}\right)^2=F(x),
\end{equation}
where $F$ is any function of its argument,
admitting the nonsolvable symmetry group $\SL(2,\mathbb{R})$, with Lie algebra having the basis
\begin{equation}\label{sl2-y}
  \mathbf{v}_1=\gen y,  \quad \mathbf{v}_2=y\gen y, \quad \mathbf{v}_3=y^2\gen y.
\end{equation}
The corresponding Lie group $\SL(2,\mathbb{R})$ is the group of linear fractional transformations
$$(x,y)\to \left(x,\frac{a y+b}{c y+d}\right), \qquad
\begin{pmatrix}
a & b\\c & d
\end{pmatrix}
\in \SL(2,\mathbb{R}).$$
The expression on the left-hand side of \eqref{Schwarz-ODE} is called the Schwarzian derivative of $y$ with respect to $x$ and is denoted by the symbol $\curl{y,x}$. It is invariant under the M\"{o}bius transformation in $y$:
$$\curl{\frac{ay+b}{cy+d};x}=\curl{y,x}, \quad ad-bc=1$$ (a unique differential invariant of order $\leq 3$ of the algebra \eqref{sl2-y}).
A two-dimensional solvable subalgebra $\lie_0$ generated by $\curl{\mathbf{v}_1,\mathbf{v}_2}$ can be used to reduce the equation to  one of  first order. The second-order differential invariants of $\lie_0$ are $x$ and $w=y_2/y_1$, in terms of which the reduced equation becomes
$$\frac{dw}{dx}=\frac{1}{2}w^2+F(x),$$ which is recognized as a Riccati equation. This equation does not inherit the symmetry $\mathbf{v}_3$ of the original equation. Indeed, the reduced vector field $\mathbf{\tilde{v}}_3$ in terms of the invariants $x$ and $w$
is nonlocal (the so-called exponential vector field)
$$\mathbf{\tilde{v}}_3=2e^{\int w dx}\gen w.$$ On the other hand, the well-known Hopf--Cole transformation $w=-2\frac{\theta'(x)}{\theta(x)}$ ($y_1=y'=\theta^{-2}$) linearizes to
\begin{equation}\label{Riccati}
  \theta''+\frac{F(x)}{2}\theta=0.
\end{equation}
Let $\theta$ and $\psi$ be two independent solutions of \eqref{Riccati}. Then, we have $\psi/\theta=W\int \theta^{-2} dx+c$, where $W$ is the (constant) Wronskian of $\theta$ and $\psi$, (we can put $W=1$ by scaling $\psi$) and $c$ is an arbitrary constant, that can be absorbed to $\psi$ so that we can put $c=0$ without loss of generality. The  solution $y$ can now be expressed as a ratio $y=\psi/\theta$ of two linearly independent solutions to \eqref{Riccati}.

When $F(x)=0$ (known as the Kummer--Schwarz equation; this equation is also encountered
in the study of geodesic curves in spaces of constant curvature), the symmetry algebra $\lie$ becomes six-dimensional and has a direct-sum structure $\lie=\Sl(2,\mathbb{R})\oplus \Sl(2,\mathbb{R})$ spanned by the vector fields \eqref{sl2-y} and
\begin{equation}\label{sl2-x}
 \mathbf{v}_4=\gen x,  \quad \mathbf{v}_5=x\gen x, \quad \mathbf{v}_6=x^2\gen x.
\end{equation}
The symmetry group is then a linear fractional group of both $x$ and $y$ coordinates.
The corresponding solution is the linear fractional (or M\"obius) transformation in $x$: $y=(ax+b)/(cx+d)$, $ad-bc\ne 0$.
\end{example}

\begin{remark}
There are only two types of third order ODE admitting 6-dimensional symmetry groups:
$$(1+y_1^2)y_3=3y_1y_2^2 \quad  \mbox{and} \quad  \curl{y,x}=0.$$ The maximal symmetry algebra for a third order ODE is seven-dimensional and is attained if and only if it is equivalent to the linear equation $y_3=0$  up to a linear point (equivalence) transformation.

In general, for any  ODE of  order $n\geq 3$,  the symmetry group has  dimension $\leq n+4$. The maximal dimension $n+4$ is attained if and only if the equation is linear or linearizable by a point transformation.
\end{remark}

\subsection{Group classification problem}\label{group-class-prob}
When a system involves an arbitrary parameter or arbitrary function, the symmetry group can have a richer symmetry for certain specific forms of these arbitrary terms. The problem of identifying such arbitrary parameters or functions is known as the group classification problem. When there are only parameters or arbitrary functions of a single argument of independent or dependent variable then the problem is easily tackled solving the determining system where splitting is almost always possible. The arbitrary functions are found by solving an ODE (the so-called classifying ODE). The situation becomes rather complicated when the system depends on arbitrary functions of more than one argument. In this case, the group classification problem is often solved by reducing it to the classification of  realizations of low-dimensional abstract Lie algebras  combined with the notion of equivalence group and the knowledge of abstract Lie theory.
\begin{example}\label{two-dim-symm-ex}
A classification problem: We wish to determine all possible forms of $F(y)$ for which the following ODE allows a two-dimensional symmetry algebra:
\begin{equation}\label{ode}
  y''=\mu y'+F(y), \quad \mu\not= 0, \quad F''\ne 0.
\end{equation}
Under the reflection $x\to -x$, we have $\mu\to -\mu$ so we can assume  $\mu>0$ without loss of generality. The special nonlinearity $F(y)=a y^3+b y$ is recognized as Duffing's equation (see page \pageref{Duffing}).
The trivial case $\mu=0$ is clearly integrable. It has the energy first integral (Hamiltonian)
$$H=\frac{1}{2}y'^2+V(y)=\text{constant},  \quad F(y)=-V'(y).$$

The symmetry classification for $\mu\ne 0$ naturally characterizes all possible integrable cases.
Eq. \eqref{ode} arises in obtaining travelling-wave solutions of the nonlinear heat (diffusion) equation
\begin{equation}\label{NLH}
  u_t=u_{xx}+F(u), \quad F''\ne 0.
\end{equation}
These are solutions of the form $u(x,t)=y(z)=y(x-ct)$, being invariant under a combination of time and space translational symmetries. In what follows we take $z\equiv x$.

We can put $\mu=1$ in \eqref{ode} by scaling the independent variable $x\to \mu^{-1}x$ and redefining   $\mu^{-2}F$ as $F$.
The equation is invariant under translation of $x$ so that $\mathbf{v}_1=\gen x$ generates a symmetry. For certain functions $F$, the symmetry group will be two-dimensional.
The general symmetry algebra is generated by vector fields of the form
\begin{equation}\label{vf-1}
  \mathbf{v}=\xi(x,y)\gen x+\eta(x,y)\gen y.
\end{equation}
The second prolongation formula for vector field $\mathbf{v}$ is
$$\pr{2}\mathbf{v}=\xi(x,y)\gen x+\eta(x,y)\gen y+\eta^{(1)}(x,y,y')\gen{y'}+\eta^{(2)}(x,y,y',y'')\gen{y''},$$ where
$$\eta^{(1)}=D_x\eta-y'D_x\xi,  \quad \eta^{(2)}=D_x\eta^{(1)}-y''D_x\xi.$$ Higher order prolongation coefficients can be calculated from the recursion formula \eqref{recurs-p-q-1}
$$\eta^{(k)}=D_x\eta^{(k-1)}-y_kD_x\xi, \quad y_k=y^{(k)},$$ or in terms of the characteristic $Q(x,y,y')=\eta-y' \xi$ as
$$\eta^{(k)}=D^{k}_xQ+\xi y_{k+1}.$$
From the infinitesimal symmetry requirement \eqref{inv-criter} we find $\eta^{(2)}-\eta^{(1)}-\eta F'=0$ on solutions. This condition, replacing $y''$ by $y'+F(y)$, is a cubic polynomial in the first derivative $y'$ with coefficients depending on $x$ and $y$.  Setting these coefficients  equal to zero we obtain the  determining system for the coefficients $\xi$, $\eta$
\begin{eqnarray}
  3 F \xi_y+\xi_x-2\eta_{xy}+\xi_{xx} &=& 0, \\
  2 \xi_y-\eta_{yy}+2\xi_{xy} &=& 0, \\
   \xi_{yy} & = &0,\\
  F'\eta-(\eta_y-2\xi_x)F-\eta_{xx}+\eta_x  & =  & 0.
   \end{eqnarray}
Differentiating twice the  first two equations with respect to $y$ and using the third one  and the condition $F''\ne 0$ ($F$ is not linear) we find
$$\xi_y=0,  \quad \eta_{yy}=0.$$ So this gives $\xi=b(x)$ and $\eta=c(x)y+d(x)$, and then \eqref{vf-1} takes the form
\begin{equation}\label{vf-1-2}
  \mathbf{v}=b(x)\gen x+[c(x)y+d(x)]\gen y, \quad b\ne 0.
\end{equation}

The determining equations are then reduced to solving the classifying ODE
\begin{equation}\label{code}
  (cy+d)F'-(c-2b')F=(c''-c')y+d''-d'
\end{equation}
with the relation
\begin{equation}\label{rel-1}
  2c'-(b'+b'')=0
\end{equation}
from which we have
\begin{equation}\label{rel-2}
  c=\frac{1}{2}(L+b'+b),
\end{equation}
where $L$ is a constant.

Now we shall require the equation under study be invariant under a two-dimensional symmetry algebra with commutation relation
\begin{equation}\label{comm}
  [\mathbf{v}_1,\mathbf{v}_2]=k\mathbf{v}_2,  \quad \mathbf{v}_1=\gen x.
\end{equation}
The case $k=0$ is abelian.

We shall find all possible forms of $\mathbf{v}_2$ that obey the commutation relation and the corresponding $F$ (not linear in $y$). The commutation relation implies that we must have
$$b'=kb,  \quad c'=kc, \quad d'=kd, $$ and then
$$b=b_0e^{kx}, \quad c=c_0e^{kx}, \quad d=d_0e^{kx},$$
$$c''-c'=c_0k(k-1)e^{kx}, \quad d''-d'=d_0k(k-1)e^{kx}.$$ So we have
\begin{equation}\label{v2}
  \mathbf{v}_2=e^{kx}[b_0 \gen x+(c_0 y+d_0)\gen y].
\end{equation}
The classifying ODE then becomes
\begin{equation}\label{code-2}
  (c_0y+d_0)F'-(c_0-2kb_0)F=k(k-1)(c_0y+d_0),
\end{equation}
where $b_0, c_0, d_0$ are constants. The relation \eqref{rel-1} gives
\begin{equation}\label{constr}
  k[2c_0-(k+1)b_0]=0.
\end{equation}

If $k=0$ then  $b$, $c$ and $d$ are arbitrary constants and $F$ should be linear in $y$ (This is seen from differentiation of \eqref{code-2}). This means the symmetry algebra should be nonabelian.

If $k=-1$ then $c_0=0$ and $b_0$ is  arbitrary. $b_0$ cannot be zero, otherwise we would have $\mathbf{v}_2=d_0 \mathbf{v}_1$. So we can put $b_0=1$. The ODE becomes now
$$F'-\beta F=2, \quad d_0=\frac{2}{\beta}, \quad \beta\ne 0,$$ which integrates to $F(y)=\alpha e^{\beta y}-\frac{2}{\beta}$. The  symmetry algebra  spanned by
\begin{equation}\label{sym-k=-1}
  \mathbf{v}_1=\gen x,  \quad \mathbf{v}_2=e^{-x}[\gen x+\frac{2}{\beta}\gen y]
\end{equation}
leaves invariant
\begin{equation}\label{eq-k=-1}
  y''=y'+\alpha e^{\beta y}-\frac{2}{\beta}.
\end{equation}
By scaling transformation $y\to (\beta/2) y$,  followed  by translation $y\to y+(1/2)\ln {\alpha}$, $\alpha>0$, we can put $\beta=2$, and $\alpha=1$, respectively. Our representative equation simplifies to
\begin{equation}\label{eq-k=-1-simple}
  y''=y'+ e^{2y}-1.
\end{equation}

If $k=1$ then we have $b_0=c_0\ne 0$ and
$$(y+\delta)F'+F=0, \quad \delta=\frac{d_0}{b_0}.$$ The corresponding $F$ and the symmetry algebra are given by
$$F=\alpha(y+\delta)^{-1}, \quad \mathbf{v}_1=\gen x,  \quad \mathbf{v}_2=e^x[\gen x+(y+\delta)\gen y].$$
Now we let $k\ne 0, -1$: If we assume $c_0\ne 0$, from \eqref{rel-1} we have
$$b_0=\frac{2c_0}{k+1}\ne 0.$$  If we put $d_0=b_0\frac{\delta(k+1)}{2}=c_0\delta$, the ODE \eqref{code-2} can be written as
\begin{equation}\label{code-3}
 (y+\delta)F'-\gamma F=\frac{2(\gamma^2-1)}{(\gamma+3)^2}(y+\delta), \quad \gamma=\frac{1-3k}{k+1}.
\end{equation}
The solution of the ODE \eqref{code-3} is given by
\begin{equation}\label{sol-nonl}
  F=M(y+\delta)^{\gamma}-\frac{2(\gamma+1)}{(\gamma+3)^2}(y+\delta),\quad \gamma\ne 1,-3.
\end{equation}
The symmetry algebra is given by
\begin{equation}\label{sym-k=other}
  \mathbf{v}_1=\gen x, \quad \mathbf{v}_2=e^{kx}\left[\gen x+\frac{k+1}{2}(y+\delta)\gen y\right], \quad k=\frac{1-\gamma}{\gamma+3}.
\end{equation}
Note that the above result also contains the subcase $k=1$ (but not $k=-1$). By change of basis
$\mathbf{v}_1\to -\mathbf{v}_1$, $\mathbf{v}_2\to \mathbf{v}_2$ and $\mathbf{v}_1\to k\mathbf{v}_1$, $\mathbf{v}_2\to \mathbf{v}_2$ the commutation relation takes the standard form for a two-dimensional nonabelian algebra: $[\mathbf{v}_1,\mathbf{v}_2]=\mathbf{v}_2$.

The form of  \eqref{sol-nonl} is not changed by the linear transformation $\tilde{y}=py+q$. The arbitrary coefficient $F(y)$ gets changed into
$ \tilde{F}(\tilde{y})=pF(\frac{\tilde{y}-q}{p}) $. By a suitable choice of $p$ and $q$ we can set $M=1$ and $\delta=0$.

The two-dimensional symmetry algebra  makes reduction of the original equation possible to quadratures by methods of differential invariants or canonical coordinates. We consider the case $k\not=-1$ with two-parameter symmetry group generated by
\begin{equation}\label{2-param-sym}
  \mathbf{v}_1=\gen x,  \quad \mathbf{v}_2=e^{kx}\left[\gen x+\frac{k+1}{2}y\gen y\right].
\end{equation}
The function $F$ in the invariant ODE \eqref{ode} is given by
\begin{equation}\label{nonlinearity}
  F(y)=y^{\gamma}+\frac{(k^2-1)}{4}y, \quad \gamma=\frac{1-3k}{k+1}.
\end{equation}
The above relation between $\gamma$ and $k$ is a reemergence  of the integrability condition \eqref{int-cond} discussed in Example \eqref{integ-2ndODE}. Compare this equation with \eqref{integ-2ndODE}.

According to Theorem \ref{inherited}, we start to reduce by the normal subalgebra $\mathbf{v}_2$, which,  in terms of the coordinates
\begin{equation}\label{rect-coord}
  s=k^{-1}e^{-kx}, \quad r=ye^{-(k+1)x/2},
\end{equation}
has the canonical form $\tilde{\mathbf{v}}_2=-\gen s$. The transformed equation becomes
\begin{equation}\label{trans-ode}
  \frac{d^2r}{ds^2}=e^{\frac{(3k-1)}{2}x}\left[-\frac{k^2-1}{4}y+F(y)\right]=r^{\gamma},
\end{equation}
which retains the symmetry $\mathbf{v}_1$. In terms of the invariants $r$ and $\rho=dr/ds$ of $\mathbf{v}_1$, it reduces to the first order ODE
$\rho d\rho/dr=r^{\gamma}$ with solution $\rho^2=2\ln r+c_1$ if $\gamma=-1$; $\rho^2=2(\gamma+1)^{-1}r^{\gamma+1}+c_1$,  otherwise. The general solution is obtained implicitly after a second quadrature from $(dr/ds)^2=\rho^2(r)$. The alterative way is to integrate directly using the integrating factor $I=r'(s)$. When $\gamma=2,3$, the solutions are obtained in terms of Jacobi elliptic functions or Weierstrass $\wp$ function, when $\gamma>3$ in terms of hyperelliptic integrals.

The case $k=-1$ is treated similarly. The canonical coordinates $r=y-x$, $s=e^{x}$ maps Eq. \eqref{eq-k=-1-simple} to $r''(s)=e^{2r}$ with symmetry $\curl{\gen s,s\gen s-\gen r}$. The corresponding solution is given by
$$y=x+\ln\curl{\frac{2\sqrt{c_1}c_2}{\exp\curl{-\sqrt{c_1}e^x}-c_2^2\exp\curl{\sqrt{c_1}e^x}}}, \quad c_1>0$$
or
$$y=x+\ln\curl{\sqrt{-c_1}\sec[\sqrt{-c_1}(e^x+c_2)]}, \quad c_1<0.$$

\begin{remark}
We note that Example \ref{two-dim-symm-ex} does not provide a complete group (or symmetry) classification. Rather, it identifies classes of $F(y)$ for which the equation is invariant only under  a two-dimensional symmetry group which is sufficient for reduction to quadratures.

For an equation depending on arbitrary functions of more than a single variable, to solve the problem of complete classification with respect to its Lie point symmetries it is imperative to use equivalence group of the equation.
In this framework, let us consider the group classification  of the more general equation
\begin{equation}\label{gen-ode}
  y''=f(x,y).
\end{equation}
Due to the absence of the first derivative $y'$ in $f$, invariance under the general symmetry group generated by \eqref{vf-1} constrains considerably infinitesimal generators $\xi, \eta$
\begin{equation}\label{vf-gen-ode}
  \xi(x,y)=\alpha(x)y+\xi(x),  \quad \eta(x,y)=\alpha'(x)y^2+\eta(x)y+\beta(x),
\end{equation}
where the functions $\alpha(x),\beta(x), \xi(x), \eta(x)$ are subject to the following equations
\begin{equation}\label{deqs-gen-ode}
 \begin{split}
    &  \xi''-2\eta'+3(\alpha f-\alpha'' y)=0,\\
     & \xi f_x+\eta f_y=(\eta-2\xi')f+\alpha'''y^2+\eta'' y+\beta''.
 \end{split}
\end{equation}
From the first equation by taking differentiation twice with respect to $y$ we find
\begin{equation}\label{split-f}
  \alpha f_{yy}=0.
\end{equation}
Hence, if $\alpha=0$, then $\xi''-2\eta'=0$ and $\eta=\xi'/2+c_0$.  We have a single classification relation between $f$, $\xi$ and $\beta$.
\begin{equation}\label{class-f}
 \xi(x)f_x+(\eta(x) y+\beta(x))f_y=(\eta-2\xi')f+\eta'' y+\beta''.
\end{equation}
For given $\xi$, $\beta$ and $c_0$, the first order linear PDE \eqref{class-f} can be integrated using method of characteristics to find the most general form of $f$. For $c_0=0$, $f(x,y)$ is given by
\begin{equation}\label{most-gen-f}
  f(x,y)=\frac{(\sqrt{\xi})''}{\sqrt{\xi}}y+\frac{1}{\sqrt{\xi}}\left(\frac{\beta}{\sqrt{\xi}}\right)'
 + \xi^{-3/2}F(\sigma),  \quad \sigma=\frac{y}{\sqrt{\xi}}-\int \xi^{-3/2}\beta(x)dx,
\end{equation}
where $F$ ($F''\ne 0$) is an arbitrary function of its argument. We note that for $\xi=x^2$,  $\beta=0$ we recover  \eqref{inv-group-inv-ode} with the arbitrary function $H$ independent of the differential invariant $w$.

Imposing a further symmetry on the equation \eqref{most-gen-f}   one can generate a second order differential equation   invariant under a two-parameter group of symmetries with the arbitrary function $F$  restricted. As an example, we require invariance under the scaling  group $(x,y)\to (\lambda x, \lambda^{k}y)$, $\lambda>0$ and pick $\xi, \beta$ for $c_0=0$ such that the generators
\begin{equation}\label{2-dim-alg}
  \mathbf{v}_1=\xi\gen x+(\frac{\xi'}{2}y+\beta)\gen y,  \quad \mathbf{v}_2=x\gen x+k y\gen y
\end{equation}
form a two-dimensional Lie algebra with commutator
$$ [\mathbf{v}_1,\mathbf{v}_2]=(\xi-x \xi')\gen x+\left(k \beta-x \beta'-\frac{\xi''}{2}x y\right)\gen y=-\mathbf{v}_1,$$
which implies
$$x\xi'-2\xi=0,  \quad (k+1)\beta-x \beta'+\frac{1}{2}(\xi'-x \xi'')y=0$$
with solutions
$\xi=x^2$, $\beta=\delta x^{k+1}$  so we find the generators
$$\mathbf{v}_1=x^2\gen x+(xy+\delta x^{k+1})\gen y,  \quad \mathbf{v}_2=x\gen x+k y \gen y.$$ With these choices of $\xi, \beta$, invariance under $\mathbf{v}_2$ will restrict the arbitrary $F$ to
\begin{equation}\label{spec-F}
  F(\sigma)=F_0 \sigma^{(k+1)/(k-1)}, \quad \sigma=\frac{y}{x}-\frac{\delta}{k-1}x^{k-1}, \quad k\ne 1.
\end{equation}
In conclusion, the corresponding  invariant differential equation becomes
\begin{equation}\label{2-param-inv-ode}
  y''=\delta k x^{k-2}+F_0x^{-3}\sigma^{(k+1)/(k-1)}
\end{equation}
and can be integrated by two quadratures in the order of $\mathbf{v}_1, \mathbf{v}_2$.  The special case $\delta=0$ of \eqref{2-param-inv-ode} agrees with the Emden--Fowler equation \eqref{EF-eq} for $n=(k+1)/(k-1)$ and in particular with \eqref{minus5} for $k=2$.

When $k=1$,
\begin{equation}\label{spec-F-1}
  F(\sigma)=F_0 e^{-(2/\delta)\sigma},  \quad \sigma=\frac{y}{x}-\delta \log x,  \quad \delta\ne 0
\end{equation}
the  equation assumes the form
\begin{equation}\label{2-param-inv-ode-2}
  y''=x^{-1}\left[\delta+F_0 e^{-(2/\delta)(y/ x)}\right],
\end{equation}
invariant under a rank-2 Lie algebra with generators
$$\mathbf{v}_1=x^2\gen x+x(\delta x+y)\gen y,  \quad \mathbf{v}_2=x\gen x+ y\gen y.$$
Using the normal coordinates $r=\sigma$ and $s=-1/x$ of $\mathbf{v}_1$ satisfying $\mathbf{v}_1(\sigma)=0$, $\mathbf{v}_1(s)=1$ we can write \eqref{2-param-inv-ode-2} as
$$\frac{d^2 \sigma}{ds^2}=F_0 e^{-(2/\delta) \sigma},$$
which is certainly invariant under $\tilde{\mathbf{v}}_1=\gen s$ and $\tilde{\mathbf{v}}_2=s\gen s+\delta \gen \sigma$.

We can form an equivalent realization of the Lie algebra spanned by the elements \eqref{2-dim-alg}  satisfying the commutation relation
\begin{equation}\label{comm-two-dim}
  [\mathbf{v}_1,\mathbf{v}_2]=(\xi-x \xi')\gen x+\left(k \beta-x \beta'-\frac{\xi''}{2}x y\right)\gen y=(1-m)\mathbf{v}_1.
\end{equation}
In this case, we have
$$\xi=x^m,  \quad \beta=\delta x^{k+m-1}$$  for which Eq. \eqref{most-gen-f} remains remains invariant under $\mathbf{v}_2$ when $F$ is limited to a power function of $\sigma$
\begin{equation}\label{spec-F-2}
  F=F_0\sigma^{\nu},  \quad \sigma=x^{-m/2}\left(y+\frac{2\delta x^k}{m-2k}\right), \quad \nu=\frac{4-3m-2k}{m-2k}, \quad m\ne 2k.
\end{equation}
The corresponding invariant equation is significantly simplified for $\nu=2$ or $m=2/5(k+2)$ ($k\ne 1/2$). In this case, there are two  values of the pairs $(m,k)$, that are worth examining.

First we choose $(m,k)=(6/7,1/7)$ so that $\nu=2$, $\beta=\delta$ (constant). After some calculation, we find the invariant equation as
\begin{equation}\label{inveq-1}
  y''=F_0 x^{-15/7}y^2+\frac{(343F_0 \delta-12)}{196x^2}(7\delta x^{1/7}+4y).
\end{equation}
This equation is further simplified if we choose $\delta=12/(343 F_0)$
\begin{equation}\label{nu-two-1}
  y''=F_0 x^{-15/7}y^2
\end{equation}
admitting the symmetry algebra
$$\mathbf{v}_1=x^{6/7}\gen x+(\frac{3}{7}x^{-1/7}y+\frac{12}{343 F_0})\gen y, \quad \mathbf{v}_2=x\gen x+\frac{1}{7}y \gen y$$ with commutation relation $[\mathbf{v}_1,\mathbf{v}_2]=1/7 \mathbf{v}_1$.

In a similar way, the choice $(m,k)=(8/7,6/7)$ so that $\nu=2$, $\beta=\delta x$ leads to the invariant equation
\begin{equation}\label{inveq-2}
  y''=F_0 x^{-20/7}y^2+\frac{(343F_0 \delta+12)}{196x^2}(7\delta x^{6/7}-4y).
\end{equation}
Again the choice $\delta=-12/(343 F_0)$ gives the invariant equation
\begin{equation}\label{nu-two-2}
 y''=F_0 x^{-20/7}y^2
\end{equation}
admitting the symmetry algebra
$$\mathbf{v}_1=x^{8/7}\gen x+(\frac{4}{7}x^{1/7}y-\frac{12}{343 F_0}x)\gen y, \quad \mathbf{v}_2=x\gen x+\frac{6}{7}y \gen y$$
with commutation relation $[\mathbf{v}_1,\mathbf{v}_2]=-1/7 \mathbf{v}_1$.
We recall that all possible forms of the equation $y''=F_0 x^n y^2$ invariant under a two-dimensional nonabelian symmetry algebra are attained only for three values of the exponents $n\in \curl{-5, -15/7, -20/7}$  (see page \pageref{n=2}).

In the Abelian case $[\mathbf{v}_1,\mathbf{v}_2]=0$ ($\nu=1$),  we find $\xi=x$, $\beta=\delta x^k$  and this leads to a linear equation.

On the other hand, if $f$ has some specific form, this linear PDE allows us to find all possible symmetries.  For example, if $f(x,y)=q(x)y^2$, then  we find that $\xi$, $\beta$ should satisfy
$$\xi'''=4q\beta,  \quad \beta''=0,  \quad 2\xi q'+(5\xi'+2c_0)q=0.$$
Once $\xi$, $\beta$ are solved form the first two equations, the last one serves as a compatibility condition.  For  the choice $q=e^x$, it can be shown that there is only one symmetry given by $\mathbf{v}=\gen x-y\gen y$ which reduces the corresponding equation to an Abel equation of second type. For another special case $q=Kx^{m}$ see the discussion on page \pageref{n=2}.

In the nonlinear case  $f_{yy}\ne 0$, the complete group classification of \eqref{gen-ode}
was carried out in \cite{Ovsyannikov2004}. The interested reader is referred to this paper for the details of the classification procedure. We here present the complete classification results in Table \ref{table-class}.

If $\alpha$ is not identically zero then  $f_{yy}=0$ (linear ODE). The pair of equations  \eqref{deqs-gen-ode} for $f=-q(x)y$ by splitting with respect to powers of $y$ gives:
$$\eta(x)=\frac{1}{2}\xi'(x)+c_0, \quad \xi'''+4q\xi'+2q' \xi=0,$$
and $\alpha, \beta$ are  solutions to the original linear ODE  $y''+q(x)y=0$. Consequently, we find that infinitesimals depend on eight arbitrary constants. We recall that general solution of $\xi$ can be expressed in terms of two linearly independent solutions of the original ODE (see formula \eqref{iterate-sol}). We list the eight vector fields found above as
\begin{equation}\label{eight-vf}
  \begin{split}
     &  \mathbf{v}_1=y y_1\gen x+y_1' y^2 \gen y, \quad  \mathbf{v}_2=y y_2\gen x+y_2' y^2 \gen y, \quad \mathbf{v}_3=y\gen y,\\
      & \mathbf{v}_4=y_1^2\gen x+y_1y_1' y \gen y, \quad \mathbf{v}_5=y_1y_2\gen x+\frac{1}{2}(y_1y_2)' y \gen y, \quad \mathbf{v}_6=y_2^2\gen x+y_2y_2' y \gen y,\\
      & \mathbf{v}_7=y_1\gen y,  \quad \mathbf{v}_8=y_2\gen y,
  \end{split}
\end{equation}
where $y_1,y_2$ are solutions to the ODE $y''+q(x)y=0$. They span a Lie algebra isomorphic to the  $\Sl(3,\mathbb{R})$ algebra.

For the special case $q=0$ we have $y_1=1$, $y_2=x$ and upon taking a suitable combination of $\mathbf{v}_5$ with $\mathbf{v}_3$, the vector fields \eqref{eight-vf} turn to the standard basis of $\Sl(3,\mathbb{R})$ in the form
\begin{equation}\label{standard-sl3}
  \begin{split}
     &  \mathbf{v}_1=y \gen x, \quad  \mathbf{v}_2=xy\gen x+ y^2 \gen y, \quad \mathbf{v}_3=y\gen y,\\
      & \mathbf{v}_4=\gen x, \quad \mathbf{v}_5=x\gen x, \quad \mathbf{v}_6=x^2\gen x+x y \gen y,\\
      & \mathbf{v}_7=\gen y,  \quad \mathbf{v}_8=x\gen y.
  \end{split}
\end{equation}
This algebra is admitted if and only if a second order ODE is either linear or linearizable by a point transformation. See Example \ref{ex-2nd-chain}.


We comment that the nonhomogeneous linear ODE
$$y''+q(x)y=r(x)$$
remains invariant under the $\Sl(2,\mathbb{R})$ subalgebra generated by
$$\mathbf{v}=\xi(x)\gen x+\frac{\xi'(x)}{2}y\gen y$$
if $r(x)=r_0 \xi(x)^{-3/2}$.

Moreover, the change of independent variable $\tau=e^{\mu x}$ converts Eq. \eqref{ode} to $$y_{\tau\tau}=\tau^{-2}f(y), $$
where  $f(y(\tau))$  should be obtained from $F(y(x))$  by the inverse transformation. This equivalent equation belongs to the class \eqref{gen-ode}. For nonconstant $\mu=A(x)$, the classification problem was solved in \cite{Ndogmo2010}.

Finally, we comment that the quadratic Li\'{e}nard equation
\begin{equation}\label{q-lienard}
  y''+f(y)y'^2+g(y)=0,
\end{equation}
where $f(y)$ and $g(y)$ are arbitrary smooth functions of $y$, can be transformed to a subclass of \eqref{gen-ode} of the form
\begin{equation}\label{change-depend}
  Y''(x)+G(Y)=0,  \quad G(Y)=W'(y)g(y)
\end{equation}
by the invertible change of dependent variable $Y=W(y)$, $W'(y)\ne 0$ where $W(y)$ is chosen as a solution of $W''-f(y)W'=0$. So complete group classification of \eqref{q-lienard} can be extracted from Table \ref{table-class}.

However, the classification problem for the standard Li\'{e}nard equation, which is a possible extension of \eqref{ode}
\begin{equation}\label{lienard}
  y''+f(y)y'+g(y)=0
\end{equation}
is more complicated. Such a classification, but without effective use of the equivalence group to obtain the simplest forms of $f$ and $g$ and the corresponding symmetry algebras,  was   treated in \cite{PandeyBinduSenthilvelanLakshmanan2009}  excluding  linearizable equations that occur when $f''=0$ and $3g'-f^2=\text{const.}$ (see Example \eqref{ex-2nd-chain} and Subsection \ref{test}).  In the nonlinearizable case, when $f\ne 0$, the maximal symmetry algebra is two-dimensional (The trivial symmetry $\mathbf{v}_1=\gen x$ is extended by an additional symmetry $\mathbf{v}_2$ of the form \eqref{vf-1-2}). The algebra can not be abelian, otherwise the equation would have to be constant coefficient linear.  A remarkable exception occurs  when $f=0$ (see Table \ref{table-class} and formulae \eqref{notable-case}-\eqref{3-dim-sym}).

An alternative classification of \eqref{lienard} can be done by solving the inverse problem: Require invariance under \eqref{v2} and split the resulting condition with respect to $y'$. We distinguish between two cases:

I. $c_0\ne 0$, $b_0\ne 0$: Without loss of generality we put $b_0=1$ and $d_0=0$. We find that $f(y)$ and $g(y)$ satisfy the first order linear ODEs
\begin{equation}\label{f-g-1-sys}
  \begin{split}
     & c_0y f'+kf+k(2c_0-k)=0, \\
      & c_0y g'+(2k-c_0)g=-c_0kyf-c_0k^2y.
  \end{split}
\end{equation}
Solving this system, we obtain
\begin{equation}\label{f-g-1}
  f(y)=C_1 y^{-k/c_0}+k-2c_0,  \quad g(y)=y\left[C_2 y^{-2k/c_0}-c_0C_1y^{-k/c_0}+c_0^2-k c_0\right],
\end{equation}
where $C_1, C_2$ are integration constants.
Letting $1-(2k)/(c_0)=\gamma\ne 1$ we can re-express \eqref{f-g-1} in the form
\begin{equation}\label{f-g-1-2}
 f(y)=C_1 y^{(\gamma-1)/2}+\frac{k(\gamma+3)}{\gamma-1}, \quad g(y)=C_2 y^{\gamma}+\frac{2kC_1}{\gamma-1}y^{(\gamma+1)/2}+\frac{2k^2(\gamma+1)}{(1-\gamma)^2}y
\end{equation}
allowing the two-dimensional nonabelian symmetry algebra spanned by the vector fields
$$\mathbf{v}_1=\gen x, \quad \mathbf{v}_2=e^{kx}[\gen x+\frac{2k}{1-\gamma}y\gen y].$$
The case $f=0$, namely $C_1=0$, $\gamma=-3$, is in particular interesting.   The nonlinearity $g(y)$ becomes
\begin{equation}\label{notable-case}
  g(y)=C_2 y^{-3}-\frac{k^2}{4}y,
\end{equation}
and  $c_0=k/2$.
The corresponding invariant equation is exactly \eqref{EP} for $\alpha=k/2$, namely, Ermakov--Pinney equation with symmetry algebra spanned by the vector fields
\begin{equation}\label{3-dim-sym}
  \mathbf{v}_1=\gen x, \quad \mathbf{v}_2=e^{kx}(\gen x+\frac{k}{2}y\gen y),  \quad \mathbf{v}_3=e^{-kx}(\gen x-\frac{k}{2}y\gen y), \quad k\ne 0.
\end{equation}
This is the only case when the symmetry algebra of \eqref{lienard} is three-dimensional.

If  $f(y)=-1$, $g(y)=-F(y)$ we recover the group classification of \eqref{ode}. Then the first equation of \eqref{f-g-1-sys} gives the condition $k(1+k-2c_0)=0$. We discard the abelian case $k=0$ to eliminate the linearity of $g(y)$. For the choice $c_0=(1+k)/2$ ($k\ne -1$)  we find
$$g(y)=\frac{1}{4}[C_1 y^{(1-3k)/(1+k)}+(1-k^2)y].$$ Compare this result with \eqref{nonlinearity}.
When $k=\pm i \kappa$,  we can replace $g(y)$ by $g(y)=C_2y^{-3}+(\kappa^2/4)y$ and we then have
$$\mathbf{v}_{2,3}=e^{\pm i \kappa x}(\gen x\pm \frac{i \kappa}{2}y\gen y):=\mathbf{w}_2+i \mathbf{w}_3.$$
The basis \eqref{3-dim-sym} gets changed to
$$\mathbf{w}_1=\mathbf{v}_1, \quad \mathbf{w}_2=\cos \kappa x\;\gen x-\frac{\kappa}{2} \sin \kappa x \;y\gen y, \quad \mathbf{w}_3=\sin \kappa x\;\gen x+\frac{\kappa}{2}\cos \kappa x\; y\gen y.$$

For $k=0$, we refer to \eqref{sl2}.

II. $c_0=0$, $b_0\ne 0$: In this case $f$ and $g$ satisfy
\begin{equation}\label{f-g-2-sys}
  \begin{split}
     &  d_0 f'+kf-k^2=0,\\
      & d_0 g'+2kg+d_0kf+d_0k^2=0
  \end{split}
\end{equation}
with solution
\begin{equation}\label{f-g-2}
  f(y)=C_1 e^{-k/d_0y},  \quad g(y)=C_1 e^{-k/d_0y}+ C_2 e^{-(2k/d_0)y}-kd_0.
\end{equation}
The special case  $f=-1$ gives the condition $k(k+1)=0$. For the choice $k=-1$ we have
$$g(y)=C_1 e^{(2/d_0)y}+d_0.$$
We recover the nonlinearity \eqref{eq-k=-1}.

Analogously to the Li\'{e}nard equation, we can request invariance  of its quadratic version \eqref{q-lienard}. Applying second prolongation of \eqref{v2} restricted to \eqref{q-lienard} and splitting with respect to the derivatives $y',y'^{2}$, we  find the classification ODEs for $f(y)$ and $g(y)$
\begin{equation}\label{class-odes}
  \begin{split}
     & (c_0y+d_0)f'+c_0f=0,  \quad 2(c_0y+d_0)f+2c_0-k=0, \\
      & (c_0y+d_0)g'+(2k-c_0)g+k^2(c_0y+d_0)=0.
  \end{split}
\end{equation}
The solution of this system leads to an integrable subclass of \eqref{q-lienard} admitting a two-dimensional nonabelian  Lie point symmetry with commutation relation \eqref{comm} and $\mathbf{v}_2$ having the form \eqref{v2} ($b_0=1$)
\begin{equation}\label{inv-coeff-i}
  f(y)=\frac{k-2c_0}{2(c_0y+d_0)}, \quad g(y)=(c_0y+d_0)^{\gamma}-\frac{k^2}{1-\gamma}(c_0y+d_0),  \quad k\ne \frac{c_0}{2},
\end{equation}
for $c_0\not=0$ and
\begin{equation}\label{inv-coeff-ii}
  f(y)=\frac{k}{2d_0},  \quad g(y)=M e^{-(2k/d_0)y}
\end{equation}
for $c_0=0$, $d_0\ne 0$.

By defining $c_0=2k/(1-\gamma)$ and putting $d_0=0$ the invariant equation corresponding to \eqref{inv-coeff-i} can be written as
\begin{equation}\label{inv-ode-1}
  y''-\frac{\gamma+3}{4}\frac{y'^2}{y}+K y^{\gamma}-\frac{k^2}{1-\gamma}y=0.
\end{equation}
The ODE \eqref{inv-ode-1} admits a maximal Lie symmetry algebra of dimension three (in addition to $\mathbf{v}_1,\mathbf{v}_2$) spanned by the vector fields
$$\mathbf{v}_1=\gen x,  \quad \mathbf{v}_2=e^{kx}\left(\gen x+\frac{2k}{1-\gamma}y\gen y\right),  \quad \mathbf{v}_3=e^{-kx}\left(\gen x-\frac{2k}{1-\gamma}y\gen y\right).$$
isomorphic to the $\Sl(2,\mathbb{R})$ algebra. The commutation relations between $\mathbf{v}_1, \mathbf{v}_2, \mathbf{v}_3$ are
\begin{equation}\label{sl2-k}
  [\mathbf{v}_1,\mathbf{v}_2]=k\mathbf{v}_2,  \quad [\mathbf{v}_2,\mathbf{v}_3]=-2k\mathbf{v}_1, \quad [\mathbf{v}_1,\mathbf{v}_3]=-k\mathbf{v}_3.
\end{equation}
By the substitution $\mathbf{v}_1\to k\mathbf{v}_1$, $\mathbf{v}_2\to -k\mathbf{v}_2$, $\mathbf{v}_3\to k\mathbf{v}_3$, they satisfy the commutation relations of the  $\Sl(2,\mathbb{R})$ algebra (see \eqref{sl2-comm}).
The general solution of \eqref{inv-ode-1} can be expressed as
\begin{equation}\label{sol-SchwKumm}
  y=\left(A e^{kx}+2B+Ce^{-kx}\right)^{2/(1-\gamma)},  \quad 4(AC-B^2)k^2=(1-\gamma)K,
\end{equation}
where $A,B,C$ are constants, two of which are arbitrary.
In the special case $\gamma=-3$, \eqref{inv-ode-1} and its solution \eqref{sol-SchwKumm} coincide with \eqref{EP} and  \eqref{EP-sol} (or \eqref{notable-case}), respectively.

In \eqref{inv-coeff-ii}, by scaling $y$ we can set $d_0=k$ and find the invariant ODE
\begin{equation}\label{inv-ode-2}
y''+\frac{1}{2}y'^2+M e^{-2y}-\frac{k^2}{2}=0
\end{equation}
admitting a three-dimensional (maximal) Lie symmetry algebra
\begin{equation}\label{three-dim}
  \mathbf{v}_1=\gen x,  \quad \mathbf{v}_2=e^{kx}(\gen x+k\gen y),  \quad \mathbf{v}_3=e^{-kx}(\gen x-k\gen y)
\end{equation}
satisfying the commutation relations \eqref{sl2-k}.
We can integrate \eqref{inv-ode-2} by using the two-dimensional solvable subalgebra $\{\mathbf{v}_1, \mathbf{v}_2\}$.  In terms of the canonical variables $r=y-kx$, $s=e^{-kx}$ for $\mathbf{v}_2$, \eqref{inv-ode-2} has  the form
\begin{equation}\label{inv-ode-2-r-s}
  \frac{d^2r}{ds^2}+\frac{1}{2}\left(\frac{dr}{ds}\right)^2+\frac{M}{k^2}e^{-2r}=0
\end{equation}
admitting the symmetry  (up to a multiple of $-k$)
$$\tilde{\mathbf{v}}_1=s\gen s+\gen r,  \quad  \tilde{\mathbf{v}}_2=\gen s.$$
We reduce \eqref{inv-ode-2-r-s} using $\gen s$ to the first order ODE with Lie point symmetry $\mathbf{w}_1=\gen r-R\gen R$
$$R\frac{dR}{dr}+\frac{1}{2}R^2+\frac{M}{k^2}e^{-2r}=0,  \quad R=r'(s)=\frac{dr}{ds}.$$
This ODE can be integrated either as a Bernoulli equation
with solution
$$R^2(r)=r'^2=c_1 e^{-r}+\frac{2M}{k^2}e^{-2r}$$ or as an exact equation
using the integrating factor $2e^r$.
Finally, second integration gives
$$4(c_1e^{r}+2Mk^{-2})=(c_1s+c_2)^2$$ or in terms of the original coordinates
$$4(c_1e^{y-kx}+2Mk^{-2})=(c_1e^{-kx}+c_2)^2$$ from which $y$ can be expressed explicitly in terms of $x$.

\begin{table}
  \centering
  \begin{tabular}{|c|c|c|c|c|}
    \hline
   $N$ & $f(x,y)$ & $\mathbf{v}_1$ & $\mathbf{v}_2$ & $\mathbf{v}_3$ \\
    \hline\hline
   $1$ & $f(y)$ & $\gen x$ & 0 & 0 \\
    \hline
   $2$ & $e^y$ & $\gen x$ & $x \gen x-2 \gen y$ & 0 \\
    \hline
   $3$ & $y^k$, $k\ne -3$  & $\gen x$ & $(k-1)x\gen x-2y \gen y$ & 0 \\
    \hline
   $4$ & $\pm y^{-3}$ & $\gen x$ & $2x\gen x+y\gen y$ & $x^2\gen x+xy\gen y$ \\
    \hline
   $5$ & $x^{-2}F(y)$ & $x\gen x$ & 0 & 0 \\
    \hline
  \end{tabular}
  \caption{Group classification of $y''=f(x,y)$}\label{table-class}
\end{table}

\end{remark}

\begin{remark}
The special case $\gamma=2$, $k=-1/5$ and $F(y)=y(y-6/25)$ of \eqref{nonlinearity} deserves a special attention, because this relates, up to a scaling of $y$, to the travelling-wave solutions of Fisher's equation (or nonlinear reaction-diffusion equation)
\begin{equation}\label{Fisher}
  u_t=u_{xx}+a u(1-u).
\end{equation}
This equation originated in 1936 to model the propagation of a gene in a population.
Notice that under the change of variable $u\to 1-u$, the parameter $a$ changes sign.
The corresponding travelling-wave solutions  are then reduced to integrating the second order ODE $r''(s)=r^2$. The general solution of this equation is $r=\wp(s/\sqrt{6}+a;0,g_3)$,  where $a, g_3$ are arbitrary constants. Here $\wp(s;g_2,g_3)$ is the Weierstrass $\wp$ function with invariants $g_2$ and $g_3$.

Changing to the original variables using  $s=-5e^{x/5}$, $r=ye^{-2x/5}$ from \eqref{rect-coord}, for $g_3\ne 0$, the solutions  can be written as
\begin{equation}\label{p-sol}
  y(x)=e^{2x/5}\wp\left(\frac{5}{\sqrt{6}}e^{x/5}+a;0,g_3\right).
\end{equation}
We note that $\wp$ is an even function since $g_2=0$ and $g_3$ is arbitrary.
They are doubly periodic with an infinite number of poles on the real axis. The solutions of biological interest are obtained for $g_3=0$. Using the fact that $\wp(s;0,0)=s^{-2}$   we find
$$y(x)=\left[\frac{5}{\sqrt{6}}+ae^{-x/5}\right]^{-2}$$
satisfying the boundary conditions $\lim_{x\to \infty}y(x)=6/25$ and  $\lim_{x\to -\infty}y(x)=0$.
In \cite{AblowitzZeppetella1979}, it was shown that for the special wave speed $c=5/\sqrt{6}$, for which the Fisher's equation passes the necessary condition to be of Painlev\'{e} type (we already discussed this constraint on page \pageref{special-wave-speed} as a condition of integrability by quadratures),  can be reduced to the canonical form  $r''=6r^2$ by a transformation of the  independent and dependent variables of the form $y(x)=f(x)r(s)$, $s=h(x)$ with $f$ and $h$ appropriately chosen.

Another context in which $\gamma=2$ appears is the study of  travelling solutions of the two-dimensional Korteweg--de Vries--Burgers ($\mu\ne 0$) and Kadomtsev--Petviashvili  ($\mu =0$) equations for the special quadratic nonlinearity $F(u)$ \cite{EstevezKuruNegroNieto2006}.  In this paper, factorization technique was applied to obtain solution \eqref{p-sol}.

We recall that the case $\gamma=3$, $k=-1/3$ corresponds to the travelling wave solutions for the Newell--Whitehead--Segel equation. The rectifying transformation is $r=-3e^{x/3}$, $s=e^{-x/3}y$.
\end{remark}

\end{example}

\section{Group-invariant solutions}

One of the main applications of the notion of symmetry group to PDEs is to construct group-invariant solutions. Suppose that $G$ is a symmetry group of the system \eqref{sys}. A solution $u=f(x)$ is called group-invariant if $g.f=f$ for all $g\in G$.
This means that a group-invariant solution does not change under the symmetry group transformations. For example, if $G$ is the group of rotations in the space of independent variables $x$, then a  solution invariant under $G$ will be a function of the radius alone in the form $u=F(|x|)$. Travelling wave solutions are solutions invariant under the group of translations. Self-similar (or similarity) solutions that frequently arise in applications correspond to scaling symmetries.

\begin{theorem}
  Suppose that the symmetry group $G$ acts  on the space of independent and dependent variables $E=X\times U$ and sweeps out generic orbits of dimension $d$ and of codimension $p+q-d$ (the number of functionally independent invariants of the group $G$). Then all the  group-invariant solutions to $\mathsf{E}=0$ can be found by solving a reduced system of differential equations $\mathsf{E}/G = 0$ in $d$ fewer independent variables.
\end{theorem}
For example, if we have a system of partial differential equations in two independent
variables, then the solutions invariant under a one-parameter symmetry group can
all be found by integrating a system of ordinary differential equations.

Reduction in the number of independent variables will be possible if the orbit dimension $d$ satisfies the inequality $p\leq d$. When $d=p$, the reduced system $\mathsf{E}/G = 0$ is a system of algebraic equations, while if $d>p$ there are no group-invariant solutions. In particular, if $d=p-1$ we have a system of ODEs.

Let G be a local Lie group of transformations with infinitesimal
generators $\mathbf{v}_1, \ldots, \mathbf{v}_r$, and the associated characteristics $Q_1,\ldots,Q_r$.  Then a function is invariant under $G$ if and only if it is a solution to the system of quasilinear first order partial differential equations characterizing the functions  invariant under $G$
\begin{equation}\label{fixed-Q}
  Q_{\alpha,l}(x,u^{(1)})=0,  \quad \alpha=1,2,\ldots,q, \quad l=1,2,\ldots,r.
\end{equation}
The group-invariant solution thus will satisfy both the original system together with the invariance constraints \eqref{fixed-Q}, which form an overdetermined system of PDEs. The method of symmetry reduction consists of solving \eqref{fixed-Q} in terms of invariant coordinates and substituting these into the original system. In the final step all non-invariant coordinates will drop out of the resulting reduced system. What remains to derive group-invariant solutions is to solve the reduced system depending on fewer independent variables.

Given a  solution invariant under a subgroup $H$ of the full symmetry group $G$ of a system, it can be transformed to other group-invariant solutions by elements $g\in G$ not in the subgroup $H$. Two group-invariant solutions are called inequivalent if one can not be transformed to the other by some group transformation $g\in G$. The corresponding reduced systems also have to be inequivalent.   Two subgroups which are conjugate under the symmetry group $G$ will produce equivalent reduced systems (systems connected with a transformation in the symmetry group).

Let $H \subset G$ be a $s$-parameter subgroup. If $u=f(x)$ is a solution invariant under $H$ and $g\in G$ is any other group element, then the transformed function $u=\tilde{f}(x)=g.f(x)$ is a solution invariant under the conjugate subgroup $K_g(H)=g.H.g^{-1}$.

For example, the stationary solutions $u=f(x)$ of an evolution equation such as heat equation and KdV (Korteweg--de Vries) equation invariant under Galilean group $\mathbf{v}_b$ can be conjugated by the Galilean boosts $e^{-c\mathbf{v}_b}\gen t e^{c\mathbf{v}_b}$ to map to travelling wave solutions $u=f(x-ct)$ and vice versa.

So the collection of all group-invariant solutions are partitioned into equivalence classes. If $\lie$ is the Lie algebra of $G$, we are interested in obtaining a  representative  list of subalgebras of $\lie$ (also called an optimal system of subalgebras) such that every subalgebra of $\lie$ is conjugate to precisely one algebra in the list and no two subalgebras in the list are conjugate. Two algebras $\lie$ and $\mathfrak{g}'$ are conjugate under $G$ if $\lie=G\mathfrak{g}'G^{-1}$. Consequently, the problem of classifying group-invariant solutions is completely tantamount to finding the optimal system of subalgebras. When this task  has been finished, it is sufficient to construct only group-invariant solutions corresponding to these representative subalgebras in the list. All other  solutions  can be obtained by applying the symmetry transformations to the representative classes of solutions.

For a finite dimensional Lie algebra of dimension $\geq 2$, this is in general a complicated problem. If $\lie$ is a direct sum of two or more algebras, there is an algorithmic classification method \cite{Winternitz1992, PateraSharpWinternitzZassenhaus1977, PateraWinternitzZassenhaus1975a},  which is an adaptation to Lie algebras of the Goursat's method  for direct products of discrete groups. When $\lie$ is semidirect sum of two algebras, for example when $\lie$ is a Levi decomposition of a semi-simple algebra and its radical (maximal solvable ideal) a method has been proposed in \cite{PateraWinternitzZassenhaus1975} (see also \cite{Winternitz1989}). In physical applications, one usually encounters with low-dimensional algebras as symmetry algebras. In such cases, subalgebras of real three- and four-dimensional Lie algebras \cite{PateraWinternitz1977} can be of great use in the classification of group-invariant solutions.

In the one-dimensional case, subalgebra classification problem is solved by using the adjoint transformations using the Baker--Campbell--Housedorff formula (or the Lie series)
\begin{equation}\label{BCH}
 \Ad(\exp(\varepsilon \mathbf{v}))\mathbf{w}=\exp{(\varepsilon \mathbf{v})}\mathbf{w}\exp{(-\varepsilon \mathbf{v})}=\mathbf{w}+\varepsilon [\mathbf{v},\mathbf{w}]+\frac{\varepsilon^2}{2!}[\mathbf{v},[\mathbf{v},\mathbf{w}]]+\ldots,
\end{equation}
where $\mathbf{v}, \mathbf{w}\in \lie$ and $\varepsilon$ is the group parameter.
This series terminates or can be easily summed up when the vector field $\mathbf{v}$ is a nilpotent element or generates a translation, rotation or scaling group. The basic idea is to take a general vector field $\mathbf{v}=\Span\curl{\mathbf{v}_1,\ldots,\mathbf{v}_r}$  in $\lie$ and to simplify it as much as possible using adjoint transformations.

\begin{example}
We give an example of one-dimensional subalgebras for the simple algebra $\ort(3,\mathbb{R})$  with commutation relations
$$[\mathbf{v}_1,\mathbf{v}_2]=\mathbf{v}_3,  \quad [\mathbf{v}_3,\mathbf{v}_1]=\mathbf{v}_2, \quad [\mathbf{v}_2,\mathbf{v}_3]=\mathbf{v}_1.$$
$\ort(3,\mathbb{R})$  has no no-trivial ideals and it has only one-dimensional (non-trivial) subalgebras.

Let $\mathbf{v}=a_1\mathbf{v}_1+a_2 \mathbf{v}_2+a_3\mathbf{v}_3\in \ort(3,\mathbb{R})$.
For $a_3\ne 0$, using the formula \eqref{BCH} and the above commutation relations and summing infinite series in $\varepsilon$ we find
$$\Ad(\exp(\varepsilon \mathbf{v}_2))\mathbf{v}_1=\cos \varepsilon \mathbf{v}_1-\sin \varepsilon \mathbf{v}_3,  \quad \Ad(\exp(\varepsilon \mathbf{v}_2))\mathbf{v}_3=\sin \varepsilon \mathbf{v}_1+\cos \varepsilon \mathbf{v}_3,$$ and from the linearity of the adjoint transformation
$$\mathbf{v}'=\Ad(\exp(\varepsilon \mathbf{v}_2))\mathbf{v}=a_1'\mathbf{v}_1+a_2'\mathbf{v}_2+a_3'\mathbf{v}_3,$$
where
$$a_1'=a_1\cos \varepsilon+a_3 \sin \varepsilon,  \quad a_2'=a_2,  \quad a_3'=-a_1\sin \varepsilon+a_3 \cos \varepsilon.$$
Note that $a_1'^2+a_2'^2+a_3'^2=a_1^2+a_2^2+a_3^2$ for any $\varepsilon\in \mathbb{R}$. This means that the positive-definite function $I(\mathbf{v})=a_1^2+a_2^2+a_3^2>0$ (unless $a_1=a_2=a_3=0$) is an invariant of the corresponding adjoint transformation.   It can be seen that $I$ is an invariant of the full adjoint action: $I(\Ad_g(\mathbf{v}))=I(\mathbf{v})$, $g\in \Ort(3,\mathbb{R})$, the Lie group of $\ort(3,\mathbb{R})$.
If  $a_3\ne 0$, we can choose $\varepsilon$ such that the coefficient $a_3'$ is zero. So we may assume that $a_3=0$ and $\mathbf{v}=a_1\mathbf{v}_1+a_2 \mathbf{v}_2$. If $a_2\ne 0$, we further conjugate this form of $\mathbf{v}$ by $\mathbf{v}_3$ (or equivalently apply an inner automorphism of the algebra) and find
$$\mathbf{v}'=\Ad(\exp(\alpha \mathbf{v}_3))\mathbf{v}=(a_1\cos \alpha-a_2 \sin \alpha)\mathbf{v}_1+(a_1 \sin \alpha+a_2\cos \alpha)\mathbf{v}_2.$$
Now we can arrange for $\alpha$ to be such that the coefficient of $\mathbf{v}_2$ is zero and hence $\mathbf{v}$ is conjugate to $a_1 \mathbf{v}_1$ for $a_1\ne 0$ or to $\mathbf{v}_1$ by scaling. The optimal system contains only one vector field $\mathbf{v}_1$.
Any solution invariant under the one-parameter group generated by $\mathbf{v}\in \ort(3,\mathbb{R})$ is equivalent to a solution invariant under the group generated by $\mathbf{v}_1$.
\end{example}

The example of one-dimensional subalgebras of the $\Sl(2,\mathbb{R})$ algebra can be found in \cite{Olver1993, Ovsyannikov1982, Hydon2000}. For the subalgebra classification of the symmetry algebra of the heat equation, which will be needed in the following example, see for example \cite{Weisner1959, Winternitz1989, Olver1993}.

\begin{example}\label{diff-pde}
  We consider the diffusion equation with constant drift $b$
\begin{equation}\label{diff}
  u_t=xu_{xx}+b u_x,  \quad b>0.
\end{equation}
As this equation contains a parameter  the structure of its symmetry group will crucially  depend on the parameter. The problem of determination of all possible symmetries can be regarded as a group classification problem.
Let the symmetry group of \eqref{diff} be generated by the vector fields
\begin{equation}\label{vect-f}
  \mathbf{v}=\tau(t,x,u)\gen t+\xi(t,x,u)\gen x+\varphi(t,x,u)\gen u,
\end{equation}
where the coefficients $\tau$, $\xi$ and $\varphi$ will be determined from the invariance criterion \eqref{inv-criter}. We need to know the second prolongation of $\mathbf{v}$
$$\pr{2}\mathbf{v}=\mathbf{v}+\varphi^t\gen{u_t}+\varphi^x\gen{u_x}+\varphi^{xx}\gen{u_{xx}}+\ldots.$$
Applying the criterion to ${\mathsf{E}}(t,x,u,u_t,u_x,u_{xx})=u_t-x u_{xx}-b u_x$ gives on solutions $$\varphi^t-x\varphi^{xx}-\xi u_{xx}-b \varphi^x=0,$$
where, from \eqref{coeff-Q-pro}, in terms of the characteristic $Q=\varphi-\xi u_x-\tau u_t$,
$$\varphi^t=D_tQ+\xi u_{xt}+\tau u_{tt}, \quad \varphi^x=D_xQ+\xi u_{xx}+\tau u_{xt},$$ and
$$\varphi^{xx}=D_x^2Q+\xi u_{xxx}+\tau u_{xxt}.$$
Eliminating $u_{xx}$ using $\mathsf{E}=0$ and splitting with respect to the derivatives $u_t$, $u_x$, $u_{xt}$ we find
that $\tau_x=\tau_u=0$, $\xi_u=0$ and $\varphi_{uu}=0$, which mean $\tau=\tau(t)$, $\xi=\xi(t,x)$, $\varphi=\phi(t,x)u+\psi(t,x)$. Using this information and setting equal to zero the coefficients of $u$, $u_t$, $u_x$, the determining system is simplified to
\begin{subequations}\label{deq}
  \begin{equation}\label{deq-1}
         2x\xi_x-\dot{\tau}x-\xi=0, \quad x^2(\xi_{xx}-2\phi_x)-x\xi_t-bx\xi_x+b \xi=0,
  \end{equation}
  \begin{equation}\label{deq-2}
     \phi_t=x \phi_{xx}+b \phi_x,  \quad \psi_t=x \psi_{xx}+b \psi_x.
  \end{equation}
\end{subequations}

Solving  for $\xi$ from the first equation of \eqref{deq-1} we find
$$\xi(t,x)=f(t)\sqrt{x}+\dot{\tau}x.$$
Then using this information in the second equation
gives
$$4x^{3/2}(2\phi_x+\ddot{\tau})+4\dot{f}x+(1-2b)f=0,$$ from which on integrating we have
$$\phi(t,x)=-\frac{1}{4}\left[2\ddot{\tau}x+4\dot{f}\sqrt{x}+\frac{(2b-1)f}{\sqrt{x}}\right]+g(t).$$ Substituting $\phi$ in the first equation of \eqref{deq-2}  we find
\begin{equation}\label{rest-deq}
  (2b-1)(2b-3)f=0,  \quad \ddot{f}=0, \quad \dddot{\tau}=0,  \quad b \ddot{\tau}+2\dot{g}=0.
\end{equation}
The last two equations specify $\tau$ and $g$ as
$$\tau(t)=\tau_2 t^2+\tau_1 t+\tau_0, \quad g(t)=-\frac{b}{2} \dot{\tau}+g_0,$$ where $\tau_2,\tau_1,\tau_0,g_0$ are the  integration constants. Finally, the last equation in \eqref{deq-2} decouples from the others, therefore we find that for any solution $\psi(t,x)$ of \eqref{diff}, the vector field $\mathbf{v}_{\psi}=\psi\gen u$ gives rise to a symmetry, which simply reflects the superposition rule of the linear equations.  Hereafter, we shall factor out this infinite-dimensional algebra. The constant $g_0$ corresponds to the scaling symmetry $\mathbf{v}_0=u\gen u$. This means we can multiply solutions by constants.
There are two cases to consider:

\underline{$b\not\in\curl{1/2, 3/2}$} ($f=0$): $\mathbf{v}$ depends on four  integration constants.  The symmetry algebra $\lie_4$ is four-dimensional. A suitable basis is given by
\begin{equation}\label{dim-4}
\begin{split}
  &  \mathbf{v}_1=\gen t, \quad \mathbf{v}_2=t\gen t+x\gen x,  \\
  & \mathbf{v}_3=t^2\gen t+2tx\gen x-(x+bt)u\gen u,  \quad   \mathbf{v}_0=u\gen u.
  \end{split}
\end{equation}

\underline{$b\in\curl{1/2, 3/2}$} ($f\neq 0$): We have $f(t)=f_1 t+f_0$. The symmetry algebra $\lie_6$  is six-dimensional. In this case,  the basis \eqref{dim-4} for $b=1/2$ is extended by two additional elements
\begin{equation}\label{dim-6-1}
\mathbf{v}_4=\sqrt{x}\gen x,  \quad   \mathbf{v}_5=t\mathbf{v}_4-\sqrt{x}u\gen u,
\end{equation}
and for $b=3/2$  by the following two
\begin{equation}\label{dim-6-2}
  \mathbf{v}_4=\sqrt{x}\gen x-\frac{1}{2\sqrt{x}}u\gen u,  \quad \mathbf{v}_5=t\mathbf{v}_4-\sqrt{x}u\gen u.
\end{equation}
\begin{table}
  \centering
  \begin{tabular}{|c|c|c|c|c|c|}
  \hline
  $[\mathbf{v}_i,\mathbf{v}_j]$  & $\mathbf{v}_1$ & $\mathbf{v}_2$ & $\mathbf{v}_3$ & $\mathbf{v}_4$ & $\mathbf{v}_5$ \\\hline
  $\mathbf{v}_1$ & 0 & $\mathbf{v}_1$ & $2\mathbf{v}_2$ & 0 & $\mathbf{v}_4$ \\
  $\mathbf{v}_2$ & -$\mathbf{v}_1$ & 0 & $\mathbf{v}_3$ & -$\mathbf{v}_4/2$ & $\mathbf{v}_5/2$ \\
  $\mathbf{v}_3$ & -2$\mathbf{v}_2$ & -$\mathbf{v}_3$ & 0 & -$\mathbf{v}_5$ & 0 \\
  $\mathbf{v}_4$ & 0 & $\mathbf{v}_4/2$ & $\mathbf{v}_5$ & 0 & -$\mathbf{v}_0/2$ \\
  $\mathbf{v}_5$ & -$\mathbf{v}_4$ & -$\mathbf{v}_5/2$ & 0 & $\mathbf{v}_0/2$ & 0 \\
  \hline
\end{tabular}
  \caption{Commutator Table}\label{comm-tab}
\end{table}
$\mathbf{v}_0$ is the center element of the symmetry algebra.

From the commutator table \eqref{comm-tab} we see that the  algebra $\lie_4$ has the structure of a direct sum $\lie_4= \Sl(2,\mathbb{R})\oplus \curl{\mathbf{v}_0}$, where $\Sl(2,\mathbb{R})\simeq \curl{\mathbf{v}_1,\mathbf{v}_2,\mathbf{v}_3}$, whereas the algebra $\lie_6$ has the semidirect sum structure $\lie_6= \Sl(2,\mathbb{R})\uplus \Heis(1)$, where $\Heis(1)\simeq \curl{\mathbf{v}_0, \mathbf{v}_4, \mathbf{v}_5}$ is the three-dimensional Heisenberg algebra with center $\mathbf{v}_0$. The symmetry algebra $\lie_6$ is isomorphic to the well-known symmetry algebra of the constant coefficient heat equation $u_t=u_{xx}$. The existence of such an isomorphism is a necessary (but not sufficient) condition  for  a point transformation $\Phi: (t,x,u)\to (\tilde{t},\tilde{x},\tilde{u})$ to exist, mapping equations into each other. Indeed, in the case of $b=1/2$, there is a point transformation (not necessarily unique)
\begin{equation}\label{map-1}
  \Phi:\quad \tilde{t}=-\frac{1}{t},  \quad \tilde{x}=\frac{2\sqrt{x}}{t},  \quad \tilde{u}=\sqrt{t}\exp\left(\frac{x}{t}\right)u,
\end{equation}
and
\begin{equation}\label{map-2}
  \Phi:\quad \tilde{t}=-\frac{1}{t},  \quad \tilde{x}=\frac{2\sqrt{x}}{t},  \quad \tilde{u}=\sqrt{tx}\exp\left(\frac{x}{t}\right)u,
\end{equation}
if $b=3/2$, mapping \eqref{diff} to the first canonical form $\tilde{u}_{\tilde{t}}=\tilde{u}_{\tilde{x}\tilde{x}}$.
If $b=1/2$, a simple change of space variable $\tilde{x}=2\sqrt{x}$ alone also does the job.

Let $b=3/2$. The pushforwards $\tilde{\mathbf{v}}_i=\Phi_{*}(\mathbf{v}_i)$ of the vector fields $\mathbf{v}_i$, $i=0,\ldots,5$ of the algebra $\lie_6$ via the map $\Phi$  are easily calculated:
\begin{equation}\label{push}
\begin{split}
    &  \tilde{\mathbf{v}}_1=t^2\gen t+tx\gen x-\frac{1}{4}(x^2+2t)u\gen u,  \quad \tilde{\mathbf{v}}_2=-\left(t\gen t+\frac{x}{2}\gen x\right)+\frac{1}{2}\tilde{\mathbf{v}}_0,    \\
     & \tilde{\mathbf{v}}_3=\gen t, \quad \tilde{\mathbf{v}}_4=-\left(t\gen x-\frac{x}{2}u\gen u\right),  \quad \tilde{\mathbf{v}}_5=\gen x, \quad \tilde{\mathbf{v}}_0=\mathbf{v}_0=u\gen u,
\end{split}
 \end{equation}
where all coordinates should be replaced by tildes (written in the new coordinates). The Lie algebra $\tilde{\mathfrak{g}}$ spanned by $\curl{\tilde{\mathbf{v}}_i}$ is recognized to be the symmetry algebra of $u_t=u_{xx}$, having the Levi  decomposition  $\tilde{\mathfrak{g}}=\curl{\tilde{\mathbf{v}}_3,\tilde{\mathbf{v}}_2,\tilde{\mathbf{v}}_1}\uplus
\curl{\tilde{\mathbf{v}}_5,\tilde{\mathbf{v}}_4,\tilde{\mathbf{v}}_0}\simeq\Sl(2,\mathbb{R})\uplus\Heis(1)$.
This can be verified by using the same steps of symmetry calculation done for Eq. \eqref{diff}.

For $b=1/2$, the pushforwards $\tilde{\mathbf{v}}_i=\Phi_{*}(\mathbf{v}_i)$ of the vector fields $\mathbf{v}_i$, $i=0,1,2,\ldots,5$ of the algebra $\lie_6$ spanned by \eqref{dim-4}-\eqref{dim-6-1} via the map $\Phi:(t,x,u)\to (t,2\sqrt{x},u)$ are exactly the basis vectors of the heat equation algebra:
\begin{equation}\label{heat-symm-basis}
  \begin{split}
      &  \tilde{\mathbf{v}}_1=\gen t, \quad \tilde{\mathbf{v}}_2=t\gen t+\frac{x}{2}\gen x, \quad \tilde{\mathbf{v}}_3=t^2\gen t+xt\gen x-\frac{1}{4}(x^2+2t)\mathbf{v}_0,\\
       & \tilde{\mathbf{v}}_4=\gen x, \quad \tilde{\mathbf{v}}_5=t\gen x-\frac{x}{2}\mathbf{v}_0,  \quad \mathbf{v}_0=u\gen u.
   \end{split}
\end{equation}
The corresponding one-parameter symmetry groups are time translations, scaling in $t,x$, projective (or inversional) transformation, Galilei boosts, space translations and scaling in $u$, respectively.

In the case  $b\notin\curl{1/2,3/2}$, the equation transforms to the second canonical form
\begin{equation}\label{2nd-cano}
  \tilde{u}_{\tilde{t}}=\tilde{u}_{\tilde{x}\tilde{x}}+\frac{\mu}{\tilde{x}^2}\tilde{u}, \qquad \mu=\frac{(2b-1)(2b-3)}{4}
\end{equation}
via the transformation
$$\tilde{t}=t,  \quad \tilde{x}=2\sqrt{x}, \quad u=x^{(1-2b)/4}\tilde{u},$$ admitting the symmetry algebra spanned by
\begin{equation}\label{dim-4-2}
\begin{split}
  &  \mathbf{v}_1=\gen {\tilde{t}}, \quad \mathbf{v}_2=\tilde{t}\gen {\tilde{t}}+\frac{\tilde{x}}{2}\gen {\tilde{x}},  \\
  & \mathbf{v}_3=\tilde{t}^2\gen {\tilde{t}}+\tilde{t}\tilde{x}\gen {\tilde{x}}-\frac{1}{4}(\tilde{x}^2+2\tilde{t})\tilde{u}\gen {\tilde{u}},  \quad   \mathbf{v}_0=\tilde{u}\gen {\tilde{u}}.
  \end{split}
\end{equation}

We thus have given the identification of the symmetry algebras of two locally equivalent differential equations, \eqref{diff} and   $u_t=u_{xx}$. In general, two equivalent differential equations have isomorphic symmetry groups. If a differential equation with a symmetry $g$ is mapped to another one by a point transformation $\Phi$, then $\tilde{g}=\Phi\cdot g\cdot\Phi^{-1}$ is a symmetry of the second equation.

The well-known solutions of $u_t=u_{xx}$ can be mapped to solutions of \eqref{diff} when $b=1/2,3/2$. We note that the given equation is also invariant under the discrete group of simultaneous time and space reflections: $(t,x,u)\to (-t,-x,u)$.

Symmetry group of the equation will now be used to derive explicit group-invariant solutions. We illustrate how to construct group-invariant fundamental  solution of \eqref{diff} for any value of $b$ using the most general element $\mathbf{v}=a_0\mathbf{v}_0+a_1\mathbf{v}_1+a_2\mathbf{v}_2+a_3\mathbf{v}_3$ of the Lie algebra $\lie_4$. A fundamental solution $K(t,x;y)$ as distributions satisfies \eqref{diff} with the initial condition $\lim_{t\searrow 0}K(t,x;y)=\delta(x-y)$, where $\delta(x-y)$ is the Dirac distribution with singularity at $y$. The initial condition puts the following restrictions on the coefficients of \eqref{vect-f} (see \cite{CraddockDooley2001, CraddockPlaten2004, Guengoer2018, Guengoer2018a} for details)
\begin{equation}\label{init-cond}
  \tau(0)=0,  \quad \xi(0,y)=0,  \quad \phi(0,y)+\dot{\tau}(0)+\xi_x(0,y)=0.
\end{equation}
These conditions are satisfied by the vector field (a projective type symmetry)
\begin{equation}\label{vf-fund}
  \mathbf{v}=t^2\gen t+2tx\gen x-(x-y+bt)u\gen u.
\end{equation}
The fundamental solution will be looked for as a solution invariant  the symmetry group generated by  \eqref{vf-fund}. Its invariants are found by solving the invariance condition $Q(x,t,u,u_t,u_x)=\varphi-\tau u_t-\xi u_x=0$ as $\eta=x/t^2$, $\zeta=ut^{b}\exp[(x+y)/t]$. The group-invariant solution has the form
$$u=t^{-b}\exp\curl{-\frac{x+y}{t}}F(\eta).$$
Substituting the above solution $u$ into \eqref{diff} we find that $F$ satisfies
$$\eta F''+b F'-yF=0.$$ We note that this ODE is free of the non-invariant coordinates.  The solution $F$, which is bounded near zero is given by
$$F=\left(\frac{\eta}{y}\right)^{(1-b)/2}I_{b-1}(2\sqrt{y \eta}),$$ where $I_\nu$ denotes the usual modified Bessel function of the first kind and order $\nu$. We recall the following asymptotic relations for small values of $x$
$$I_\nu\sim\frac{1}{2^{\nu}\nu!}x^{\nu},  \quad I_{-\nu}\sim\frac{2^{\nu}}{(-\nu)!}x^{-\nu}, \quad \nu\not\in \mathbb{Z}.$$
We thus have constructed the following fundamental solution, up to a nonzero multiplicative constant $C$,
$$K(t,x;y)=\frac{C}{t}\left(\frac{x}{y}\right)^{\frac{1-b}{2}}\exp\curl{-\frac{x+y}{t}}
I_{b-1}\left(\frac{2\sqrt{xy}}{t}\right).$$ The value of $C$ is determined from the normalization condition.

The special case $b=1/2$ with $C=\pi^{-1/2}$ produces the elementary  solution
\begin{equation}\label{fund-sol-1}
  K(t,x;y)=\frac{1}{\sqrt{\pi ty}}\exp\curl{-\frac{x+y}{t}}\cosh\left(\frac{2\sqrt{xy}}{t}\right).
\end{equation}
Translation along $t$ gives the full heat kernel $K(t,x;t_0,x_0)=K(t-t_0,x;x_0)$.

Similarly, in the case $b=3/2$ with $C=\pi^{-1/2}$, $K$ is again elementary
\begin{equation}\label{fund-sol-2}
  K(t,x;y)=\frac{1}{\sqrt{\pi t x}}\exp\curl{-\frac{x+y}{t}}\sinh\left(\frac{2\sqrt{xy}}{t}\right).
\end{equation}
Both fundamental solutions satisfy the normalization condition $$\int_0^{\infty}K(0,x;y)dx=1.$$

We remark that solutions \eqref{fund-sol-1}-\eqref{fund-sol-2} can be recovered using the transformations \eqref{map-1}-\eqref{map-2} and the separable solutions (solutions invariant under $\mathbf{\tilde{v}}_3+y \mathbf{v}_0$, $y>0$ a parameter)
$$\tilde{u}(\tilde{t},\tilde{x})=Ce^{y \tilde{t}}
\begin{cases}
  \cosh(\sqrt{y}\, \tilde{x}), &  \\
  \sinh(\sqrt{y}\, \tilde{x}), &
\end{cases}
$$
of the heat equation $\tilde{u}_{\tilde{t}}=\tilde{u}_{\tilde{x}\tilde{x}}$.

In the case when $b\ne 1/2, 3/2$, group-invariant solutions can be derived using the one dimensional subalgebras (an optimal system of subalgebras) of the algebra $\lie_4$ (see \eqref{1-dim-subalgebras}).

Now let us look at the action of a subgroup, say $\mathbf{v}_3$ of \eqref{dim-4}, on the solution $u(t,x)$, $\exp\curl{\varepsilon \mathbf{v}_3}u(t,x)$. We find by exponentiating $ \mathbf{v}_3$ the group action generated by $\mathbf{v}_3$
\begin{equation}\label{action}
  \tilde{u}_{\varepsilon}(t,x)=(1+\varepsilon t)^{-b}\exp\curl{\frac{-\varepsilon x}{1+\varepsilon t}}u\left(\frac{x}{(1+\varepsilon t)^{2}},\frac{t}{1+\varepsilon t}\right).
\end{equation}
Formula \eqref{action} states that if $u(t,x)$ is a solution of \eqref{diff}, then  $\tilde{u}_{\varepsilon}(t,x)$ is also a solution. Simple solutions like stationary (time-independent) solutions can be mapped to new  solutions. A stationary solution $u(x)$ satisfies the linear ODE $xu_{xx}+bu_x=0$ with two-parameter solution $u(x)=c_0+c_1 x^{1-b}$ if $b\ne 1$,  and $u(x)=c_0+c_1 \ln x$ if  $b=1$. They get mapped to the following new solutions, respectively:
\begin{equation}\label{new-sol}
\begin{split}
  & u_{\varepsilon}(t,x)=(1+\varepsilon t)^{-b}\exp\curl{\frac{-\varepsilon x}{1+\varepsilon t}}\left[c_0+c_1(1+\varepsilon t)^{2(b-1)}x^{1-b}\right],  \quad b\ne 1,\\
    & u_{\varepsilon}(t,x)=(1+\varepsilon t)^{-1}\exp\curl{\frac{-\varepsilon x}{1+\varepsilon t}}\left[c_0+c_1 \log\frac{\sqrt{x}}{1+\varepsilon t}\right], \quad b=1.
\end{split}
\end{equation}
The remaining symmetry group of translation in $t$ and scaling in $t$ and $x$: $(t,x)\to (\lambda (t+t_0),\lambda x)$  can be applied to these solutions to obtain new ones depending on two more group parameters $t_0,\lambda$ (solutions invariant under the group $\SL(2,\mathbb{R})$).

\end{example}

\begin{example}(\cite{KovalenkoKopasStogniy2013})
  A tricky classification problem for the generalized  Fokker--Planck--Kolmogorov equation with a varying coefficient
\begin{equation}\label{gen-kolmogorov}
  u_t-u_{xx}+p(x)u_y=0.
\end{equation}
We assume that $p$ is not constant, otherwise by a simple change of variable $u=U(t,x,z), z=t-p^{-1}y$, it would be equivalent to the heat equation $U_t=U_{xx}$.
The special case $p(x)=x$ is known as the celebrated Kolmogorov equation. In the sequel we shall show that in this case the Lie point symmetry group is maximal.

The general element of the symmetry algebra  is written in the form
\begin{equation}\label{vf-kolmogorov}
  \mathbf{v}=\tau(t,x,y,u)\gen t+\xi(t,x,y,u)\gen x+\eta(t,x,y,u)\gen y+\varphi(t,x,y,u)\gen u.
\end{equation}
The linearized symmetry condition
$$\varphi^t-\varphi^{xx}+p'(x)\xi u_y+p(x)\varphi^y=0$$  should be satisfied on solutions. Here $\varphi^t, \varphi^{xx}, \varphi^y$ are the coefficients in the appropriate prolongation of $\mathbf{v}$
$$\varphi^t=D_tQ+\tau u_{tt}+\xi u_{tx}+\eta u_{ty}, \quad \varphi^y=D_yQ+\tau u_{yt}+\xi u_{yx}+\eta u_{yy},$$
$$\varphi^{xx}=D^2_{x}Q+\tau u_{xxt}+\xi u_{xxx}+\eta u_{xxy},$$ where $Q=\varphi-\tau u_t-\xi u_x-\eta u_y$ is the characteristic function.

We  find it more convenient to substitute $u_t$ from \eqref{gen-kolmogorov}, rather than $u_{xx}$ as was done in the previous example. The coefficients of the terms involving $u_{xxx}$ imply   that
$$\tau=\tau(t,y), \quad \xi=\xi(t,x,y), \quad \eta=\eta(t,y),  \quad \varphi=f(t,x,y)u+g(t,x,y).$$
From the linearity of the equation we a priori know that $g$ must solve the equation.
Using this simplification and setting the coefficients of the terms $u_{xx}, u_y, u_x$, and $u$ equal to zero lead to the overdetermined system of homogeneous linear PDEs
\begin{equation}\label{det-sys}
  \begin{split}
     & 2\xi_x-\tau_t-p\tau_y=0,  \\
      & \xi p'+p^2 \tau_y-\eta_t+p(\tau_t-\eta_y)=0, \\
      & \xi_t-\xi_{xx}+p\xi_y+2f_x=0, \\
      & f_{t}-f_{xx}+pf_y=0.
  \end{split}
\end{equation}
Inspection of the determining system \eqref{det-sys} suggests that while the principal symmetry algebra $\lie_0$ (valid for any $p$) is obvious
\begin{equation}\label{pr-lie}
  \lie_0=\curl{\mathbf{v}_1,\mathbf{v}_2,\mathbf{v}_3,\mathbf{v}_4}=\curl{\gen t, \gen y, u\gen u, g\gen u},
\end{equation}
where $g$ is a solution of \eqref{gen-kolmogorov},
it is not straightforward to obtain a classifying equation for $p$. To remedy this situation we pass to use the equivalence group and algebra of the equation. It is quite simple to find the equivalence transformations (including the discrete ones) for a linear equation as such and is given by
\begin{equation}\label{equiv-tr}
\begin{split}
   & \tilde{t}=\alpha^2 t+t_0, \quad \tilde{x}=\alpha x+x_0,  \quad \tilde{y}=\beta y+\delta t+y_0, \quad \tilde{u}=\gamma u+\rho(t,x), \\
    & \tilde{p}=\alpha^{-2}(\beta p+\delta),
\end{split}
\end{equation}
where $\alpha,\beta,\gamma,\delta, t_0, x_0, y_0$  are parameters and $\rho(t,x)$ solves $\rho_t-\rho_{xx}=0$. The continuous part of this group is generated by the equivalence algebra $\lie_{\mathcal{E}}$ with basis
\begin{equation}\label{equiv-algebra}
\begin{split}
   & \mathbf{e}_1=\gen t, \quad  \mathbf{e}_2=\gen x, \quad \mathbf{e}_3=\gen y,  \\
    & \mathbf{e}_4=t\gen y+\gen p, \quad \mathbf{e}_5=y\gen y+p\gen p, \\
    & \mathbf{e}_6=t\gen t+x\gen x-2p\gen p, \quad \mathbf{e}_7=u\gen u, \quad \mathbf{e}_8=\rho(t,x)\gen u.
\end{split}
\end{equation}
Now the idea is to project the algebra $\lie_{\mathcal{E}}$ to the space $(x,p)$ so that we have
$\tilde{\lie}_{\mathcal{E}}=\proj_{(x,p)}\lie_{\mathcal{E}}$ with basis
\begin{equation}\label{proj}
\tilde{\mathbf{e}}_2=\gen x, \quad \tilde{\mathbf{e}}_4=\gen p, \quad \tilde{\mathbf{e}}_5=p\gen p, \quad \tilde{\mathbf{e}}_6=x\gen x-2p\gen p.
\end{equation}
This is a four-dimensional solvable subalgebra. We are interested in the extensions of the principal algebra by searching the nontrivial invariant curves in the space $(x,p)$ to the algebra with general element
$$\tilde{\mathbf{e}}=c_2 \tilde{\mathbf{e}}_2+c_4 \tilde{\mathbf{e}}_4+c_5 \tilde{\mathbf{e}}_5+c_6 \tilde{\mathbf{e}}_6$$ for all possible choices of the parameters $c_2, c_4, c_5, c_6$. The possible invariants result from the following combinations
$$\mathbf{e}^{(1)}=x\gen x+\gen p, \quad \mathbf{e}^{(2)}=\gen x+p\gen p, \quad  \mathbf{e}^{(3)}=x\gen x+k p\gen p,$$
$$\mathbf{e}^{(4)}=x\gen x-2p\gen p,  \quad \mathbf{e}^{(5)}=\gen x+\gen p,$$ where $k$ is a constant. The corresponding independent invariants  lead to five different possibilities
\begin{equation}\label{invariants}
  p(x)\in\curl{\ln x, e^x, x^k, x^{-2},x}.
\end{equation}
Actually, they can be replaced by their slightly more general forms under the action of the equivalence group.
For example, $p(x)=x^k$ invariant under $\mathbf{e}^{(3)}$  is equivalent to $p(x)=\lambda(x+\mu)^k+\nu$.

The principal algebra $\lie_0$ will extend for these choices of the coefficient $p$. Given these values, the full symmetry algebra can be obtained by solving the determining equations \eqref{det-sys}. The strategy is to solve the second equation of \eqref{det-sys} for $\xi$
\begin{equation}\label{xi}
  \xi=-\frac{p^2}{p'}\tau_y+\frac{1}{p'}\eta_t-\frac{p}{p'}(\tau_t-\eta_y), \quad p'\ne 0
\end{equation}
and to substitute into the first one resulting in the equation
\begin{equation}\label{tau-eta}
  A(x)\tau_t+B(x)\tau_y+C(x)\eta_t+D(x)\eta_y=0,
\end{equation}
where
$$A(x)=2pp''-3p'^2,  \quad B(x)=p(2pp''-5p'^2), $$
$$C(x)=-2p'',  \quad D(x)=-2(pp''-p'^2).$$
This equation places restrictions on the dependence of $\tau$ and $\eta$ on $t, y$. With this information, the third equation is integrated to determine the form of $f$. Finally, all remaining dependencies  are explicitly specified from the last equation of \eqref{det-sys} as a compatibility condition. We note that $A(x)\equiv 0$ for $p(x)=x^{-2}$, while  $B,C,D$ remain nonzero,  which means that we must have $\tau=\tau(t)$ and $\eta=\text{const.}$ In a similar way, $C(x)\big\vert_{p=x}\equiv 0$,   and   $D(x)\big\vert_{p=e^x}\equiv 0$.

1.) $p(x)=\ln x$:  Splitting \eqref{tau-eta} with respect to $\ln x, (\ln x)^2$, it follows that
$$2(\eta_t+\eta_y)=3\tau_t, \quad \eta_y=5\tau_y+2\tau_t, \quad \tau_y=0.$$ This is easily solved for $\tau, \eta$. We  keep only the integration constants leading to the extensions. We find $\ddot{\tau}=0$ so we take $\tau=t$ ($\tau=1$ leads to the generator of the principal algebra) and then $\eta=y+t/2$, $\xi=x/2$. The remaining equations imply $f=0$ and we conclude that the additional symmetry element is
$$\mathbf{v}_4=2t\gen t+x\gen x+(t+2y)\gen y.$$

2.) $p(x)=e^x$:  Following the same vector fields as before we find two additional elements as
$$\mathbf{v}_4=\gen x+y\gen y, \quad \mathbf{v}_5=2y\gen x+y^2\gen y-(e^x+y)u\gen u.$$

We use $\mathbf{v}_5$ to obtain new solutions from the $y$-independent ($\gen y$ invariant) solution of \eqref{gen-kolmogorov}, namely the heat kernel
$$u=U(t,x)=\frac{c}{\sqrt{t}}\exp\left[-\frac{x^2}{4t}\right].$$ $\mathbf{v}_5$ transforms this solution to
$$u_{\varepsilon}(t,x,y)=\frac{c}{\sqrt{t}(1+\varepsilon y)}\exp\left[-\frac{\varepsilon e^x}{1+\varepsilon y}-\frac{\tilde{x}^2}{4t} \right],  \quad \tilde{x}=x-2\log(1+\varepsilon y).$$

Another class of solutions can be obtained from the $\gen t$ invariant solutions $U(x,y)$ satisfying the PDE
$$U_{xx}+e^x U_y=0.$$
A three-parameter separable solution in terms of Bessel functions of order zero is given by
\begin{equation}\label{separable-sol}
  U(x,y)=e^{\mu y}[A J_0(2\sqrt{\mu} e^{x/2})+B Y_0(2\sqrt{\mu} e^{x/2})],
\end{equation}
where $\mu$, $A, B$ are arbitrary constants. They are transformed to other time independent  solutions
$$u=c(1+\varepsilon y)^{-1}\exp\left[-\frac{\varepsilon e^x}{1+\varepsilon y} \right]U(\tilde{x},\tilde{y}),$$ where
$$\tilde{x}=x-2\log(1+\epsilon y),  \quad \tilde{y}=\frac{y}{1+\varepsilon y}.$$

The remaining symmetries can also be applied to this solution to produce new multi-parameter solutions.

3.) $p(x)=x^k$, $k\ne 0,1,-2$: Eq. \eqref{tau-eta} has the form
$$x^k[(k+2) \tau_t-2\eta_y]+(3k+2)x^{2k}\tau_y+2(k-1)\eta_t=0.$$ This implies that  the cases when $k\in\curl{1,-2,-2/3}$ need separate analysis. When this is not the case, we have
$$\eta_y=\frac{k+2}{2}\tau_t,  \quad \tau_y=\eta_t=0,$$
and then
$$\ddot{\tau}=0, \quad \eta_y=\eta'(y)=\frac{k+2}{2}\dot{\tau}.$$
We find only one additional symmetry
$$\mathbf{v}_4=2t\gen t+x\gen x+(k+2)y\gen y.$$

Below we shall see that the case $k=-2$  in which $\tau=\tau(t)$ and $\eta=0$ (we discard $\eta=\text{const.}$) and $k=1$ lead to more additional symmetries.

4.) $p(x)=x^{-2}$: The first two equations are already satisfied. The third one gives
$$f(t,x,y)=-\frac{\ddot{\tau}}{8}x^2+F(t,y).$$ The last equation determines $F$ and $\tau$ as
$$\dddot{\tau}=0,  \quad F_t=-\frac{\ddot{\tau}}{4}, \quad F_y=0.$$

In conclusion, the additional symmetry vector fields are obtained as
$$\mathbf{v}_4=2t\gen t+x\gen x,  \quad \mathbf{v}_5=4t^2\gen t+4tx\gen x-(2t+x^2)u\gen u.$$

The subalgebra $\curl{\mathbf{v}_1,\mathbf{v}_4,\mathbf{v}_5}$ spans an $\Sl(2,\mathbb{R})$ algebra, which also leaves invariant the one-dimensional heat equation $u_t-u_{xx}=0$. Acting on solutions by $\mathbf{v}_5$ induces the formula transforming solutions $U(t,x,y)$ to the new solutions
\begin{equation}\label{projective-action}
  u_{\varepsilon}(t,x,y)=\frac{1}{\sqrt{1+4\varepsilon t}}\exp\left[\frac{-\varepsilon x^2}{1+4\varepsilon t}\right]U(\tilde{t},\tilde{x},y),
\end{equation}
where $\tilde{t}=t(1+4\varepsilon t)^{-1}$, $\tilde{x}=x(1+4\varepsilon t)^{-1}$,  and $\varepsilon$ is the group parameter. Time independent solutions (invariant under $\mathbf{v}_1=\gen t$) of the form $u=U(x,y)$ satisfy the PDE (one-dimensional heat equation with variable diffusivity)
$$U_y=x^2U_{xx},$$ which is equivalent to the heat equation  \cite{Guengoer2018}
\begin{equation}\label{one-dim-heat}
 V_y=V_{zz}
\end{equation}
by the transformation
$$U=\sqrt{x}\exp\left[-\frac{y}{4}\right]V(z,y), \quad z=\ln x, \quad x>0.$$ This implies we can transform the heat  kernel  (fundamental solution up to a multiplicative constant)
\begin{equation}\label{heat-kernel}
  V(z,y)=\frac{C}{\sqrt{y}}\exp\left[-\frac{z^2}{4y}\right]
\end{equation}
to the  solution
\begin{equation}\label{quadratic-sol}
 u_{\varepsilon}(t,x,y)=\frac{c}{(1+4\varepsilon t)} \sqrt{\frac{x}{y}}\exp\left[-\frac{\varepsilon x^2}{1+4\varepsilon t}-\frac{y^2+z^2}{4y}\right], \quad z=\ln \tilde{x}.
\end{equation}
The translational symmetry group $t\to t+a$ puts the singularity at $t=-(4\varepsilon)^{-1}$ to some $t_0>0$.

\begin{remark}
In general, the time independent solutions of \eqref{gen-kolmogorov} satisfy the heat equation with diffusivity $a(x)=1/p(x)$ ($y$ plays the role of time)
$$U_y=a(x)U_{xx}.$$
It can be shown that such an equation is equivalent to the standard one by a point transformation if and only if $a(x)=(a_2 x^2+a_1 x+a_0)^2$ \cite{Guengoer2018},  which of course includes the case $p=x^{-2}$ discussed above.

Another choice when this is possible is the third case $p=x^k$ where $k=-4$. In this case, a point transformation taking
\begin{equation}\label{quartic-heat}
  U_y=x^4 U_{xx}
\end{equation}
to the heat equation  \eqref{one-dim-heat} is given by
$$U(x,y)=x V(y,z), \quad z=-\frac{1}{x}.$$

A particular extension of the equation under study allowing power diffusivity
\begin{equation}\label{extension-pde}
  u_t=x^5 u_{xx}-x u_y
\end{equation}
deserves consideration. In Ref. \cite{ZhangZhengGuoWu2020}, the symmetry algebra of this equation has been shown to be eight-dimensional with a basis
\begin{equation}\label{8-dim-lie}
  \begin{split}
     & \mathbf{v}_1=\gen t, \quad  \mathbf{v}_2=\gen y, \quad  \mathbf{v}_3=u\gen u,\\
      & \mathbf{v}_4=3t\gen t-x\gen x+2y \gen y, \quad \mathbf{v}_5=-y\gen t+x^2\gen x+x u\gen u,\\
      & \mathbf{v}_6=-\frac{y^2}{2}\gen t+x^2 y \gen x+\left(xy+\frac{1}{2x}\right)u\gen u,\\
      & \mathbf{v}_7=3ty\gen t-x(3tx+y)\gen x+y^2\gen y-(x^{-2}+3tx+3y)u\gen u,\\
      & \mathbf{v}_8=-\frac{y^3}{3}\gen t+x^2y^2\gen x+\left(xy^2+\frac{y}{x}-t\right)u\gen u.
  \end{split}
\end{equation}
One can transform time independent solutions $u=U(x,y)$ of \eqref{extension-pde}
$$U=\frac{Cx}{\sqrt{y}}\exp\left[-\frac{1}{4x^2 y^2}\right]$$ to new  time dependent  solutions using some appropriate subalgebra of the symmetry algebra.

Finally, we note the following solution singular at $(0,x_0,0),$ $x_0>0$
\begin{equation}\label{singular-sol}
  u=\frac{C H(t)x}{y^2}\exp\left[\frac{-3x_0^2t^2x^2+3x_0tx(x+x_0)y-(x^2+x_0x+x_0^2)y^2}{x_0^2x^2y^3}\right],
\end{equation}
where $H(t)$ is the Heaviside function and $C$ is a constant. This solution is invariant under the 4-dimensional subalgebra
\begin{equation}\label{4-dim}
  \begin{split}
     & \mathbf{w}_1=3x_0^2ty\gen t-x_0^2x(3tx+y)\gen x+x_0^2y^2\gen y+\left[1-\frac{x_0^2}{x^2}-3x_0^2tx-3x_0^2y\right]u\gen u, \\
      & \mathbf{w}_2=[3t(x_0+5x_0^3 y)-y]\gen t+[5x_0^3y-x_0x-(15x_0^3 t-1)x^2]\gen x+\\
      & \qquad +x_0y(5x_0^2y+2)\gen y+\left[x-\frac{5x_0^2}{x^2}-15x_0^3 tx-15x_0^3y\right]u\gen u,\\
      & \mathbf{w}_3=\frac{1}{2}(3x_0t-y)y\gen t-\frac{x}{2}[3x_0tx+(x_0-2x)y]\gen x+\frac{x_0}{2}y^2\gen y+\\
      & \qquad +\frac{1}{2x^2}[x-x_0-3x_0x^2y+(2y-3x_0t)x^3]u\gen u,\\
      & \mathbf{w}_4=\mathbf{v}_8,
  \end{split}
\end{equation}
leaving invariant the PDE
$$Lu=u_t+xu_y-x^5u_{xx}=\delta(t)\delta(x-x_0)\delta(y).$$

\end{remark}

5.) $p(x)=x$ (Kolmogorov equation): From Eq. \eqref{tau-eta} we get
$$3\tau_t+5x\tau_y-2\eta_y=0,$$ which implies $\tau=\tau(t)$, $\eta(t,y)=(3/2) \dot{\tau}y+h(t)$. Using these results in the third equation gives us
$$f(t,x,y)=-\frac{3\dddot{\tau}}{4}xy-\frac{\ddot{\tau}}{2}x^2-\frac{\ddot{h}}{2}x+F(y,t).$$
The last equation reduces to the condition
\begin{equation}\label{last-eq}
  \frac{x}{4}[5\dddot{\tau}x+3\tau^{(4)}y+2\dddot{h}]-\ddot{\tau}-xF_y-F_t=0.
\end{equation}
Splitting \eqref{last-eq} with respect to $x$ and $x^2$ gives
$$F_t=-\ddot{\tau},  \quad F_y=\frac{1}{2}\dddot{h}, \quad \dddot{\tau}=0$$ from which
$$F=\frac{1}{2}\dddot{h}y+H(t),  \quad h^{(4)}=0,  \quad \dot{H}=-\ddot{\tau}.$$
Finally,  the general element of the symmetry algebra depends on eight arbitrary integration constants, only five of them leads to additional vector fields:
\begin{eqnarray*}
\mathbf{v}_4&=&\gen x+t\gen y,  \quad \mathbf{v}_5=2t\gen t+x\gen x+3y\gen y-2u\gen u,  \\
\mathbf{v}_6&=& t^2\gen t+(xt+3y)\gen x+3ty\gen y-(2t+x^2)u\gen u, \\
\mathbf{v}_7&=& 3t^2\gen x+t^3\gen y+3(y-tx)u\gen u, \quad \mathbf{v}_8=2t\gen x+t^2\gen y-xu\gen u.
\end{eqnarray*}
We conclude that the maximal symmetry algebra is attained when $p(x)=x$. We note that this equation is also invariant under the discrete transformations
$$(t,x,y,u)\to (t,-x,-y,-u).$$
The construction of the Kolmogorov's fundamental solution $K(x,y,t;x_0,y_0,t_0)$ with the distributional limit property
$$\lim_{t\to t^{+}_0}K(x,y,t;x_0,y_0,t_0)=\delta(x-x_0,y-y_0),$$ using  the approach of Example \eqref{diff-pde} can be found in \cite{Guengoer2018a, KovalenkoKopasStogniy2013} and is given by
\begin{equation}\label{kolmogorov-fundamental}
  K=\frac{\sqrt{3}}{2\pi(t-t_0)^2}\exp\left[-\frac{(x-x_0)^2}{4(t-t_0)}-\frac{3}{(t-t_0)}
\left(\frac{y-y_0}{t-t_0}-\frac{x+x_0}{2}\right)^2\right].
\end{equation}

\end{example}

\begin{example}
  The nonlinear diffusion (porous medium) equation
\begin{equation}\label{poros}
  u_t=u^4 u_{xx}
\end{equation}
is invariant under a five dimensional symmetry group generated by
\begin{equation}\label{five-dim}
\begin{split}
   & \mathbf{v}_1=\gen t, \quad \mathbf{v}_2=\gen x,  \quad \mathbf{v}_3=2t\gen t+x\gen x,\\ & \mathbf{v}_4=x\gen x+\frac{u}{2}\gen u,  \quad \mathbf{v}_5=x^2\gen x+xu\gen u.
\end{split}
\end{equation}
The study of possible invariant solutions was presented in \cite{Gungor2002}. We shall only discuss solutions invariant under the dilatational subalgebra
$$\mathbf{d}=\mathbf{v}_3-\mathbf{v}_4=2t\gen t-\frac{u}{2}\gen u.$$ This implies the reduction ansatz $u=t^{-1/4}F(x)$ leading to the ODE  $F''=-1/4 F^{-3}$, which is the special form of the Ermakov--Pinney equation (see \eqref{EP}). This equation inherits the $\Sl(2,\mathbb{R})$ algebra with basis $\mathbf{v}_2, \mathbf{v}_4, \mathbf{v}_5$ of \eqref{five-dim}. Thus we obtain the particular solution
$$u=t^{-1/4}F(x)=\pm t^{-1/4}(A+2Bx+Cx^2)^{1/2}, \quad 4(B^2-AC)=1.$$

\end{example}

\begin{example}
We now consider the (2+1)-dimensional nonlinear Darboux equation
\begin{equation}\label{Darboux-eq}
   \square u+\frac{b}{t}u_t+au^k=0,
\end{equation}
where $a,b,k$ are arbitrary real parameters and $\square=\partial_{t}^{2}-\Delta_2=\partial_{t}^{2}-\partial_{x}^{2}-\partial_{y}^{2}$ is the wave (also called d'Alembert) operator in the (2+1)-dimensional Minkowski space with metric $dt^2-dx^2-dy^2$.

The linear wave equation  $\square u=0$ obtained for $a=0$, $b=0$ (see Example 2.43 of \cite{Olver1993}) is known to be invariant under the ten-dimensional conformal group $\conf(\mathbb{R}^{1,2})$ of the pseudo-Euclidean space $\mathbb{R}^{1,2}$. A suitable basis of the corresponding Lie algebra (conformal algebra) is given by the seven elements of the Poincar\'e-similitude subalgebra
\begin{equation}\label{p21}
\begin{split}
  & \mathbf{p}_t=\gen t,\quad \mathbf{p}_x=\gen x, \quad   \mathbf{p}_y=\gen y,   \\
  & \mathbf{k}_{xt}=t\gen x+x\gen t,  \quad    \mathbf{k}_{yt}=t\gen y+y\gen t, \quad   \mathbf{j}_{xy}=y\gen x-x\gen y,  \\
  & \mathbf{d}=\dil,
\end{split}
\end{equation}
plus the three conformal generators
\begin{equation}\label{p22}
\begin{split}
  & \mathbf{c}_t=(x^2+y^2+t^2)\gen t+2xt\gen x+2yt\gen y-tu\gen u, \\
  & \mathbf{c}_x=2xt\gen t+(t^2+x^2-y^2)\gen x+2xy\gen y-xu\gen u, \\
  & \mathbf{c}_y=2yt\gen t+2xy\gen x+(t^2-x^2+y^2)\gen y-yu\gen u.
\end{split}
\end{equation}
Eq. \eqref{Darboux-eq} can be viewed as a special case of the nonlinear perturbation of the singular operator $\square +(b/t)\gen t=0$
\begin{equation}\label{perturbed-wave}
  \square u+\frac{b}{t}u_t+f(t,u)=0.
\end{equation}
In this case the symmetry gets considerably reduced:
\begin{equation}\label{infsyn-sub-lin}
   \mathbf{p}_x=\gen x, \quad   \mathbf{p}_y=\gen y, \quad \mathbf{j}_{xy}=y\gen x-x\gen y. \end{equation}
In the special case $f(t,u)=f(u)$, invariance under the dilation
\begin{equation}\label{dil}
   \mathbf{d}=\dil+q u\gen u,   \quad q=\frac{2}{1-k},  \quad k\ne 1
\end{equation}
restricts $f$ to $f(u)=a u^k$ and  the corresponding equation becomes \eqref{Darboux-eq}. With the condition $k=(b+5)/(b+1)$, $b\ne -1$ ($q=-(b+1)/2$), it admits further the following two conformal transformations
\begin{equation}\label{infsym-Darboux-eq}
  \begin{split}
      & \mathbf{c}_x=2xt\gen t+(t^2+x^2-y^2)\gen x+2xy\gen y-(b+1)xu\gen u,\\
      & \mathbf{c}_y=2yt\gen t+2xy\gen x+(t^2-x^2+y^2)\gen y-(b+1)yu\gen u.
  \end{split}
\end{equation}

We comment that the singularly perturbed linear version of \eqref{perturbed-wave} with $f(t,u)=at^{-2}u$   preserves the same symmetry without any restriction on $q$ which can be taken zero.
If, in addition, $2b-b^2+4a=0$  then there is an additional generator
$$\mathbf{p}_t=\gen t-\frac{b}{2t}u\gen u,$$ which is mapped to the translational vector field $\gen t$ by the transformation $u=t^{-b/2}v$ and the new equation becomes the usual wave equation $\square v=0$.

On the other hand, the maximal symmetry algebra of \eqref{Darboux-eq} is achieved when $b=0$ ($k=5$), which reduces to the Klein--Gordon equation (also called classical $\phi^6$-field equation)
\begin{equation}\label{KG}
  \square u=a u^5.
\end{equation}
Its symmetry algebra is  precisely the 10-dimensional conformal algebra  with basis given  by \eqref{p21}-\eqref{p22}, excluding the linearity reflecting symmetry.

The 1+1-dimensional linear variant of \eqref{Darboux-eq} is known as the Euler--Poisson--Darboux (EPD) equation:
\begin{equation}\label{EPD}
  u_{tt}+\frac{b}{t}u_{t}-u_{xx}=0.
\end{equation}
The values $b=0,2$ are exceptional (Eq. is equivalent to the wave equation as mentioned above. See \eqref{vf-Liouv} for its symmetry algebra). The EPD equation admits a four-dimensional symmetry algebra with basis (\cite{Miller1973})
\begin{equation}\label{4-dim-sym}
\begin{split}
   & \mathbf{v}_1=\gen x, \quad \mathbf{v}_2=t\gen t+x \gen x,\\
    & \mathbf{v}_3=2tx\gen t+(x^2+t^2) \gen x-b x u\gen u, \quad \mathbf{v}_4=u\gen u.
\end{split}
\end{equation}
The non-zero commutation relations are
\begin{equation}\label{comm-sl2}
  [\mathbf{v}_1,\mathbf{v}_2]=\mathbf{v}_1,  \quad [\mathbf{v}_1,\mathbf{v}_3]=2\mathbf{v}_2, \quad [\mathbf{v}_2,\mathbf{v}_3]=\mathbf{v}_3.
\end{equation}
The symmetry algebra has the direct-sum structure $\lie=\Sl(2,\mathbb{R})\oplus \{\mathbf{v}_4\}$.

The one-dimensional subalgebras of $\lie$ (\cite{PateraWinternitz1977}) are
\begin{equation}\label{1-dim-subalgebras}
\begin{split}
   &  \curl{\mathbf{v}_1},  \quad \curl{\mathbf{v}_4},  \quad \curl{\mathbf{v}_2+p \mathbf{v}_4}, \quad p\geq 0, \\
    & \curl{\mathbf{v}_1+\mathbf{v}_3+q \mathbf{v}_4}, \quad \curl{\mathbf{v}_1+\epsilon\mathbf{v}_4}, \quad \epsilon=\pm 1, \quad q\in \mathbb{R}.
\end{split}
\end{equation}
We can put $\epsilon=1$ by invariance under $x\to -x$.
All  solutions invariant under these subalgebras can be found by symmetry reduction technique \cite{Miller1973}. For example, the solution corresponding to the subalgebra $\curl{\mathbf{v_1}+\mathbf{v_4}}$ has the form $u=e^{x}F(t)$, where $F(t)$ satisfies the ODE
$$tF''+bF'-tF=0$$ with solution in terms of modified Bessel functions  of the form $$F(t)=t^{(1-b)/2}I_{\pm (b-1)/2}(t).$$
For even integer $b$, this solution is elementary.

Solution  invariant under the third subalgebra admits the reduction formula $u=(x-t)^pF(\tau)$, $\tau=(x+t)/(x-t)$. Substitution into the EPD
equation gives the  hypergeometric equation
\begin{equation}\label{hypergeom}
  \tau(1-\tau)F''+\left[1-p-\frac{b}{2}-\left(1-p+\frac{b}{2}\right)\tau\right]F'+\frac{bp}{2}F=0.
\end{equation}
So the solution can be expressed in terms of hypergeometric function as
$$u=(x-t)^p {}_2F_1\left(-p,\frac{b}{2},1-p-\frac{b}{2};\tau\right), \quad 1-p-\frac{b}{2}\not\in \mathbb{Z}^{\leq 0}.$$
Solutions invariant under $ \curl{\mathbf{v}_1+\mathbf{v}_3}$ will have the form $u=t^{-b/2}F(\tau)$, $\tau=(1+x^2-t^2)/t$ where $F(\tau)$ satisfies
$$(\tau^2+4)F''+2\tau F'-\frac{b(b-2)}{4}F=0.$$
The special cases $b=0,2$ are trivially integrated. Otherwise, the general solution of this equation
can be expressed in terms of Legendre functions with imaginary argument $i\tau/2$ and order $(b-2)/2$, known as oblate spheroidal harmonics for $\tau>0$ and $b=2(n+1)$, $n\in \mathbb{Z}^{0}$.

The $\Sl(2,\mathbb{R})$ invariance can be used to derive identities for hypergeometric and Bessel functions, and transformation formulas for other solutions of the EPD equation.

The symmetry group corresponding to $\mathbf{v}_3$ is found by integrating the following Cauchy problem
\begin{equation}\label{Cauchy}
  \frac{d}{d\epsilon}\tilde{t}(t,x;\epsilon)=2\tilde{t}\tilde{x},  \quad \frac{d}{d\epsilon}\tilde{x}(t,x;\epsilon)=\tilde{t}^2+\tilde{x}^2, \quad \frac{d}{d\epsilon}\tilde{u}(t,x;\epsilon)=-b\tilde{x}\tilde{u}
\end{equation}
subject to the initial conditions $\tilde{t}(t,x;0)=t$, $\tilde{x}(t,x;0)=x$, $\tilde{u}(t,x;0)=u$. The first two equations written in canonical coordinates gives the translational group $\tilde{r}=r$ and $\tilde{s}=s+\epsilon$, where
$$r=\frac{t}{t^2-x^2},  \quad s=\frac{x}{t^2-x^2}$$ satisfying $\mathbf{v}_3(r)=0$, $\mathbf{v}_3(s)=1$. Solving  for $\tilde{t}$ and $\tilde{x}$ we find the transformation of $(t,x)$:
$$\tilde{t}=\frac{t}{1-2\epsilon x-\epsilon^2(t^2-x^2)},  \quad \tilde{x}=\frac{x+\epsilon(t^2-x^2)}{1-2\epsilon x-\epsilon^2(t^2-x^2)}.$$ Finally integrating the last equation in \eqref{Cauchy} we find
$$\tilde{u}=\sigma(t,x;\epsilon)^{b/2}u,  \quad \sigma(t,x;\epsilon)=1-2\epsilon x-\epsilon^2(t^2-x^2).$$ The factor $\sigma$ has the property $\sigma(t,x;\epsilon)=\sigma(\tilde{t},\tilde{x};-\epsilon)^{-1}$. The symmetry group action on a solution $U(t,x)$ induces the transformation formula for solutions
$$u(t,x)=\sigma(t,x;-\epsilon)^{-b/2}U\left(\frac{t}{\sigma(t,x;-\epsilon)},\frac{x-\epsilon(t^2-x^2)}{\sigma(t,x;-\epsilon)}\right).$$

In Ref. \cite{Miller1973},  transformations leaving EPD equation invariant but  changing the parameter $b$ to derive a variety of generating functions for the standard hypergeometric functions ${}_2F_1$ were also considered.

The Euler--Darboux (ED) equation
\begin{equation}\label{ED}
  u_{\xi\eta}+\frac{1}{\xi+\eta}(\alpha u_{\xi}+\beta u_{\eta})=0
\end{equation}
includes the EPD equation \eqref{EPD} as a special case. This equation has important applications in aerodynamics. The change of coordinates $\xi=t+x$, $\eta=t-x$ transforms \eqref{EPD} to \eqref{ED} with $\alpha=\beta=b/2$. Eq. \eqref{ED} for $\alpha=\beta=N\in \mathbb{Z}$ admits a general solution
\begin{equation}\label{gensol}
  u(\xi,\eta)=\frac{\partial^{N-1}}{\partial \xi^{N-1}}\frac{f(\xi)}{\xi+\eta}+\frac{\partial^{N-1}}{\partial \eta^{N-1}}\frac{g(\eta)}{\xi+\eta},
\end{equation}
where $f$ and $g$ are arbitrary functions. This implies that the general solution of \eqref{EPD} can be obtained for  $b=2N$.
For example, for $N=1$ ($b=2$), the d'Alembert  solution of the radial wave equation is obtained
$$u=\frac{f(\xi)+g(\eta)}{\xi+\eta}=\frac{1}{2t}(f(t+x)+g(t-x)).$$

A basis for the symmetry algebra of \eqref{ED} is
\begin{equation}\label{gen-ED}
\begin{split}
   & \mathbf{v}_1=\gen \xi- \gen \eta,  \quad \mathbf{v}_2=\xi\gen \xi+\eta \gen \eta-\frac{(\alpha+\beta)}{2}u\gen u, \\
    & \mathbf{v}_3=\xi^2\gen \xi-\eta^2 \gen \eta-(\beta \xi-\alpha \eta)u\gen u, \quad \mathbf{v}_4=u\gen u
\end{split}
\end{equation}
with the same non-zero commutation relations \eqref{comm-sl2}. The EPD and ED equations have isomorphic symmetry algebras. For values $\alpha=\beta\not=2N$,  we have symmetry methods at our disposal for producing exact solutions of the ED equation.

A complete classification of group invariant solutions of \eqref{KG} would go far beyond the scope of these notes. Instead, we will focus on the solutions invariant under one-parameter subgroup generated by the conformal elements
$$\mathbf{c}=\alpha _0 \mathbf{c}_t+\alpha _1 \mathbf{c}_x+\alpha _2 \mathbf{c}_y,$$ where $\alpha _0, \alpha _1, \alpha _2$ are arbitrary parameters. Invariance under this subalgebra leads to the reduction formula
\begin{equation}\label{ansatz}
u=r^{-1}f(\xi),\quad \xi =r^{-2}(\beta _0 t+\beta _1 x+\beta _2
y),\quad r^2=t^2-x^2-y^2
\end{equation}
with the constraint among the parameters
$$\beta _0 \alpha _0-\beta _1 \alpha _1-\beta _2 \alpha _2=0.$$
Substitution of \eqref{ansatz} into \eqref{KG} gives the second order ODE
\begin{equation}\label{reduced}
\beta ^2 f''=a f^5,\qquad \beta ^2=\beta _0^2-\beta _1^2-\beta _2^2\ne 0,
\end{equation}
which has the first integral
$$f'^2=\frac{a}{3\beta^2}(f^6+c).$$
A further change of the dependent variable $f=F^{-1/2}$ simplifies it to
$$F'^2=\frac{4a}{3\beta^2}(c F^3+1),$$
whose solution for $c\ne 0$ can be expressed in terms of Weierstrass $\wp$ function as
$$F=\wp(z;0,g_3), \quad z=\sqrt{\frac{ac}{3\beta^2}}(\xi+\xi_0), \quad g_3=-4c^{-1}.$$

In addition to the invariant solutions, the subgroup action on solutions obtained through the exponentiation  $\exp\curl{\varepsilon_0 \mathbf{c}_t+\varepsilon_1 \mathbf{c}_x+\varepsilon_2 \mathbf{c}_y}$  can be utilized to derive a formula for the transformation of solutions. More precisely, whenever $U(\mathbf{r})=U(t,x,y)$ is a solution
to \eqref{KG}, so is
$$u({\bf r})=\sigma ^{-1/2} U({\bf \tilde r}),\quad
\sigma(\mathbf{r};\boldsymbol{\varepsilon}) =1-2{\boldsymbol{\varepsilon}}.{\bf r}+\varepsilon ^2 r^2,\quad \tilde
{\bf r}=\sigma ^{-1}({\bf r}-r^2{\boldsymbol{\varepsilon}}), \quad {\bf
r}=(t,x,y),$$ where $\boldsymbol{\varepsilon}=(\varepsilon _0,\varepsilon
_1,\varepsilon _2)$ is the vector of group parameters, $r^2=t^2-x^2-y^2$ and $\varepsilon^2=|\boldsymbol{\varepsilon}|^2=\varepsilon_0^2+\varepsilon_1^2+\varepsilon_2^2$. Note the relation $\sigma(\mathbf{r};\boldsymbol{\varepsilon}) =\sigma(\mathbf{\tilde{r}};-\boldsymbol{\varepsilon})^{-1}$. The reader can compare the conformal transformations $\tilde
{\bf r}=\sigma ^{-1}({\bf r}-r^2{\boldsymbol{\varepsilon}})$ of the $\mathbb{R}^{1,2}$ space leaving its metric form invariant, $d\tilde{\mathbf{r}}^2=\sigma^{-2}d\mathbf{r}^2$,   with those of the Euclidean plane $\mathbb{R}^2$ given by \eqref{conf-cx}-\eqref{conf-cy}.

\end{example}

\begin{remark}
The $(n+1)$-dimensional conformally invariant extension of Eq. \eqref{KG} \cite{FushchichShtelenSerov1993} is
\begin{equation}\label{gen-KG}
   \square u(x)=u_{x_0x_0}-\Delta_n u=a u^{(n+3)/(n-1)}, \quad n\geq 2,
\end{equation}
where $\Delta_n=\sum_{\mu=1}^n u_{x_\mu x_\mu}$ is the Laplace operator on $\mathbb{R}^{n}$ and $x=(x_0,x_1,\ldots,x_n)\in \mathbb{R}^{1,n}$.
The symmetry algebra is $(n+2)(n+3)/2$-dimensional conformal algebra of the conformal group $\conf(\mathbb{R}^{1,n})$ with basis
\begin{equation}\label{sym-conf-inv-KG}
  \begin{split}
     &  \mathbf{p}_\mu=\gen {x_\mu}, \quad \mathbf{d}=\sum_{\mu=0}^{n}x_\mu\gen {x_\mu}+\frac{1-n}{2}u\gen u, \quad \mu=0,1,2,\ldots,n,\\
     &   \mathbf{j}_{\mu\nu}=x_\mu\gen {x_\nu}-x_\nu\gen {x_\mu}, \quad \mathbf{j}_{0\nu}=x_0\gen {x_\nu}+x_\nu\gen {x_0}, \quad \mu,\nu=1,2,\ldots,n\\
     & \mathbf{c}_\mu=2x_\mu\mathbf{d}-\varepsilon x^2\gen {x_\mu}+(1-n)x_{\mu}u\gen u, \quad x^2=x_0^2-\sum_{\mu=1}^{n}x_\mu^2,
     \end{split}
\end{equation}
where $\varepsilon=1$ for $\mu=0$ and $\varepsilon=-1$ for $\mu=1,2,\ldots,n$. The subalgebras spanned by the basis elements $\curl{\mathbf{p}_\mu, \mathbf{j}_{\mu\nu}, \mathbf{j}_{0\nu}}$ and $\curl{\mathbf{p}_\mu, \mathbf{j}_{\mu\nu}, \mathbf{j}_{0\nu}, \mathbf{d}}$ are called the Poincar\'e algebra  of dimension $(n+1)(n+2)/2$ and the similitude algebra of dimension $(n^2+3n+4)/2$, respectively.
\end{remark}

\subsection{Linearization by symmetry structure}
We know that a linear differential equation or one that is linearizable by a point transformation admits an infinite-dimensional symmetry algebra. By checking the existence of such a symmetry structure we can construct linearizing transformations. A well-known example is the potential Burgers' equation \cite{Olver1993}
\begin{equation}\label{Burgers}
  u_t=u_{xx}+u_{x}^2,
\end{equation}
admitting the infinite-dimensional symmetry algebra generated by the vector fields
$$\mathbf{v}_1=\gen x, \quad \mathbf{v}_2=\gen t, \quad \mathbf{v}_3=\gen x, \quad \mathbf{v}_4=2t\gen x-x\gen u,$$
$$\mathbf{v}_5=2t\gen t+x\gen x,  \quad \mathbf{v}_6=4t^2\gen t+xt\gen x-(2t+x^2)\gen u, \quad \mathbf{v}_{\rho}=\rho(x,t)e^{-u}\gen u,$$
where $\rho$ is a solution to the linear heat equation $\rho_t=\rho_{xx}$. It is easy to see that the point transformation $\tilde{u}=e^u$ maps the basis vector fields to \eqref{heat-symm-basis} of the linear equation together with $\tilde{\mathbf{v}}_{\rho}=\rho(x,t)\gen {\tilde{u}}$ and Eq. \eqref{Burgers} to the linear heat equation $\tilde{u}_t=\tilde{u}_{xx}$. The transformation $v=u_x$ relates \eqref{Burgers} to the usual Burgers' equation
\begin{equation}\label{Burgers2}
  v_t=v_{xx}+2vv_x.
\end{equation}
So we have established the celebrated Cole--Hopf transformation $v=(\ln \tilde{u})_x=\tilde{u}_x/\tilde{u}$ taking solutions of  \eqref{Burgers2} to positive solutions of the heat equation.

As well, for higher order PDEs, we can use the same method of looking at the maximal point symmetry algebra (if possible) and constructing a linearizing transformation from its infinite-dimensional symmetry involving an arbitrary function as being solution to a linear PDE. We comment that there might be PDEs with finite symmetry algebra that can be linearized by nonpoint transformations.

\begin{example}(Exercise 2.22 of \cite{Olver1993})
  The nonlinear Thomas equation
\begin{equation}\label{thomas}
  \mathsf{E}=u_{tx}+\alpha u_t+\beta u_x+\gamma u_x u_t=0, \quad \gamma\ne 0
\end{equation}
can be shown to be linearizable by a point transformation. We first compute its symmetry algebra. A general element of the algebra is written in the form
\begin{equation}\label{vf-thomas}
  \mathbf{v}=\tau(t,x,u)\gen t+\xi(t,x,u)\gen x+\varphi(t,x,u)\gen u.
\end{equation}
We need only the prolongation of $\mathbf{v}$ involving coefficients of $u_{tx}, u_t, u_x$
$$\pr{2}\mathbf{v}=\mathbf{v}+\varphi^t\gen {u_t}+\varphi^x\gen {u_x}+\varphi^{tx}\gen {u_{tx}}.$$
Acting on $\mathsf{E}$ by the prolonged vector field provides the linearized equation
\begin{equation}\label{linearized}
 \varphi^{tx}+(\alpha+\gamma u_x)\varphi^t+(\beta+\gamma u_t)\varphi^x=0.
\end{equation}
Replacing $u_{tx}$ by $-(\alpha u_t+\beta u_x+\gamma u_x u_t)$ and splitting with respect to the derivatives $u_{tt}, u_{tt}u_x$ and  $u_{xx}, u_{xx}u_t$ we find
$$\tau_x=\tau_u=0,  \qquad \xi_t=\xi_u=0.$$ Using the simplification $\tau=\tau(t)$ and $\xi=\xi(x)$ and splitting the remaining determining equations resulting from \eqref{linearized}  with respect to the  derivatives involving $u_tu_x$, $u_t$, $u_x$ we find
\begin{equation}\label{det-eqs}
  \begin{split}
     u_tu_x: & \quad \varphi_{uu}+\gamma \varphi_u=0, \\
     u_t:  &  \quad   \varphi_{xu}+\gamma \varphi_x+\alpha \xi'=0, \\
     u_x:  & \quad \varphi_{tu}+\gamma \varphi_t+\beta \dot{\tau}=0,   \\
     u^0:  & \quad \varphi_{tx}+\alpha \varphi_t+\beta \varphi_x=0.
   \end{split}
\end{equation}
From the first equation we have $\varphi(t,x,u)=f(t,x)+e^{-\gamma u}g(t,x)$, $\gamma\ne 0$. Substitution back into \eqref{det-eqs} determines the form of $f$
$$f(t,x)=-\frac{1}{\gamma}(\alpha \xi+\beta \tau)+c,$$ where $c$ is an arbitrary constant. The last equation of  \eqref{det-eqs} gives
$$e^{-\gamma u}(g_{tx}+\alpha g_t+\beta g_x)=\frac{\alpha\beta}{\gamma}(\dot{\tau}+\xi').$$
This implies that if one of $\alpha, \beta$ is zero then $g$ has to satisfy a linear PDE
and $\tau$, $\xi$ remain arbitrary. The symmetry algebra is then infinite dimensional and is spanned by
$$\mathbf{v}=\tau(t)\gen t+\xi(x)\gen x+\left[-\frac{\alpha}{\gamma}\xi(x)+c+g(t,x)e^{-\gamma u}\right]\gen u$$ if $\beta=0$, and
$$\mathbf{v}=\tau(t)\gen t+\xi(x)\gen x+\left[-\frac{\beta}{\gamma}\tau(t)+c+g(t,x)e^{-\gamma u}\right]\gen u$$ if $\alpha=0$. The function $g$ satisfies $g_{tx}+\alpha g_t=0$ and $g_{tx}+\beta g_x=0$, respectively.

Otherwise, we must have $\dot{\tau}+\xi'=0$. This fixes $\tau$ and $\xi$ as $\tau=\lambda t+c_1$, $\xi=-\lambda x+c_2$, where $\lambda, c_1,c_2$ are arbitrary constants. Summarizing, we conclude that the general element  $\mathbf{v}$ depends on four parameters $c_1,c_2, c, \lambda$ and an arbitrary solution of a linear PDE.
So \eqref{thomas} is invariant under an infinite dimensional symmetry algebra with the basis vector fields
\begin{equation}\label{thomas-gen}
\begin{split}
    & \mathbf{v}_1=\gen t, \quad \mathbf{v}_2=\gen x, \quad \mathbf{v}_3=\gen u, \quad \mathbf{v}_4=t\gen t-x\gen x+\gamma^{-1}(\alpha x-\beta t)\gen u, \\
    &   \mathbf{v}(\rho)=e^{-\gamma u}\rho(t,x)\gen u,
\end{split}
\end{equation}
where $\rho(t,x)$ solves the linear equation
\begin{equation}\label{linear-pde}
  \rho_{tx}+\alpha \rho_t+\beta\rho_x=0.
\end{equation}
\end{example}
The fact that the Lie symmetry  algebra depends on solutions of a linear PDE suggests that there must exist a change of dependent variable that will linearize the original equation. Such a transformation is given by $\tilde{u}=F(u)$ with the property $F'=e^{\gamma u}$ so that $\tilde{u}=\gamma^{-1}e^{\gamma u}$ where $\tilde{u}$ satisfies
$$\tilde{u}_{tx}+\alpha\tilde{u}_{t}+\beta\tilde{u}_{x}=0.$$ The linear transformation $\tilde{u}=e^{-\alpha x-\beta t}v(t,x)$ removes the parameters $\alpha$ and $\beta$, resulting in
\begin{equation}\label{hyperbolic-pde}
  v_{tx}-\alpha \beta v=0.
\end{equation}
This equation inherits
the symmetry  $\tilde{\mathbf{v}}_4=t\gen t-x\gen x$ from $\mathbf{v}_4$. Its invariants are $\tau=tx$ and $v$ so we reduce it by the transformation $v=G(\tau)$ to the ODE
$$\tau G''+G'-\alpha\beta G=0.$$ The general solution of this equation can be written in terms of the modified Bessel functions
$$G(\tau)=c_1 I_0(2\sqrt{\alpha\beta \tau})+c_2 K_0(2\sqrt{\alpha\beta \tau}).$$
The choice $c_1=1$, $c_2=0$ yields the fundamental solution (or Green's function) at the origin $(0,0)$ (see \cite{Aksenov2017})
$$G(t,x;0,0)=I_0(2\sqrt{\alpha\beta tx}).$$ It is sufficient to use the translational invariance to obtain the Green's function at $(t_0,x_0)$
$$G(t,x;t_0,x_0)=I_0\left(2\sqrt{\alpha\beta (t-t_0)(x-x_0)}\right).$$
A particular the solution of the given PDE will be in the form
$$u=\frac{1}{\gamma}\log[\gamma G(tx)-\alpha x-\beta t].$$

\begin{example}
We consider a third order KdV-type evolution equation
  $$u_t=u_{xxx}+6u^{-1}u_xu_{xx}+u^{-2}u_x^3.$$
The symmetry algebra is infinite-dimensional with basis elements
$$\mathbf{v}_1=\gen t, \quad \mathbf{v}_2=\gen x, \quad  \mathbf{v}_3=3t\gen t+x\gen x,  \quad \mathbf{v}_4=u\gen u, \quad \mathbf{v}_{\rho}=\rho(t,x)u^{-2}\gen u,$$ where $\rho$ is any solution to the linear KdV equation $\rho_t=\rho_{xxx}$. The transformation $\tilde{u}=u^3$ maps the symmetry algebra to that of the linear KdV equation $\tilde{u}_t=\tilde{u}_{xxx}$,  and hence the original equation to the linear one.
\end{example}

It may not always be practical to calculate the maximal infinite-dimensional symmetry algebra by application of the Lie symmetry algorithm. When this is the case, it is usually useful to give some tests involving only certain finite-dimensional subalgebra of the maximal symmetry algebra. But, in this case, for PDEs the construction of linearizing transformation is a tricky task.

For ODEs, we already encountered such a test (due to Lie) in Example \ref{2nd-ODE-build}. See also Remark \ref{Lie-canonical}. It is sufficient to identify a two-dimensional subalgebra equivalent to the canonical forms $A_{2,2}$ (abelian) or $A_{2,4}$ (nonabelian) as the eight-dimensional symmetry algebra of a linearizable second order ODE. Finding linearizing coordinates is immediate. A further example is the following second member of Riccati chain, known to be linearizable to $z_3=z'''=0$ by the Hopf--Cole transformation $y=z'(x)/z(x)$.

\begin{example}\label{ex-2nd-chain}
\begin{equation}\label{2nd-chain}
  \mathsf{E}(x,y^{(2)})=y_2+3yy_1+y^3=0.
\end{equation}
This equation admits a linearly connected (rank-one) nonabelian 2-dimensional algebra  spanned by
$$\mathbf{v}_1=y\gen x-y^3\gen y,  \quad \mathbf{v}_2=\lambda(x,y)\mathbf{v}_1,  \quad \lambda(x,y)=\frac{x}{y}\left(1-\frac{1}{2}xy\right)$$
with commutation relation $[\mathbf{v}_1,\mathbf{v}_2]=\mathbf{v}_1$.
This is verified by noting
$$\pr{2}\mathbf{v}_1(\mathsf{E})=-3(y^2+y_1)\mathsf{E},$$ and
$$\pr{2}\mathbf{v}_2(\mathsf{E})=\frac{1}{2}[3x^2(y^2+y_1)-4]\mathsf{E},$$ then checking the infinitesimal invariance criterion \eqref{inv-criter}. This knowledge ensures us that Eq. \eqref{2nd-chain} has a eight-dimensional symmetry algebra and thus can be linearized.

A linearizing coordinate transformation $r=r(x,y)$, $s=s(x,y)$ is found by solving the set of linear first order PDEs $\mathbf{v}_1(r)=0$, $\mathbf{v}_1(s)=1$, $\mathbf{v}_2(r)=0$, $\mathbf{v}_2(s)=s$ as
$$r=x-\frac{1}{y},  \quad s=\frac{x}{y}-\frac{x^2}{2}.$$ Note that we simply have chosen $r$ as a joint invariant and $s=\lambda(x,y)$. The transformed algebra is the canonical form $A_{2,4}:\curl{\gen s,s\gen s}$. This transformation linearizes Eq. \eqref{2nd-chain} to $s''(r)=0$. This is readily seen by calculating the following derivative and using the given equation:
$$\frac{d^2s}{dr^2}=-\frac{y^3(y_2+3yy_1+y^3)}{(y^2+y_1)^3}=0.$$ We note that the full eight-dimensional symmetry algebra of the equation is isomorphic to  that of the equation  $s''=0$. This is the $\Sl(3,\mathbb{R})$ algebra generated by the vector fields
\begin{equation}\label{sl3}
  \begin{split}
     & \gen r,  \quad r \gen r,  \quad s \gen r, \quad rs\gen r+s^2\gen s, \\
      & \gen s,  \quad r \gen s,  \quad s \gen s, \quad r^2\gen r+rs\gen s.
  \end{split}
\end{equation}
The corresponding  group  is just the  projective group $\SL(3,\mathbb{R})$ in the $(r,s)$ plane
$$(\tilde{r},\tilde{s}): \quad (r,s)\to \left(\frac{a_1 r+a_2 s+a_3}{a_7 r+a_8 s+a_9},\frac{a_4 r+a_5 s+a_6}{a_7 r+a_8 s+a_9}\right)$$ with the condition
$$\begin{vmatrix}
    a_1 & a_2 & a_3  \\
    a_4 & a_5 & a_6 \\
    a_7 & a_8 & a_9
  \end{vmatrix}\ne 0.$$
The  projective group $\SL(3,\mathbb{R})$ maps the family of straight lines  onto themselves. Another interesting fact is that this group has the lowest order differential invariant starting at order 7. But obviously there are lower order relative differential invariants such as the second order one $I=s''(r)$. This means that, for some differential function $\mu$, $\tilde{s}''(\tilde{r})=\mu s''(r)$ under the projective group.  A differential function $I:J^n\to \mathbb{R}$ is a relative differential invariant of order $n$ of a Lie algebra $\lie$ if for some differential function $\lambda(x,u^{(n)})$
$$\pr{n}\mathbf{v}(I)=\lambda(x,u^{(n)})I$$  for every prolonged vector field $\pr{n}\mathbf{v}\in\lie^{(n)}$. If $\lambda=0$, $I$ is a differential invariant of order $n$. The seventh order differential invariant can be expressed in terms of relative differential invariants (see \cite{KomrakovLychagin1993}).
\end{example}

In the following example we shall make use of the existence of an infinite-dimensional symmetry algebra to find a linearizing transformation.
\begin{example}
\begin{equation}\label{lineariz-ode}
  y_2+p(x)y_1+q(x)y=\sigma \frac{y_1^2}{y},
\end{equation}
where $\sigma$ is an arbitrary constant.
This equation is invariant under the abelian symmetry algebra represented by the vector fields
$$\mathbf{v}_1=y\gen y,  \quad \mathbf{v}_2(\rho)=\rho(x)y^{\sigma}\gen y, \quad \sigma\ne 1,$$   where $\rho(x)$ satisfies the linear equation
\begin{equation}\label{lin-f}
  \rho''(x)+p\rho'+(1-\sigma)q\rho(x)=0.
\end{equation}
We note that under the condition \eqref{lin-f}
$$\pr{2}\mathbf{v}_2(\mathsf{E})=\lambda(x,y) \mathsf{E}, \quad \lambda(x,y)=\sigma \rho(x) y^{\sigma-1},$$ where $\mathsf{E}=y_2+py_1+qy-\sigma y^{-1} y_1^2$. The vector fields
$\mathbf{v}_1, \mathbf{v}_2(\rho)$ reflect the linearity of the equation modulo a transformation. Indeed, the point transformation $y=z^{1/(1-\sigma)}$, $\sigma\ne 1$ maps $\mathbf{v}_2(\rho)$ to $\tilde{\mathbf{v}}_2(\rho)=\rho(x)\gen z$ (preserving the homogeneity  in $z$) and the equation to the linear homogeneous one
$$z_2+p(x)z_1+(1-\sigma)q(x)z=0.$$
If we extend Eq. \eqref{lineariz-ode} to
\begin{equation}\label{lineariz-ode-2}
  y_2+p(x)y_1+q(x)y=\sigma \frac{y_1^2}{y}+g(x,y)
\end{equation}
and impose invariance under $\mathbf{v}_2(\rho)$ we find $g(x,y)=h(x)y^{\sigma}$, $\sigma\not=1$. If $\sigma=1$ we go back to the case $h=0$. The resulting equation  continues to be linearizable by the same transformation. The transformed equation is the nonhomogeneous linear one
$$z_2+p(x)z_1+(1-\sigma)q(x)z=(1-\sigma)h(x).$$
\end{example}

\begin{example}
  Reduction to canonical form (with the coefficients $p,q$ absent) of the second order linear ODE
\begin{equation}\label{2nd-homo-ode}
  y''+p(x)y'+q(x)y=0.
\end{equation}
Let $y_1,y_2=\sigma(x)y_1$, $y_1\sigma'\ne 0$ be linearly independent solutions of \eqref{2nd-homo-ode} with  Wronskian $W(x):=W(y_1,y_2)$.
By the linear superposition principle, it admits a two-parameter symmetry group
\begin{equation}\label{superpos}
  (x,y)\to(x,y+\varepsilon_1 y_1+\varepsilon_2 y_2)
\end{equation}
generating the two-dimensional rank-1 abelian subalgebra of type $A_{2,2}$ (see \eqref{two-dim-cano-abelian}) of its eight-dimensional full symmetry algebra with basis
\begin{equation}\label{superpos-alg}
  \mathbf{v}_1=y_1\gen y, \quad  \mathbf{v}_2=y_2\gen y=\sigma(x)  \mathbf{v}_1.
\end{equation}
The canonical coordinates for the algebra \eqref{superpos-alg} are
\begin{equation}\label{r-s}
  r=\sigma(x)=\frac{y_2}{y_1},  \quad s=\frac{y}{y_1}
\end{equation}
Action of the chain rule formula $D_r=(D_x r)^{-1}D_x$ for the total differentiation operator on $s$
\begin{equation}\label{dsdr2}
\frac{d^2s}{dr^2}=D^2_rs=-(y_1\sigma')^{-3}(y'y_1-yy'_1)[(p\sigma'+\sigma'')y_1+2y'_1\sigma'], \end{equation}
implies that if we choose $\sigma$ as
$$r=\sigma=\int y_1^{-2} \, e^{-\int p dx} dx,$$
then  the given ODE is reduced to the canonical form $d^2s/dr^2=0$. A basis of the Lie symmetry algebra of the canonical equation is given by \eqref{sl3}. The pair of canonical coordinates $(r,s)$ given by
$$r=\sigma,  \quad s=\frac{y}{y_1}$$ can be regarded as a special form of the equivalence transformation of \eqref{2nd-homo-ode}
\begin{equation}\label{equiv-tr-lode}
  \bar{x}=\phi(x),  \quad y=\psi(x)\bar{y},  \quad \phi'(x)\ne 0, \quad \psi\ne 0.
\end{equation}
The solution $s=c_1+c_2 r$ of the canonical equation produces  the well-known general solution formula $y=c_1 y_1+c_2 y_2=y_1(c_1 +c_2 \sigma)$ (superposition principle).

\begin{remark}
  For an  ODE of order  $n>2$, Lie proved that the dimension $d_n$  of the Lie algebra of its symmetry group is at most $n+4$. Maximal dimension $d_n=n+4$ is reached for a linear homogeneous ODE, equivalent to the reduced  one $y_n=y^{(n)}=0$ by the point transformations in the plane. The corresponding Lie group is the semi-direct sum of $\GL(n,\mathbb{R})$ with an $n$-dimensional abelian Lie group. A basis for its Lie algebra is given by
  \begin{equation}\label{basis-n+4}
    \begin{split}
       & \mathbf{v}_1=y\gen y,\quad \mathbf{v}_2=\gen x, \quad \mathbf{v}_3=x\gen x, \quad \mathbf{v}_4=x^2\gen x+(n-1)xy\gen y, \\
        & \mathbf{v}_{i+4}=y_{i}\gen y, \quad i=1,2,\cdots, n,
    \end{split}
  \end{equation}
where $y_i$'s form a basis for the solution space to this equation.

  The general solution of the equation can be written in the form
\begin{equation}\label{gen-sol-max-sym}
  y=\sum_{i=1}^n c_i u^{n-i}_1u^{i-1}_2,
\end{equation}
   where $u_1, u_2$ are linearly independent solutions to a second order linear, homogeneous ODE. For example, the third order linear ODE \eqref{3rd-lin} has a Lie point symmetry group of dimension seven and admits the general solution
\begin{equation}\label{iterate-sol}
  \xi=c_1 u_1^2+c_2 u_1u_2+c_3 u_2^2,
\end{equation}
where $u_1,u_2$    are solutions to the second order ODE $u''+Iu=0$.
\end{remark}

If we extend \eqref{2nd-homo-ode} to the nonlinear form (a variant of the Ermakov--Pinney equation)
\begin{equation}\label{GEP-0}
  y''+p(x)y'+q(x)y=\tau(x)y^{-3}, \quad \tau(x)\ne 0,
\end{equation}
then in terms of $(r,s)$ in \eqref{r-s} we have
\begin{equation}\label{dsdr2-GEP}
  \frac{d^2s}{dr^2}=e^{2\int p dx}\tau(x) s^{-3}.
\end{equation}
We require that the left hand side of \eqref{dsdr2-GEP} depends only on $s$. This implies
$$\tau(x)=Ke^{-2\int p dx}$$ for some nonzero constant $K$. Using Abel's formula for Wronskians we can write $\tau(x)=K W_0^{-2}W^2(x)$, where $W_0\not=$ is the value of $W(x)$ at some  point $x_0$.
With this condition we conclude that a slightly modified form of the standard Ermakov--Pinney equation
\begin{equation}\label{GEP}
  y''+p(x)y'+q(x)y=KW_0^{-2}W^2(x)y^{-3}
\end{equation}
is equivalent to the standard one $s''=KW_0^{-2}s^{-3}$ (see \eqref{specialEP}) admitting the general solution
$$s^2=A+2B r+Cr^2, \quad AC-B^2=KW_0^{-2}.$$ Hence, a nonlinear superposition principle for the  general solution of \eqref{GEP-0} is reached in the form
\begin{equation}\label{nonlin-superpos}
  y^2=s^2 y_1^2=A y_1^2+2B y_1 y_2+C y_2^2,  \quad (AC-B^2)W_0^2=K
\end{equation}
in terms of the linearly independent solutions of the auxiliary linear equation \eqref{2nd-homo-ode}. We already encountered this remarkable equation  a number of times  (see, for example, \eqref{EP}).
An $\Sl(2,\mathbb{R})$ algebra as  a symmetry algebra of \eqref{GEP}  with a  basis  given by
\begin{equation}\label{sl2-GEP}
\begin{split}
  & \mathbf{v}_1=\gamma(x)(y_1^2\gen x+y_1y'_1y\gen y), \quad \mathbf{v}_2=\gamma(x)(y_1y_2\gen x+\frac{1}{2}(y_1y_2)'y\gen y), \\
  & \mathbf{v}_3=\gamma(x)(y_2^2\gen x+y_2y'_2y\gen y), \quad \gamma(x)=\exp\left[\int p(x)dx\right].
\end{split}
\end{equation}
is admitted.
Their commutation relations are
\begin{equation}\label{sl2-comm-EP}
 [\mathbf{v}_1,\mathbf{v}_2]=W_0\mathbf{v}_1,  \quad  [\mathbf{v}_1,\mathbf{v}_3]=2W_0\mathbf{v}_2, \quad [\mathbf{v}_2,\mathbf{v}_3]=W_0\mathbf{v}_3.
\end{equation}
The nonzero constant $W_0$ can be put to unity by rescaling the elements  of the algebra.
For further information on the Ermakov--Pinney equation \eqref{GEP} in the special case $p=0$, see Refs. \cite{CarinenaGuengoerTorres2020, GuengoerTorres2018}.

\end{example}

\subsection{Lie's Linearization Theorem}\label{test}
This theorem states that amongst second-order
ordinary differential equations, Eq. $y_2=f(x,y,p)$ with $p=y_1$ is point equivalent to the trivial free-particle  equation $y_2=0$ if and only if the following fourth-order Tresse (absolute) invariants
\cite{MilsonValiquette2015}
\begin{equation}\label{abs-inv}
  I_1=f_{pppp}=0,  \quad I_2=\widehat{D}_x^2f_{pp}-4\widehat{D}_xf_{yp}-f_p \widehat{D}_x f_{pp}+6f_{yy}-3f_yf_{pp}+4f_pf_{yp}=0,
\end{equation}
are identically zero.
Here $\widehat{D}_x=\gen x+p\gen y+f\gen p$ is the truncation of the usual total derivative operator $D_x$ restricted to $y_2=f$. The first condition implies that $f$ should be at most cubic in $p=y_1$.
\begin{example}
Any equation admitting a $A_{2,3}$-type symmetry can be linearized if and only if it has the form
\begin{equation}\label{linear-A23}
  xy_2=ay_1^3+by_1^2+\left(1+\frac{b^2}{3a}\right)y_1+\frac{b(9a+b^2)}{27a^2},
\end{equation}
where $a\ne 0$ and $b$ are arbitrary constants. The necessary and sufficient conditions \eqref{abs-inv} for linerizability are satisfied. We can simplify \eqref{linear-A23} using the equivalence group transformation shifting only the parameter $b$ as $b\to b+\varepsilon$ \cite{Sinkala2020}. This means that we can put $b=0$ by choosing  the group parameter $\varepsilon=-b$.  This transformation is given  by
\begin{equation}\label{b=0}
  \tilde{x}=\sqrt{e^{-b}(1+x^2)-1},  \quad \tilde{y}=e^{-b/2}\left(y+\frac{b}{3a}x\right).
\end{equation}
So  \eqref{linear-A23} has effectively  been  reduced to
\begin{equation}\label{linear-A23-2}
  \tilde{x}\tilde{y}_2=a\tilde{y}_1^3+\tilde{y}_1.
\end{equation}
Moreover, $a$ can be transformed to $a=\epsilon=\pm 1$ by scaling $y$ (or $x$).
Two equivalent equations have isomorphic symmetry groups so that there should be a point transformation mapping the two Lie algebras and Eqs. \eqref{linear-A23-2} and $y_2=0$  into each other. See \cite{Sinkala2020} for the construction of such a map. An easier way of achieving linearization is by using a two-dimensional subalgebra of type $A_{22}$ of \ref{Lie-canonical}  of its entire $\Sl(3,\mathbb{R})$ symmetry algebra given by
$$\mathbf{v}_1=\frac{1}{\tilde{x}}\gen {\tilde{x}},  \quad \mathbf{v}_2=\frac{\tilde{y}}{\tilde{x}}\gen {\tilde{x}}.$$ Suitable canonical coordinates are $r=\tilde{y}$ (an invariant), $s=\tilde{x}^2/2$, which maps the subalgebra to $\curl{\gen s,r\gen s}$ and Eq. \eqref{linear-A23-2} to the linear equation $s''(r)=\tilde{y}_1^{-3}(\tilde{y}_1-\tilde{x} \tilde{y}_2)=-a$. Consequently,  the solution in tilde coordinates  is
$$a \tilde{y}^2+\tilde{x}^2+c_1 \tilde{y}+c_2=0.$$ Reverting  to $(x,y)$ coordinates via \eqref{b=0} and redefining arbitrary constants $c_1, c_2$ leads to the solution of  \eqref{linear-A23}
$$ay^2+\frac{2b}{3}xy+\left(1+\frac{b^2}{9a}\right)x^2+C_1\left(y+\frac{b}{3a}x\right)+C_2=0,$$ which geometrically describes a family of ellipses or hyperbolas  as its discriminant $\Delta=-4a$ varies for $a<0$ or $a>0$, respectively.

Equation of Example \ref{ex-2nd-chain} with two-dimensional symmetry group generated by $\curl{\gen x,x\gen x-y\gen y}$ of the type $A_{23}$ of \ref{Lie-canonical} (modulo a point transformation $(x,y)\to (y^{-1},x+y^{-1})$) belongs to the class \eqref{linear-A23} with $a=-1$, $b=6$ or to \eqref{linear-A23-2} with $a=-1$ and hence  once again its linearizability has been established.
\end{example}

In general, if the second order equation $y_2=f(x,y,p)$ is known to satisfy linearizability conditions \eqref{abs-inv}, it is sufficient to pick a two-dimensional intransitive subalgebra of type $A_{2,2}$ or $A_{2,4}$ (up to change of basis) of the $\Sl(3,\mathbb{R})$ symmetry algebra  to transform to a linear equation.

Linearization criteria \eqref{abs-inv} for the ODE \eqref{lineariz-ode-2} is fulfilled if $g(x,y)$ satisfies the PDE
$$y^2 g_{yy}-\sigma y g_y+\sigma g=0,$$ allowing the general solution $g(x,y)=h(x)y^{\sigma}+f(x)y$ for $\sigma\ne 1$. The arbitrary function $f(x)$ can be put to zero by taking a linear combination with $p(x)$.

For $\sigma=1$, we find $g(x,y)=h(x)y \ln y$, $y>0$, excluding the linear term $f(x)y$. Then, the vector field  $\mathbf{v}(\rho)=\rho(x)y\gen y$, where $\rho$ solves the linear ODE $\rho''+p\rho'-h \rho=0$, is a symmetry because $\pr{2}\mathbf{v}(\mathsf{E})=\rho(x)\mathsf{E}=0$ on $\mathsf{E}=y_2+py'+qy-y^{-1}y_1^2-h(x)y\ln y=0$. It  permits us to find the linearizing transformation $y=e^z$. The linearized ODE is
$$z''+p(x)z_1-h(x)z=-q(x).$$

We comment that there are also similar tests for second order systems of ODEs based on the structure of symmetry algebras \cite{SohMahomed2001}. For systems of PDEs, a general approach to linearizability can be found in the paper \cite{KumeiBluman1982} (see also the book  \cite{BlumanCheviakovAnco2010}).

\begin{example}
In this example,  we subject the quadratic Li\'{e}nard's equation \eqref{q-lienard}
to the linearization test \eqref{test}.
The first condition $I_1=0$ is identically satisfied. The second invariant condition gives a differential relation between $f$ and $g$ such that $g''+(fg)'=0$, or
\begin{equation}\label{linearity-cond-f-g}
  g'+fg=A,
\end{equation}
where $A$ is the integration constant. This condition can be analysed in two different ways.

1) Given $f(y)$, we must have
$$ g(y)=e^{-\int f(y)dy}[A \int e^{\int f(y)dy}dy+B].$$
By construction, the symmetry algebra of the corresponding ODE is eight-dimensional. It is a straightforward calculation to see that the two-dimensional abelian subalgebra $\lie_2$ spanned by
\begin{equation}\label{abelian-sym-alg}
  \mathbf{v}_{1}(y_1)=\Omega(y)y_{1}\gen y, \quad \mathbf{v}_{2}(y_2)=\Omega(y)y_{2}\gen y,  \quad \Omega'+f\Omega=0,
\end{equation}
where $y_{1},y_2$ form a fundamental set of solutions of the linear equation
\begin{equation}\label{lin-ode-1}
  z''+A z=-B
\end{equation}
is admitted by this ODE.
The fact that the  Lie algebra $\lie_2$ depends on the solutions of a linear ODE suggests that there must be a change of dependent variable $Y=Y(y)$  such that $\Omega Y'=1$ mapping the initial ODE to  the linear ODE
$$Y''+AY=-B,$$
where $Y$ satisfies $Y''-fY'=0$ (compare this mapping with \eqref{change-depend}). As a special case in which $f(y)=1$, $g(y)=A+Be^{-y}$, the following equation is singled out  of the class \eqref{q-lienard}
\begin{equation}\label{special-q-lienard}
  y''+y'^2+A+B e^{-y}=0.
\end{equation}
It is linearized to \eqref{lin-ode-1} by the point transformation $Y=e^{y}$.

2.) One can view the condition \eqref{linearity-cond-f-g} as a finite equation for $f(y)=(A-g')/g$ when $g(y)\ne 0$ is given. In this case, the above symmetry \eqref{abelian-sym-alg} still holds, but \eqref{lin-ode-1} is replaced by
\begin{equation}\label{lin-ode-2}
  z''+Az=0
\end{equation}
and $\Omega, Y$ satisfy
$$\Omega(y)=g\exp\left[-A\int g^{-1}dy\right],  \quad Z'=\Omega^{-1}.$$

\end{example}

\subsection{Concluding remarks}
So far we restricted our attention to local Lie point transformations that leave the system of differential equations invariant. They are usually called classical symmetries of the system.  There are several uses of them in the context of differential equations like  group-invariant solutions for PDEs and integration for ODEs among others. One the important generalizations of classical symmetries is the notion of generalized (or Lie--B\"acklund)  symmetries. They were originally introduced by E. Noether in her celebrated  paper \cite{Noether1918} (see \cite{NoetherTavel2005} for a more recent English translation), which related variational symmetries and (local) conservation laws. They also play a central role in soliton theory of completely integrable nonlinear partial differential equations.  Evolution equations like KdV (Korteweg--de Vries) and KP (Kadomtsev--Petviashvili) are known to admit infinite hierarchies of generalized symmetries and conservation laws.   An approach to classification of integrable systems is to identify systems admitting generalized symmetries up to differential substitutions  preserving the differential structure of the system. The infinitesimal generators of generalized symmetries  depend on derivatives of dependent variables. In the special case of one dependent variable, every first order generalized symmetry  determines a contact symmetry (local diffeomorphisms of the $n$-th order jet space $J^n$, preserving the contact ideal generated by  contact forms of the equation) and vice versa.

Over the last few decades, numerous useful extensions of the classical Lie approach to group-invariant solutions have been developed. Among others, they include  the method of partially invariant solutions \cite{Ovsyannikov1982}, non-classical method, methods of  conditional (point or generalized) \cite{OlverVorob'ev1996},  and of  symmetries of potential  systems \cite{BlumanCheviakovAnco2010}.

\section*{Acknowledgement}
I would like to thank Prof. Willy Hereman  for  testing the symmetry calculations of some  Examples with his Maxima program symmgrp2020 and spotting various misprints.


\end{document}